\documentclass[a4paper,10pt,final]{article}
\usepackage[bottom]{footmisc}

\usepackage{epic,eepic}
\usepackage{graphicx,color}
\usepackage[color,notref,notcite]{showkeys}
\usepackage{graphicx}
\usepackage{url}
\usepackage{amssymb,amsmath}
\usepackage{colortbl}
\usepackage{multicol}
\usepackage{multirow}
\usepackage{bm}
\usepackage{subfigure}
\usepackage{framed}
\usepackage{multirow}

\newenvironment{proof}{\noindent {\bf Proof } }{$\square$ }

\usepackage[normalem]{ulem}
\definecolor{violet}{rgb}{0.580,0.,0.827}

\allowdisplaybreaks

\definecolor{labelkey}{rgb}{0.6,0,1}

\newcounter{cst}
\catcode`\@=11
\def\ctel#1{C_{\refstepcounter{cst}\@bsphack
\protected@write\@auxout{}%
           {\string\newlabel{#1}{{\thecst}{\thepage}}}\thecst}}
\catcode`\@=12

\newcommand{\cter}[1]{C_{\ref{#1}}}

\newcounter{cexp}
\catcode`\@=11
\def\terml#1{T_{\refstepcounter{cexp}\@bsphack
\protected@write\@auxout{}%
           {\string\newlabel{#1}{{\thecexp}{\thepage}}}\thecexp}}
\catcode`\@=12

\newcommand{\mathbi}[1]{{\boldsymbol #1}}

\newcommand{\eop}{{\unskip\nobreak\hfil\penalty50
           \hskip2em\hbox{}\nobreak\hfil\mbox{\rule{1ex}{1ex} \qquad}
   \parfillskip=0pt
   \finalhyphendemerits=0\par\medskip}}
\renewenvironment{proof}[1][]{\noindent {\bf Proof#1. } }{\eop}

\definecolor{shadecolor}{gray}{0.92}

\newcommand{\ba}{\begin{array}{llll}   }
\newcommand{\bac}{\begin{array}{c}}
\newcommand{\bari}{\begin{array}{r}}
\newcommand{\ea}{\end{array}}

\newcommand{\ban}{\begin{array}{llll}}
\newcommand{\ean}{\end{array}}

\newcommand{\be}{\begin{equation}}
\newcommand{\ee}{\end{equation}}

\newcommand{\beqsys }{\beqtab \left \{ \begin{array}{l}}
\newcommand{\eeqsys }{\end{array} \right . \eeqtab }

\newcommand{\benum}{\begin{enumerate}}
\newcommand{\eenum}{\end{enumerate}}

\newcommand{\beqtab}{\begin{eqnarray}} 
\newcommand{\eeqtab}{\end{eqnarray}}

\newcommand{\dsp}{\displaystyle}

\newcommand{\bchi}{\mathbi{\chi}}
\newcommand{\bfa}{\mathbi{a}}
\newcommand{\bfA}{\mathbi{A}}

\newcommand{\bxi}{\mathbi{\xi}}

\newcommand{\bG}{\mathbi{G}}

\newcommand{\bu}{\ubarre}
\newcommand{\bvarphi}{\mathbi{\varphi}}

\renewcommand{\d}{{\rm d}}

\newcommand{\disc}{{\mathcal D}}
\newcommand{\dr}{\partial}
\newcommand{\dt}{{\delta\!t}}

\renewcommand{\div}{{\rm div}}

\newcommand{\grad}{\nabla}

\newcommand{\half}{{\frac 1 2}}

\newcommand{\interp}{{\mathcal I}}

\newcommand{\N}{\mathbb N}

\newcommand{\norm}[2]{\| #1 \|_{#2}}

\renewcommand{\O}{\Omega}

\newcommand{\Pidisc}{\Pi_\disc}
\renewcommand{\phi}{\varphi}

\newcommand{\R}{\mathbb R}

\newcommand{\ubarre}{{\overline u}}

\newcommand{\vbarre}{{\overline v}}

\newcommand{\x}{\mathbi{x}}

\newcommand{\y}{\mathbi{y}}

\def\argmin{\mathop{\,\rm argmin\;}}

\def\weak{\mbox{\rm-w}}

\newcommand{\z}{\mathbi{z}}
\newcommand{\Z}{\mathbb Z}

\newtheorem{theorem}{Theorem}[section] 
\newtheorem{remark}[theorem]{Remark}

\newtheorem{lemma}[theorem]{Lemma}
\newtheorem{definition}[theorem]{Definition}
\newtheorem{proposition}[theorem]{Proposition}
\newtheorem{corollary}[theorem]{Corollary}

\title{Uniform-in-time convergence of numerical methods for
non-linear degenerate parabolic equations}
\author{J\'er\^ome Droniou\footnote{School of Mathematical Sciences,
Monash University, Victoria 3800, Australia. \texttt{jerome.droniou@monash.edu}.}~
and Robert Eymard\footnote{Universit\'e Paris-Est, Laboratoire d'Analyse et de Math\'ematiques Appliqu\'ees,
UMR 8050, 5 boulevard Descartes, Champs-sur-Marne
77454 Marne-la-Vall\'ee Cedex 2, France. \texttt{Robert.Eymard@univ-mlv.fr}.}}

\begin{document}

\maketitle

\begin{abstract}
Gradient schemes is a framework that enables the unified convergence analysis
of many numerical methods for elliptic and parabolic partial differential
equations: conforming and non-conforming Finite Element, Mixed Finite Element and Finite Volume methods.
We show here that this framework can be applied to a family of degenerate non-linear parabolic
equations (which contain in particular the Richards', Stefan's and Leray--Lions' models),
and we prove a uniform-in-time strong-in-space convergence result for the
gradient scheme approximations of these equations.
In order to establish this convergence, we develop several discrete compactness
tools for numerical approximations of parabolic models, including
a discontinuous Ascoli-Arzel\`a theorem and a uniform-in-time weak-in-space discrete
Aubin-Simon theorem. The model's degeneracies, which
occur both in the time and space derivatives, also requires us to develop a discrete compensated
compactness result.
\end{abstract}
\medskip

\medskip


\section{Introduction}

\subsection{Motivation}

The following generic nonlinear parabolic model
\be\ba
\partial_t \beta(\ubarre)-\div \left( \bfa(\x,\nu(\ubarre),\grad\zeta(\ubarre)) \right)= f &\hbox{ in } \Omega\times(0,T),\\
\beta(\ubarre)(\x,0) = \beta(u_{\rm ini})(\x)&\hbox{ in } \Omega,\\
\zeta(\ubarre)=0 &\hbox{ on } \dr \Omega\times(0,T),
\ea\label{pbintrot}\ee
where $\beta$ and $\zeta$ are non-decreasing, $\nu$ is such that $\nu' = \beta'\zeta'$ and  $\bfa$ is a Leray--Lions operator, arises in various frameworks (see next section for precise hypotheses on the data). This model includes
\begin{enumerate}
\item The Richards model, setting $\zeta(s)=s$, $\nu = \beta$ and $\bfa(\x,\nu(\ubarre),\grad\zeta(\ubarre)) = K(\x,\beta(\ubarre)) \grad \ubarre$, which describes the
flow of water in a heterogeneous anisotropic underground medium, 
\item The Stefan model \cite{pelletier}, setting $\beta(s)=s$, $\nu = \zeta$, $\bfa(\x,\nu(\ubarre),\grad\zeta(\ubarre)) = K(\x,\zeta(\ubarre)) \grad \zeta(\ubarre)$,
which arises in the study of a simplified heat diffusion process in a melting medium,
\item The $p-$Laplace problem, setting $\beta(s) = \zeta(s) = \nu(s) = s$ and $\bfa(\x,\nu(\ubarre),\grad\zeta(\ubarre)) =  |\grad \ubarre|^{p-2} \grad \ubarre$,
which is involved in the motion of glaciers \cite{glo03} or flows of incompressible turbulent
fluids through porous media \cite{dia94}.

General Leray--Lions operators $\bfa(\x,s,\bxi)$ have growth, monotony and coercivity
properties (see \eqref{hypanl2}--\eqref{hypanl4} below)
which ensure that $-\div(\bfa(\x,w,\nabla\cdot))$ maps $W^{1,p}_0(\O)$ into $W^{-1,p'}(\O)$,
and thanks to which this differential operator is viewed as a generalisation of the $p$-Laplace operator.
\end{enumerate}

The numerical approximation of these models has
been extensively studied in the literature -- see the fundamental work on the Stefan's problem \cite{MR954786} and \cite{MR2230161,EGHNS} for some of its
numerical approximations, \cite{MR1916296,MR1750075} for the Richards' problem, and \cite{dro-06-ll,dro-12-gra} and references therein for some studies of convergence of numerical methods for the Leray--Lions' problem.
In \cite{rul1996opt}, fully discrete implicit schemes are considered in 2D domains
for the problem $\partial_t e-\Delta u=f$, $e\in \beta(u)$ with $\beta$ a maximal monotone operator; error estimates are obtained and the results are relevant, e.g., for the Stefan problem and the porous medium equation.

More generally, studies have been carried out on numerical time-stepping approximations of non-linear abstract parabolic equations.
In \cite{lub1995lin} the authors study the stability and convergence properties of linearised implicit methods for the time discretization
of nonlinear parabolic equations in the general framework of  Hilbert spaces. The time discretisation of nonlinear evolution equations in an abstract Banach space
setting of analytic semigroups is studied in \cite{gon2002bac}; this setting covers fully nonlinear parabolic initial-boundary value problems with smooth coefficients.
\cite{akr2011gal} deals with a general formulation for semi-discretisations of linear parabolic evolution problems in Hilbert spaces; this
time-stepping formulation encompasses continuous and discontinuous Galerkin methods, as well as Runge Kutta methods. The study in \cite{akr2011gal} has been
extended in \cite{akr2009opt} to semi-linear equations, i.e. with the addition of a right-hand side which is locally Lipschitz-continuous with respect to the unknown.
In the same directions, we also quote \cite{lub1993runge,lub1995run,lub1996run,ost2002con,gwi2014ful} for Runge-Kutta time discretizations
of linear and quasilinear parabolic equations (reaction-diffusion, Navier-Stokes equations, etc.). Multisteps methods have also been
considered, see e.g. \cite{ost2004sta}. 

However, most of these studies are only applicable under regularity assumptions on the
solution or data, and to semi-linear equations or semi-discretised schemes. None deals with
as many non-linearities and degeneracies as in \eqref{pbintrot}. Moreover, 
the results in these works mostly yield space-time averaged convergences,
e.g. in $L^2(\O\times (0,T))$. Yet, the quantity of
interest is often not $\bu$ on $\O\times (0,T)$ but $\bu$ at a given time, for example $t=T$.
Current numerical analyses therefore do not ensure that this quantity of interest is properly approximated by
numerical methods.

The usual way to obtain pointwise-in-time approximation results for numerical
schemes is to prove estimates in $L^\infty(0,T;L^2(\O))$ on
$u-\bu$, where $u$ is the approximated solution. Establishing such error
estimates is however only feasible when uniqueness of the solution $\bu$
to \eqref{pbintrot} can be proved, which is the case for Richards' and
Stefan's problems (with $K$ only depending on $\x$), but not for more complex non-linear parabolic problems
as \eqref{pbintrot} or even $p$-Laplace problems.
It moreover requires some regularity assumptions on $\bu$,
which clearly fail to hold for \eqref{pbintrot} (and simpler $p$-Laplace problems);
indeed, because of the possible
plateaux of $\beta$ and $\zeta$, the solution's gradient can develop jumps.

The purpose of this article is to prove that, using Discrete Functional Analysis
techniques (i.e. the translation to numerical analysis of nonlinear analysis
techniques), an $L^\infty(0,T;L^2(\O))$ convergence
result can be established for numerical approximations of \eqref{pbintrot}, without having to assume
non-physical regularity assumptions on the data. Note that,
although Richards' and Stefan's models are formally equivalent when
$\beta$ and $\zeta$ are strictly increasing (consider $\beta=\zeta^{-1}$
to pass from one model to the other), they change nature when these
functions are allowed to have plateaux. Stefan's model can degenerate to an ODE (if $\zeta$ is constant
on the range of the solution) and Richards' model can become a non-transient elliptic
equation (if $\beta$ is constant on this range). The innovative technique we develop in this paper
is nonetheless generic enough to work directly on \eqref{pbintrot} and
with a vast number of numerical methods.

That being said, a particular numerical framework must be selected
to write precise equations and estimates. The framework we choose
is that of gradient schemes, which has the double benefit of covering
a vast number of numerical methods, and of having already been studied
for many models -- elliptic, parabolic, linear or non-linear, possibly degenerate,
etc. -- with various boundary conditions.
The schemes or family of schemes included in the gradient schemes framework, and to which our results
therefore directly apply, currently are:
\begin{itemize}
\item Galerkin methods, including conforming Finite Element
schemes,
\item finite element with mass lumping \cite{cia-91-fin},
\item the Crouzeix-Raviart non-conforming finite element, with or without mass lumping \cite{crouzeix-raviart-73,ern2004theory},
\item the Raviart-Thomas mixed finite elements \cite{brezzi-fortin},
\item the vertex approximate gradient scheme \cite{eym-12-sma},
\item the hybrid mimetic mixed family \cite{dro-10-uni}, which includes mimetic
finite differences \cite{bre-05-fam}, mixed finite volume \cite{dro-06-mix} and the SUSHI scheme \cite{sushi},
\item the discrete duality finite volume scheme in dimension 2 \cite{her-03-app,and-07-dis}, and the
CeVeFE-discrete duality finite volume scheme in dimension 3 \cite{cou-10-dis},
\item the multi-point flux approximation O-method \cite{aav-96-dis,edw-98-mpfa}.
\end{itemize}
We refer the reader to \cite{koala,dro-12-gra,eym-12-stef,eym-11-gra,zamm2013} for more details.
Let us finally emphasize that the unified convergence study of numerical schemes for
Problem \eqref{pbintrot},
which combines a general Leray--Lions operator and nonlinear functions $\beta$ or $\zeta$, seems to be new
even without the uniform-in-time convergence result.

\medskip

The paper is organised as follows. In Section \ref{sec:contprob}, we present
the assumptions and the notion of weak solution for \eqref{pbintrot} and, in Section \ref{sec:genpcp},
we give an overview of the ideas involved in the proof of uniform-in-time convergence.
This overview is given not in a numerical analysis context but
in the context of a pure stability analysis of \eqref{pbintrot} with very little regularity on the data, for which
the uniform-in-time convergence result also seems to be new.
Section \ref{sec:GS} presents the gradient schemes for our generic model \eqref{pbintrot}. We give in Section
\ref{sec:prelim} some preliminaries to the convergence study, in particular a crucial
uniform-in-time weak-in-space discrete Aubin-Simon compactness result. Section \ref{sec:proof}
contains the complete convergence proof of gradient schemes for \eqref{pbintrot}, including
the uniform-in-time convergence result. This proof is initially conducted under a simplifying assumption
on $\beta$ and $\zeta$. We demonstrate in Section \ref{sec:hypstruct} that, in the case $p\ge 2$,
this assumption can be removed thanks to a discrete compensated compactness result. 
We also remark in this section that our results apply to the model considered in \cite{rul1996opt}.
An appendix, Section \ref{sec:comptime}, concludes the article
with technical results, in particular a generalisation of the
Ascoli-Arzel\`a compactness result to discontinuous functions
and a characterisation of the uniform convergence of a sequence of functions;
these results are critical to establishing our uniform-in-time convergence result.
We believe that the Discrete Functional Analysis results we establish in
order to study the approximations of \eqref{pbintrot} -- in particular the
discrete compensated compactness theorem (Theorem \ref{thm-comp-comp}) -- could
be critical to the numerical analysis of other degenerate or coupled models of physical importance.

Note that the main results and their proofs have been sketched and illustrated by some numerical examples
in \cite{DE-fvca7}, for $\bfa(\x,\nu(\ubarre),\grad\zeta(\ubarre)) = \grad\zeta(\ubarre)$
and $\beta={\rm Id}$ or $\zeta={\rm Id}$.

\subsection{Hypotheses and weak sense for the continuous problem}\label{sec:contprob}

We consider the evolution problem \eqref{pbintrot} under the following hypotheses.
\begin{subequations}
\begin{align}
&\begin{array}{l}\O \mbox{ is an open bounded subset of $\R^d$
($d\in\N^\star$) and }T>0,\end{array} \label{hypomega} \\
&\begin{array}{l}
\zeta\in C^0(\R)\mbox{ is non--decreasing, Lipschitz continuous
 with Lipschitz constant } L_\zeta>0, \\
\zeta(0) = 0 \mbox{ and, for some $M_0,M_1>0$, }
|\zeta(s)| \ge M_0 |s| - M_1\mbox{  for all } s\in\R.
\end{array}
\label{hypzeta}\\
&\begin{array}{l}
\beta\in C^0(\R)\hbox{ is non--decreasing, Lipschitz continuous
with Lipschitz constant $L_\beta> 0$},\\
\hbox{and }\beta(0) = 0.\end{array}
\label{hypbeta}\\
&\label{def:chi}
 \begin{array}{l}
\dsp \forall s\in\R\,,\quad \nu(s) = \int_0^s \zeta'(q)\beta'(q) \d q.
\end{array}\\
& \begin{array}{l}
\bfa~:~\O\times \R\times\R^d\to\R^d\hbox{ is a Carath\'eodory function}
\end{array} \label{hypanl1}
\intertext{(i.e. a function such that, for a.e. $\x \in \Omega$, $(s,\bxi) \mapsto \bfa(\x, s,\bxi)$ is continuous 
and, for any $(s,\bxi) \in \R\times\R^d$, $\x \mapsto \bfa(\x,s,\bxi )$ is measurable)
and, for some $p\in (1,+\infty)$,}
&\begin{array}{l}
\exists \underline a\in(0,+\infty)  \ : \ \bfa(\x,s,\bxi)\cdot\bxi \ge \underline a|\bxi|^p,\hbox{ for a.e. }\x\in\O,\ \forall s\in \R,\ \forall\bxi\in\R^d,
\end{array}
\label{hypanl2}
\\
&\begin{array}{l}
(\bfa(\x,s,\bxi) - \bfa(\x,s,\bchi))\cdot(\bxi-\bchi)\ge 0,\hbox{ for a.e. }\x\in\O,\ \forall s\in \R,\ \forall\bxi,\bchi\in\R^d,
\end{array}
\label{hypanl3}
\\
&\begin{array}{l} \exists \overline{a}\in L^{p'}(\O)\,,\,\exists\mu\in (0,+\infty)\ :\\
\quad |\bfa(\x,s,\bxi)|\le \overline{a}(\x) + \mu|\bxi|^{p-1}, \hbox{ for a.e. }\x\in\O,\ \forall s\in \R,\ \forall\bxi\in\R^d.
\end{array}
\label{hypanl4}
\intertext{We also assume, setting $p'=\frac{p}{p-1}$ the dual
exponent of the $p$ previously introduced,}
&\begin{array}{l}
u_{\rm ini}\in L^2(\O),\quad
f \in L^{p'}(\O\times(0,T)). 
\end{array}\label{hypfgnlt}
\end{align}
\label{hypgnl}
\end{subequations}
We denote by $R_\beta$ the range of $\beta$ and define the pseudo-inverse
function $\beta_r:R_\beta\to\R$ of $\beta$ by
\be
\ba
\forall s\in R_\beta\,,\;
\dsp\beta_r(s)&=&\dsp\left\{\begin{array}{ll} \inf\{t\in\R\,|\,\beta(t)=s\}&\mbox{ if $s> 0$},\\
0&\mbox{ if $s=0$},\\
\sup\{t\in\R\,|\,\beta(t)=s\}&\mbox{ if $s< 0$},\end{array}\right.\\
&=&\dsp\mbox{closest $t$ to $0$ such that $\beta(t)=s$.}
\ea
\label{defbetar}
\ee
Since $\beta(t)$ has the same sign as $t$, we have $\beta_r\ge 0$ on $R_\beta\cap \R^+$ and $\beta_r\le 0$ on
$R_\beta\cap \R^-$.
We then define $B:R_\beta\to [0,\infty]$ by
\[
B(z)=\int_0^z \zeta(\beta_r(s))\,ds.
\]
Since $\beta_r$ is non-decreasing, this expression is always well-defined in $[0,\infty)$.
The signs of $\beta_r$ and $\zeta$ ensure that $B$ is non-decreasing on $R_\beta\cap \R^+$
and non-increasing on $R_\beta\cap \R^-$, and therefore has
limits (possibly $+\infty$) at the endpoints of $R_\beta$. 
We can thus extend $B$ as a function defined on $\overline{R_\beta}$
with values in $[0,+\infty]$.

The precise notion of solution to \eqref{pbintrot} that we consider is the following:
\be\left\{\ba
\bu\in L^p(0,T;L^p(\O))\,,\; \zeta(\ubarre) \in L^p(0,T;W^{1,p}_0(\O))\,,\\
B(\beta(\ubarre))\in  L^\infty(0,T;L^1(\O)),\ \beta(\ubarre)\in C([0,T];L^2(\O)\weak),\partial_t\beta(\ubarre)\in L^{p'}(0,T;W^{-1,p'}(\O)),\\
\beta(\ubarre)(\cdot,0) = \beta(u_{\rm ini})\mbox{ in $L^2(\O)$},\\
\dsp\int_0^T \langle \partial_t\beta(\ubarre)(\cdot,t), \vbarre(\cdot,t)\rangle_{W^{-1,p'},W^{1,p}_0}\d t \\
\dsp \quad +  
\int_0^T \int_\O \bfa(\x,\nu(\ubarre(\x,t)),\grad\zeta(\ubarre)(\x,t))\cdot\grad \vbarre(\x,t) 
\d\x\d t = \int_0^T \int_\O f(\x,t) \vbarre(\x,t) \d\x\d t\,,\\
\qquad \forall \vbarre\in L^p(0;T;W^{1,p}_0(\O)).
\ea\right.\label{ellgenfllt}
\ee
where $C([0,T];L^2(\O)\weak)$ denotes the space of continuous functions $[0,T]\mapsto L^2(\O)$ for
the weak-$*$ topology of $L^2(\O)$. Here and in the following, we
remove the mention of $\O$ in the duality bracket $\langle \cdot, \cdot\rangle_{W^{-1,p'},W^{1,p}_0}
=\langle \cdot, \cdot\rangle_{W^{-1,p'}(\O),W^{1,p}_0(\O)}$.

\begin{remark}\label{rem-defdt_u}
The derivative $\partial_t\beta(\ubarre)$ is to be understood in the usual sense of
distributions on $\O\times (0,T)$. Since the set
$\mathcal T=\{\sum_{i=1}^q\varphi_i(t)\gamma_i(\x)\,:\,
q\in\N,\varphi_i\in C^\infty_c(0,T),\gamma_i\in C^\infty_c(\O)\}$ of tensorial
functions in $C^\infty_c(\O\times(0,T))$ is dense in $L^p(0,T;W^{1,p}_0(\O))$,
one can ensure that this distribution derivative $\partial_t \beta(\ubarre)$ belongs to 
$L^{p'}(0,T;W^{-1,p'}(\O))=(L^p(0,T;W^{1,p}_0(\O)))'$ by checking that the linear form
\[
\varphi\in \mathcal T\mapsto \langle \partial_t \beta(\ubarre),\varphi\rangle_{\mathcal D',\mathcal D}
=-\int_0^T\int_\O \beta(\ubarre)(\x,t)\partial_t\varphi(\x,t)\d\x \d t
\]
is continuous for the norm of $L^p(0,T;W^{1,p}_0(\O))$.
\end{remark}

Note that the continuity property of $\beta(\bu)$ in \eqref{ellgenfllt} is natural. Indeed,
since $\beta(\bu)\in L^\infty(0,T;L^2(\O))$ (this comes from $B(\beta(\bu))\in L^\infty(0,T;L^1(\O))$ and
\eqref{growthB}), the PDE in the sense of distributions shows
that for any $\varphi\in C^\infty_c(\O)$
the mapping $T_\varphi:t\mapsto \langle \beta(\bu)(t),\varphi\rangle_{L^2}$
belongs to $W^{1,1}(0,T)\subset C([0,T])$.
By density of $C^\infty_c(\O)$ in $L^2(\O)$ and the integrability properties of $\beta(\bu)$,
we deduce that $T_\varphi\in C([0,T])$ for any $\varphi\in L^2(\O)$, which precisely 
establishes the continuity of $\beta(\bu):[0,T]\to L^2(\O)\weak$.

This notion of $\beta(\bu)$ as a function continuous in time is nevertheless
a subtle one. It is to be understood in the sense that the function $(\x,t)\mapsto \beta(\bu(\x,t))$
has an a.e. representative which is continuous $[0,T]\mapsto L^2(\O)\weak$. In other words,
there is a function $Z\in C([0,T];L^2(\O)\weak)$ such that $Z(t)(\x)=\beta(\bu(\x,t))$ for a.e.
$(\x,t)\in\O\times (0,T)$. We must however make sure, when dealing with pointwise values in
time, to separate $Z$ from $\beta(\bu(\cdot,\cdot))$ as $\beta(\bu(\cdot,t_1))$ may not make
sense for a particular $t_1\in [0,T]$.
That being said, in order to adopt a simple notation, we will denote by $\beta(\bu)(\cdot,\cdot)$
the function $Z$, and by $\beta(\bu(\cdot,\cdot))$ the a.e.-defined composition of $\beta$ and $\bu$.
Hence, it will make sense to talk about $\beta(\bu)(\cdot,t)$ for a particular $t_1\in [0,T]$,
and we will only write $\beta(\bu)(\x,t)=\beta(\bu(\x,t))$ for a.e. $(\x,t)\in\O\times (0,T)$.
Note that from this a.e. equality we can ensure that $\beta(\bu)(\cdot,\cdot)$ takes
its values in the closure $\overline{R_\beta}$ of the range of $\beta$.

\subsection{General ideas for the uniform-in-time convergence result}\label{sec:genpcp}

As explained in the introduction, the main innovative result of this article
is the uniform-in-time convergence result (Theorem \ref{th:uniftime} below).
Although it's stated and proved in the context of numerical approximations
of \eqref{pbintrot}, we emphasize that the ideas underlying its proof are also applicable to theoretical
analysis of PDEs. Let us informally present these ideas on the following
continuous approximation of \eqref{pbintrot}:
\be\ba
\partial_t \beta(\ubarre_\varepsilon)-\div \left( \bfa_\varepsilon(\x,\nu(\ubarre_\varepsilon),\grad\zeta(\ubarre_\varepsilon)) \right)= f &\hbox{ in } \Omega\times(0,T),\\
\beta(\ubarre_\varepsilon)(\x,0) = \beta(u_{\rm ini})(\x)&\hbox{ in } \Omega,\\
\zeta(\ubarre_\varepsilon)=0 &\hbox{ on } \dr \Omega\times(0,T)
\ea\label{pbintroteps}\ee
where $\bfa_\varepsilon$ satisfies Assumptions \eqref{hypanl1}--\eqref{hypanl4} with constants
not depending on $\varepsilon$ and, as $\varepsilon\to 0$, $\bfa_\varepsilon\to \bfa$ locally
uniformly with respect to $(s,\bxi)$.

We want to show here how to deduce from averaged convergences
a strong uniform-in-time convergence result. We therefore assume the following convergences
(up to a subsequence as $\varepsilon\to 0$),
which are compatible with basic compactness results that can be obtained on
$(\ubarre_\varepsilon)_{\varepsilon}$
and also correspond to the initial convergences \eqref{conv:base}
on numerical approximations of \eqref{pbintrot}:
\be\label{aver-ueps}
\ba
\beta(\ubarre_\varepsilon)\to \beta(\ubarre)\mbox{ in $C([0,T];L^2(\O)\weak)$}\,,\;
\nu(\ubarre_\varepsilon)\to \nu(\ubarre)\mbox{ strongly in $L^1(\O\times(0,T))$},\\
\zeta(\ubarre_\varepsilon)\to \zeta(\ubarre)
\mbox{ weakly in $L^p(0,T;W^{1,p}_0(\O))$}\,,\\
\bfa_\varepsilon(\cdot,\nu(\ubarre_\varepsilon),\nabla\zeta(\ubarre_\varepsilon))
\to \bfa(\cdot,\nu(\ubarre),\nabla\zeta(\ubarre))
\mbox{ weakly in $L^{p'}(\O\times(0,T))^d$}.
\ea
\ee
We will prove from these convergences that, along the same subsequence, $\nu(\ubarre_\varepsilon)\to
\nu(\ubarre)$ strongly in $C([0,T];L^2(\O))$, which is our uniform-in-time convergence
result.

We start by noticing that the weak-in-space uniform-in-time convergence
of $\beta(\ubarre_\varepsilon)$ gives, for any $T_0\in [0,T]$ and any family $(T_\varepsilon)_{\varepsilon>0}$
converging to $T_0$ as $\varepsilon\to 0$, $\beta(\ubarre_\varepsilon)(T_\varepsilon,\cdot)
\to \beta(\ubarre)(T_0,\cdot)$ weakly in $L^2(\O)$. Classical strong-weak semi-continuity properties of
convex functions (see Lemma \ref{lem:Zconv}) and the convexity of $B$
(see Lemma \ref{lembetaB}) then ensure that
\be\label{conv-theorique1}
\int_\O B(\beta(\ubarre)(\x,T_0))\d\x\le \liminf_{\varepsilon\to 0}
\int_\O B(\beta(\ubarre_\varepsilon)(\x,T_\varepsilon))\d\x.
\ee

The second step is to notice that, by \eqref{hypanl3} for $\bfa_\varepsilon$,
\[
\int_0^{T_\varepsilon} \int_\O \left[\bfa_\varepsilon(\cdot,\nu(\ubarre_\varepsilon),
\nabla\zeta(\ubarre_\varepsilon)) - \bfa_\varepsilon(\cdot,\nu(\ubarre_\varepsilon),
\nabla\zeta(\ubarre))\right]
\cdot\left[\grad\zeta(\ubarre_\varepsilon)-\nabla\zeta(\ubarre)\right]\d\x\d t \ge 0.
\]
Developing this expression and using the convergences \eqref{aver-ueps}, we find that
\be\label{conv-theorique2}
\liminf_{\varepsilon\to 0}\int_0^{T_\varepsilon} \int_\O \bfa_\varepsilon(\cdot,\nu(\ubarre_\varepsilon),
\nabla\zeta(\ubarre_\varepsilon))\cdot\nabla\zeta(\ubarre_\varepsilon)(\x,t)\d\x\d t
\ge \int_0^{T_0}\int_\O \bfa(\cdot,\nu(\ubarre),\nabla\zeta(\ubarre)) \cdot \nabla\zeta(\ubarre)
\d\x\d t.
\ee

We then establish the following formula:
\begin{multline}
\int_\O B(\beta(\bu_\varepsilon(\x,T_\varepsilon))) \d \x 
+ \int_{0}^{T_\varepsilon}   \int_\O \bfa_\varepsilon(\x,\nu(\bu_\varepsilon(\x,t)),\nabla\zeta(\bu_\varepsilon)(\x,t))
\cdot\nabla \zeta(\bu_\varepsilon)(\x,t)\d\x \d t \\
=  \int_\O B(\beta(u_{\rm ini}(\x))) \d \x 
+  \int_{0}^{T_\varepsilon}   \int_\O f(\x,t) \zeta(\bu_\varepsilon)(\x,t) \d \x \d t.\label{form:ippsoleps}
\end{multline}
This energy equation is formally obtained by multiplying \eqref{pbintroteps}
by $\zeta(\bu_\varepsilon)$ and integrating by parts, using
$(B\circ\beta)'=\zeta\beta'$ (see Lemma \ref{lembetaB}); the rigorous justification
of \eqref{form:ippsoleps} is however quite technical -- see Lemma \ref{ricipp} and
Corollary \ref{ric:bou}.
Thanks to \eqref{conv-theorique2}, we can pass to the $\limsup$ in \eqref{form:ippsoleps}
and we find, using the same energy equality with $(\ubarre,\bfa,T_0)$ instead of
$(\ubarre_\varepsilon,\bfa_\varepsilon,T_\varepsilon)$,
\be\label{conv-theorique3}
\limsup_{\varepsilon\to 0} \int_\O B(\beta(\ubarre_\varepsilon(\x,T_\varepsilon)))\d\x
\le \int_\O B(\beta(\ubarre(\x,T_0)))\d\x.
\ee
Combined with \eqref{conv-theorique1}, this shows that $\int_\O
B(\beta(\ubarre_\varepsilon(\x,T_\varepsilon)))\d\x\to \int_\O B(\beta(\ubarre(\x,T_0)))\d\x$.
A uniform convexity property of $B$ (see \eqref{Bunifconv}) then allows us to deduce
that $\nu(\ubarre_\varepsilon(\cdot,T_\varepsilon))\to
\nu(\ubarre(\cdot,T_0))$ strongly in $L^2(\O)$ and thus
that $\nu(\ubarre_\varepsilon)\to \nu(\ubarre)$
strongly in $C([0,T];L^2(\O))$ (see Lemma \ref{equiv-unifconv}).

\begin{remark} A close examination of this proof indicates that equality in the energy relation
\eqref{form:ippsoleps} is not required for $\ubarre_\varepsilon$. An inequality
$\le$ would be sufficient. This is particularly important in the context of
numerical methods which may introduce additional numerical diffusion (for example
due to an implicit-in-time discretisation) and
therefore only provide an upper bound in this energy estimate, see \eqref{app:ipp}.
It is however essential that the limit solution $\ubarre$ satisfies the equivalent of \eqref{form:ippsoleps}
with an equal sign (or $\ge$).
\end{remark}

\section{Gradient discretisations and gradient schemes}\label{sec:GS}

\subsection{Definitions}

We give here a minimal presentation of gradient discretisations and gradient schemes,
limiting ourselves to what is necessary to study the discretisation of \eqref{pbintrot}.
We refer the reader to \cite{koala,eym-12-sma,dro-12-gra} for more details.

A gradient scheme can be viewed as a general formulation of several discretisations of \eqref{pbintrot},
that are based on a nonconforming approximation of the weak formulation of the problem.
This approximation is constructed by using discrete space and mappings, the
set of which are called a gradient discretisation.

\begin{definition}\label{def-stcons}{\bf (Space-Time gradient discretisation for homogeneous
Dirichlet boundary conditions)}

We say  that $\disc = (X_{\disc,0}, \Pi_\disc,\nabla_\disc, \interp_\disc,(t^{(n)})_{n=0,\ldots,N})$ is a space-time gradient discretisation for homogeneous Dirichlet boundary conditions if

\begin{enumerate}
\item the set of discrete unknowns $X_{\disc,0}$ is a finite dimensional real vector space,
\item the linear mapping $\Pi_\disc~:~X_{\disc,0}\to L^{\infty}(\O)$ is a piecewise constant reconstruction operator in the following sense: there exists a set $I$ of degrees of freedom
and a family $(\Omega_i)_{i\in I}$ of disjoint subsets of $\Omega$ such that
$X_{\disc,0}=\R^I$, $\O=\bigcup_{i\in I}\O_i$ and,
for all $u=(u_i)_{i\in I}\in X_{\disc,0}$ and all $i\in I$, $\Pi_\disc u=u_i$
on $\O_i$,
\item the linear mapping $\nabla_\disc~:~X_{\disc,0}\to L^p(\O)^d$ gives a reconstructed discrete gradient. 
It must be chosen such that $\Vert \nabla_\disc \cdot \Vert_{L^p(\O)^d}$ is a norm on $X_{\disc,0}$,
\item $\interp_\disc~:~L^2(\O)\to X_{\disc,0}$ is a linear interpolation operator,
\item $t^{(0)}=0<t^{(1)}<t^{(2)}<\ldots<t^{(N)}=T$.
\end{enumerate}
We then set $\dt^{(n+\half)} = t^{(n+1)} -t^{(n)}$ for $n=0,\ldots,N-1$, and $\dt_\disc = \max_{n=0,\ldots,N-1} \dt^{(n+\half)}$. We define the dual semi-norm $|w|_{\star,\disc}$ of $w\in X_{\disc,0}$ by
\be
|w|_{\star,\disc}=\sup\left\{\int_\O \Pi_{\disc} w(\x)\Pi_{\disc}z(\x)\d\x\,:\,
z\in X_{\disc,0}\,,\;||\nabla_\disc z||_{L^p(\O)^d}= 1\right\}.
\label{nstnormemel}
\ee
\end{definition}

\begin{remark}[Boundary conditions]
Other boundary conditions can be seamlessly handled by gradient schemes, see \cite{koala}. 
\end{remark}

\begin{remark}[Nonlinear function of the elements of $X_{\disc,0}$]\label{rem:nlfxd}
Let $\disc$ be a gradient discretisation in the sense of Definition  \ref{def-stcons}. For any $\chi:\R\mapsto \R$
and any $u=(u_i)_{i\in I}\in X_{\disc,0}$, we define $\chi_I(u)\in X_{\disc,0}$ by
$\chi_I(u)=(\chi(u_i))_{i\in I}$. As indicated by
the subscript $I$, this definition depends on the choice of the degrees of freedom
in $X_{\disc,0}$. That said, these degrees of freedom are usually canonical and
the index $I$ can be dropped. An important consequence of the fact that
$\Pi_{\disc}$ is a \emph{piecewise constant} reconstruction is the following:
\be\label{prop:pcr}
\forall \chi:\R\mapsto \R\,,\;
\forall u\in X_{\disc,0}\,,\quad \Pi_\disc \chi(u)=\chi(\Pi_\disc u).
\ee
\end{remark}

It is customary to use the notations $\Pi_\disc$ and $\grad_\disc$ also for
space-time dependent functions. Moreover, we will need a notation for
the jump-in-time of piecewise constant functions in time.
Hence, if $(v^{(n)})_{n=0,\ldots,N}\subset X_{\disc,0}$,
we set
\be\ba
\hbox{for a.e. }\x\in \O,\ \Pi_{\disc} v(\x,0) = \Pi_{\disc} v^{(0)}(\x)\mbox{ and, }
\forall n=0,\ldots,N-1\,,\;\forall t\in   (t^{(n)},t^{(n+1)}],\\
\qquad \Pi_\disc v(\x,t) = \Pi_\disc v^{(n+1)}(\x)\,,\;
\grad_\disc v(\x,t) = \grad_\disc v^{(n+1)}(\x)\\
\qquad \dsp\mbox{and }  \delta_{\disc} v(t) = \delta_{\disc}^{(n+\half)}v:=\frac{v^{(n+1)}-v^{(n)}}{\dt^{(n+\half)}}\in X_{\disc,0}.
\ea\label{def-stfunctions}\ee

If $\disc = (X_{\disc,0}, \Pi_\disc,\nabla_\disc, \interp_\disc,(t^{(n)})_{n=0,\ldots,N})$
is a space-time gradient discretisation in the sense of Definition \ref{def-stcons},
the associated gradient scheme for Problem \eqref{pbintrot} is obtained
by replacing in this problem the continuous
space and mappings with their discrete ones. Using the notations in Remark \ref{rem:nlfxd}, the
implicit-in-time gradient
scheme therefore consists in
considering a sequence $(u^{(n)})_{n=0,\ldots,N}\subset X_{\disc,0}$ such that
\be\left\{\ba
\mbox{$u^{(0)}=\interp_\disc u_{\rm ini}$ and, for all $v=(v^{(n)})_{n=1,\ldots,N}\subset X_{\disc,0}$,} \\
\dsp \int_0^T\int_\O \left[\Pi_\disc \delta_{\disc}\beta(u)(\x,t) \Pi_\disc v(\x,t) + \bfa(\x, \Pi_\disc \nu(u)(\x,t),\grad_\disc \zeta( u)(\x,t))\cdot\grad_\disc v(\x,t)\right]
\d\x\d t \\
\qquad\dsp = \int_0^T\int_\O f(\x,t)  \Pi_\disc v(\x,t) \d\x\d t.
\ea\right.\label{ellgenfdiscllt}\ee

\begin{remark}[Time-stepping] 
Scheme \eqref{ellgenfdiscllt} is implicit-in-time because of the choice,
in the definitions of $\Pi_\disc$ and $\nabla_\disc$ in \eqref{def-stfunctions},
of $v^{(n+1)}$ when $t\in (t^{(n)},t^{(n+1)}]$.
As a consequence, $u^{(n+1)}$ appears in $\bfa(\x,\cdot,\cdot)$
in \eqref{ellgenfdiscllt} for $t\in (t^{(n)},t^{(n+1)}]$.
Instead of a fully implicit method, we could as well consider a Crank-Nicolson scheme
or any scheme between those two ($\theta$-scheme). This would consist in
choosing $\theta\in [\frac{1}{2},1]$ and in replacing these terms $u^{(n+1)}$
with $u^{(n+\theta)}=\theta u^{(n+1)}+(1-\theta)u^{(n)}$. All results established here for
\eqref{ellgenfdiscllt} would hold for such a scheme. We refer the reader to
the treatment done in \cite{dro-12-gra} for the details.
\end{remark}

\subsection{Properties of gradient discretisations}\label{sec:propgs}

In order to establish the convergence of the associated gradient schemes, sequences
of space-time gradient discretisations are required to satisfy four properties: \emph{coercivity},
\emph{consistency}, \emph{limit-conformity} and \emph{compactness}.

\begin{definition}[Coercivity] \label{def-coer}
If $\disc$ is a space-time gradient discretisation in the sense of Definition
\ref{def-stcons}, the norm of $\Pi_\disc$ is denoted by
\[
C_\disc=\max_{v\in X_{\disc,0}\backslash\{0\}}
\frac{||\Pi_\disc v||_{L^p(\O)}}{||\nabla_\disc v||_{L^p(\O)^d}}.
\]
A sequence $(\disc_m)_{m\in\N}$ of space-time gradient discretisations in the sense
of Definition \ref{def-stcons} is said to be coercive
if there exists $C_P\ge 0$ such that, for any $m\in\N$, $C_{\disc_m}\le C_P$.
\end{definition}

\begin{definition}[Consistency] \label{def-cons}
If $\disc$ is a space-time gradient discretisation in the sense of Definition
\ref{def-stcons}, we define
\be
\forall\varphi\in L^2(\O)\cap W^{1,p}_0(\O),\ \widehat{S}_\disc(\varphi)  = \min_{w\in X_{\disc,0}}\left(||\Pi_{\disc} w-\varphi||_{L^{\max(p,2)}(\O)}
+||\nabla_{\disc} w-\nabla\varphi||_{L^p(\O)^d}\right).
\label{def:sdisc}\ee
A sequence $(\disc_m)_{m\in\N}$ of space-time gradient discretisations in the sense
of Definition \ref{def-stcons} is said to be consistent if
\begin{itemize}
\item for all $\varphi \in L^2(\O)
\cap W^{1,p}_0(\O)$, $\widehat{S}_{\disc_m}(\varphi)\to 0$ as $m\to\infty$,
\item for all $\varphi \in L^2(\O)$, $\Pi_{\disc_m}\interp_{\disc_m}\varphi\to \varphi$
in $L^2(\O)$ as $m\to\infty$, and
\item $\dt_{\disc_m}\to 0$ as $m\to\infty$.
\end{itemize}
\end{definition}

\begin{definition}[Limit-conformity] \label{def-graddivcons}
If $\disc$ is a space-time gradient discretisation in the sense of Definition
\ref{def-stcons} and $W^{\div,p'}(\O)=\{\bvarphi\in L^{p'}(\O)^d\,:\,\div\bvarphi\in L^{p'}(\O)\}$, we define
\be
\begin{array}{l}
\dsp\forall \bvarphi\in W^{\div,p'}(\O)\,,\;
\dsp W_{\disc}(\bm{\varphi}) = \max_{u\in X_{\disc,0}\setminus\{0\}}\frac {\left\vert
\dsp\int_\O \left(\grad_{\disc} u(\x)\cdot\bvarphi(\x) + \Pi_{\disc} u(\x) \div\bvarphi(\x)\right)  \d\x \right\vert} {\Vert  \nabla_\disc u \Vert_{L^p(\O)^d}}.
\end{array}
\label{defwdisc}\ee
A sequence $(\disc_m)_{m\in\N}$ of space-time gradient discretisations in the sense
of Definition \ref{def-stcons} is said to be limit-conforming if, for all $\bm{\varphi}\in W^{\div,p'}(\O)$, $ W_{\disc_m}(\bm{\varphi})\to 0$ as  $m\to\infty$.
\end{definition}

\begin{remark}
The convergences $\widehat{S}_{\disc_m}\to 0$ on $L^2(\O)\cap W^{1,p}_0(\O)$
and $W_{\disc_m}\to 0$ on $W^{\div,p'}(\O)$ only need to be checked on dense subsets
of these spaces \cite{koala,eym-12-sma}.
\end{remark}
\begin{definition}[Compactness] \label{def-comp}
If $\disc$ is a space-time gradient discretisation in the sense of Definition
\ref{def-stcons}, we define
\[
\forall \bxi\in\R^d\,,\;T_\disc(\bxi)=\max_{v\in X_{\disc,0}\backslash\{0\}}
\frac{||\Pi_\disc v(\cdot+\bxi)-\Pi_\disc v||_{L^p(\R^d)}}{||\nabla_\disc v||_{L^p(\O)^d}},
\]
where $\Pi_\disc v$ has been extended by $0$ outside $\O$.

A sequence $(\disc_m)_{m\in\N}$ of space-time gradient discretisations is
said to be compact if
\[
\lim_{\bxi\to 0} \sup_{m\in\N} T_{\disc_m}(\bxi)=0.
\]
\end{definition}

We refer the reader to \cite{dro-12-gra,koala} for 
a proof of the following lemma.

\begin{lemma}[Regularity of the limit]\label{lem:reglim}
Let $(\disc_m)_{m\in\N}$ be a sequen\-ce of space-time gradient discretisations,
in the sense of Definition \ref{def-stcons}, that is \emph{coercive} and \emph{limit-conforming}
in the sense of Definitions \ref{def-coer} and \ref{def-graddivcons}. Let, for any $m\in\N$,
$v_m=(v^{(n)}_m)_{n=0,\ldots,N_m}\subset X_{\disc_m,0}$ be such that, with
the notations in \eqref{def-stfunctions}, $(\nabla_{\disc_m}v_m)_{m\in\N}$
is bounded in $L^p(\O\times (0,T))^d$.

Then there exists $v\in L^p(0,T;W^{1,p}_0(\O))$ such that, up to a subsequence as $m\to\infty$,
$\Pi_{\disc_m}v_m\to v$ weakly in $L^p(\O\times (0,T))$ and $\nabla_{\disc_m}v_m
\to \nabla v$ weakly in $L^p(\O\times(0,T))^d$.
\end{lemma}

\subsection{Main results}

Uniform-in-time convergence of numerical solutions to schemes for parabolic
equations starts with a weak convergence with respect to the space
variable. This weak convergence is then used to prove a stronger convergence.
We therefore first recall a standard definition related to the weak topology of
$L^2(\O)$ (we also refer the reader to Proposition \ref{propweakunifconv}
in the appendix for a classical characterisation of the weak topology of bounded
sets in $L^2(\O)$).

\begin{definition}[Uniform-in-time $L^2(\O)$-weak convergence]
\label{defweakunifconv} 
Let $\langle\cdot,\cdot\rangle_{L^2(\O)}$ denote the inner product in $L^2(\O)$,
let $(u_m)_{m\in\N}$ be a sequence of functions $[0,T]\to L^2(\O)$
and let $u:[0,T]\mapsto L^2(\O)$.

We say that $(u_m)_{m\in\N}$ converges weakly in $L^2(\O)$ uniformly on $[0,T]$ to
$u$ if, for all $\varphi\in L^2(\O)$,
as $m\to\infty$ the sequence of functions $t\in [0,T]\to \langle u_m(t),\varphi\rangle_{L^2(\O)}$
converges uniformly on $[0,T]$ to the function $t\in [0,T]\to \langle u(t),\varphi\rangle_{L^2(\O)}$.
\end{definition}

Our first theorem states weak or space-time averaged convergence
properties of gradient schemes for \eqref{pbintrot}. These results
have already been established for Leray--Lions', Richards' and Stefan's models,
see \cite{dro-12-gra,eym-12-stef,zamm2013}.
The convergence proof we provide afterwards however covers more
non-linear model and is more compact than the previous proofs.

\begin{theorem}[Convergence of gradient schemes]\label{th:weakconv}
We assume \eqref{hypgnl} and we take a sequence $(\disc_m)_{m\in\N}$ of
space-time gradient discretisations, in the sense of Definition \ref{def-stcons},
that is coercive, consistent, limit-conforming and compact (see Section \ref{sec:propgs}).
Then for any $m\in\N$ there exists a solution $u_m$ to \eqref{ellgenfdiscllt} with
$\disc=\disc_m$.

Moreover, if we assume that
\be\label{hyp:struct}
(\forall s\in\R\,,\;\beta(s)=s)\quad
 \mbox{or} \quad
(\forall s\in\R\,,\;\zeta(s)=s),
\ee
then there exists a solution $\bu$ to \eqref{ellgenfllt} such that, up to a subsequence,
the following convergences hold as $m\to\infty$:
\be\label{conv:base}
\ba
\dsp \mbox{$\Pi_{\disc_m}\beta(u_m)\to \beta(\bu)$ weakly in $L^2(\O)$ uniformly on $[0,T]$
(see Definition \ref{defweakunifconv}),}\\
\dsp \mbox{$\Pi_{\disc_m}\nu(u_m)\to \nu(\bu)$ strongly in $L^1(\O\times(0,T))$,}\\
\dsp \mbox{$\Pi_{\disc_m}\zeta(u_m)\to \zeta(\bu)$ weakly in $L^p(\O\times(0,T))$,}\\
\dsp \mbox{$\nabla_{\disc_m}\zeta(u_m)\to \nabla\zeta(\bu)$ weakly in $L^p(\O\times(0,T))^d$.}
\ea
\ee
\end{theorem}

\begin{remark} Since $|\nu|\le L_\zeta|\beta|$ and $|\nu|\le L_\beta|\zeta|$,
the $L^\infty(0,T;L^2(\O))$ bound on $\Pi_{\disc_m}\beta(u_m)$ and the $L^p(\O\times (0,T))$
bound on $\Pi_{\disc_m}\zeta(u_m)$ (see Lemma \ref{estimldlp} and Definition
\ref{def-coer}) shows that the
strong convergence of $\Pi_{\disc_m}\nu(u_m)$ is also valid in
$L^q(0,T;L^r(\O))$ for any $(q,r)\in [1,\infty)\times [1,2)$, any
$(q,r)\in [1,p)^2$ and, of course, any space interpolated between these two cases.
\end{remark}

\begin{remark} We do not assume the existence of a solution $\bu$ to the continuous
problem, our convergence analysis will establish this existence.
\end{remark}

\begin{remark} Assumption \eqref{hyp:struct} covers Richards' and Stefan's models,
as well as many other non-linear parabolic equations. As we prove in
Section \ref{sec:hypstruct}, this assumption is actually not required if $p\ge 2$.
However, we first state and prove Theorem 
\ref{th:weakconv} under \eqref{hyp:struct} in order to simplify the presentation.
See also Remark \ref{rem-struct}.
\end{remark}

The main innovation of this paper is the following theorem, which states the
\emph{uniform-in-time strong-in-space} convergence
of numerical methods for fully non-linear degenerate parabolic equations with no regularity assumptions on the
data.

\begin{theorem}[Uniform-in-time convergence]\label{th:uniftime}
Under Assumptions \eqref{hypgnl}, let $(\disc_m)_{m\in\N}$ be a sequence of
space-time gradient discretisations, in the sense of Definition \ref{def-stcons},
that is coercive, consistent, limit-conforming and compact (see Section \ref{sec:propgs}).
We assume that $u_m$ is a solution to \eqref{ellgenfdiscllt}
with $\disc=\disc_m$ that converges as $m\to\infty$ to a solution $\bu$ of \eqref{ellgenfllt}
in the sense \eqref{conv:base}.

Then, as $m\to\infty$, $\Pi_{\disc_m}\nu(u_m)\to \nu(\bu)$ strongly in 
$L^\infty(0,T;L^2(\O))$.
\end{theorem}

\begin{remark} Since the functions $\Pi_{\disc_m}\nu(u_m)$ are piecewise
constant in time, their convergence in $L^\infty(0,T;L^2(\O))$ is actually
a uniform-in-time convergence (not ``uniform a.e. in time'').
\end{remark}

The last theorem completes our convergence result by stating the strong
space-time averaged convergence of the discrete gradients. Its proof is
inspired by the study of gradient schemes for Leray--Lions operators
made in \cite {dro-12-gra}.

\begin{theorem}[Strong convergence of gradients]\label{th:strgrad}
Under Assumptions \eqref{hypgnl}, let $(\disc_m)_{m\in\N}$ be a sequence of
space-time gradient discretisations, in the sense of Definition \ref{def-stcons},
that is coercive, consistent, limit-conforming and compact (see Section \ref{sec:propgs}).
We assume that $u_m$ is a solution to \eqref{ellgenfdiscllt}
with $\disc=\disc_m$ that converges as $m\to\infty$ to a solution $\bu$ of \eqref{ellgenfllt}
in the sense \eqref{conv:base}. We also assume that $\bfa$ is strictly monotone in the
sense:
\be\label{a:strictmon}
(\bfa(\x,s,\bxi) - \bfa(\x,s,\bchi))\cdot(\bxi-\bchi)> 0,\hbox{ for a.e. }\x\in\O,\ \forall s\in \R,\ \forall\bxi\not=\bchi\in\R^d.
\ee
Then, as $m\to\infty$, $\Pi_{\disc_m}\zeta(u_m)\to \zeta(\bu)$ strongly in $L^p(\O\times (0,T))$
and $\nabla_{\disc_m} \zeta(u_m)\to \nabla  \zeta(\bu)$ strongly in $L^p(\O\times(0,T))^d$.
\end{theorem}

\begin{remark}\label{rem-struct} Theorems \ref{th:uniftime} and \ref{th:strgrad}
do not require the structural assumption \eqref{hyp:struct}; they
only require that the convergences \eqref{conv:base} hold.
\end{remark}

\section{Preliminaries}\label{sec:prelim}

We establish here a few results which will be used in the analysis
of the gradient scheme \eqref{ellgenfdiscllt}.

\subsection{Uniform-in-time compactness for space-time gradient discretisations}

Aubin-Simon compactness results roughly consist in establishing the compactness of a sequence of space-time functions from some strong bounds on the functions with respect to the space variable
(typically, bounds in a Sobolev space with positive exponent)
and some weaker bounds on their time derivatives (typically, bounds in a Sobolev space with
a negative exponent, i.e. the dual of a Sobolev space with positive exponent).
Several variants exist, including for piecewise constant-in-time
functions appearing in the numerical approximation of parabolic equations
\cite{jungel1,jungel2,amann,gal-12-com}. Although quite strong in space,
the convergence results provided by these discrete versions of Aubin-Simon theorems 
are only averaged-in-time -- i.e. in an $L^p(0,T;E)$ space where $E$ is a normed space.

Theorem \ref{unifweakGDcomp} can be considered as a discrete form of an Aubin-Simon theorem,
that establishes a \emph{uniform-in-time} but \emph{weak-in-space} compactness result.
The corresponding convergence is therefore weaker than in Theorem \ref{th:uniftime},
but it is a critical initial
step for establishing the \emph{uniform-in-time strong-in-space} convergence result.
Given that the functions considered here are piecewise constant in time, it might
be surprising to obtain a uniform-in-time convergence result; everything
hinges on the fact that the jumps in time tend to vanish as the time step goes to zero.
The proof of Theorem \ref{unifweakGDcomp} is based on the results
in Section \ref{sec:comptime}, and in particular on the discontinuous Ascoli-Arzel\`a
theorem stated and proved there.

\begin{theorem}[Uniform-in-time weak-in-space discrete Aubin-Simon theorem]
\label{unifweakGDcomp}~

Let $T>0$ and take a sequence $(\disc_m)_{m\in\N} = (X_{\disc_m,0}, \Pi_{\disc_m},\nabla_{\disc_m}, \interp_{\disc_m},(t_m^{(n)})_{n=0,\ldots,N_m})_{m\in\N}$
of space-time gradient discretisations, in the sense of Definition \ref{def-stcons}, that
is consistent in the sense of Definition \ref{def-cons}.

For any $m\in\N$, let $v_m=(v_m^{(n)})_{n=0,\ldots,N_m} \subset X_{\disc_m,0}$.
If there exists $q>1$ and $C>0$ such that, for any $m\in\N$,
\begin{equation}
||\Pi_{\disc_m}v_m||_{L^\infty(0,T;L^2(\O))}\le C\quad\mbox{ and }\quad
\int_0^T |\delta_m v_m(t)|_{\star,\disc_m}^{q}\d t\le C,
\label{est-deltavm}\end{equation}
then the sequence $(\Pi_{\disc_m}v_m)_{m\in\N}$ is relatively compact
uniformly-in-time and weakly in $L^2(\O)$, i.e. it
has a subsequence that converges in the sense of Definition \ref{defweakunifconv}.

Moreover, any limit of such a subsequence is continuous $[0,T]\to L^2(\O)$
for the weak topology of $L^2(\O)$.
\end{theorem}

\begin{remark}\label{rem-deltavm}
The bound on $|\delta_m v_m|_{\star,\disc_m}$
is often a consequence of a numerical scheme satisfied by $v_m$ and
of a bound on $||\nabla_{\disc_m}v_m||_{L^p(\O \times (0,T))^d}$,
see the proof of Lemma \ref{estl1l2snstdt} for example.
\end{remark}

\begin{proof}
This result is a consequence of the discontinuous Ascoli-Arzel\`a theorem
(Theorem \ref{genascoli}) with $K=[0,T]$ and $E$ the ball of radius $C$ in $L^2(\O)$ endowed with the
weak topology. We let $(\varphi_l)_{l\in\N}\subset C^\infty_c(\O)$ be a dense
sequence in $L^2(\O)$ and equipp $E$ with the metric \eqref{def-distweak}
from these $\varphi_l$ (see Proposition \ref{propweakunifconv}).
The set $E$ is metric compact and therefore complete, and the functions
$\Pi_{\disc_m}v_m$ take their values in $E$.
It remains to estimate $d_E(v_m(s),v_m(s'))$. In what follows, we drop the index $m$ in $\disc_m$
for the sake of legibility.

Let us define the interpolant $P_{\disc}\varphi_l\in X_{\disc,0}$ by
\be\label{def-PD}
P_\disc\varphi_l=\argmin_{w\in X_{\disc,0}}\left(||\Pi_\disc w-\varphi_l||_{L^{\max(p,2)}(\O)}
+||\nabla_\disc w-\nabla\varphi_l||_{L^p(\O)^d}\right).
\ee
For $0\le s\le s'\le T$, by writing $\Pi_\disc v_m(s')-\Pi_{\disc}v_m(s)$
as the sum of its jumps $\dt^{(n+\half)}\Pi_\disc\delta_\disc^{(n+\half)}v_m$
at the points $(t^{(n)})_{n=n_1,\ldots,n_2}$ between $s$ and $s'$,
the definition of $|\cdot|_{\star,\disc}$, H\"older's inequality and Estimate \eqref{est-deltavm}
give
\begin{multline}\label{added1}
\left|\int_\O \left(\Pi_\disc v_m(\x,s')-\Pi_{\disc}v_m(\x,s)\right)\Pi_\disc P_\disc\varphi_l(\x)\d\x
\right|\\
= \left|\int_{t^{(n_1)}}^{t^{(n_2+1)}}\int_\O \Pi_{\disc}\delta_\disc v(t)(\x)
\Pi_\disc P_\disc\varphi_l(\x)\d\x\d t\right|
\le C^{1/q} (t^{(n_2+1)}-t^{(n_1)})^{1/q'}||\nabla_\disc P_\disc\varphi_l||_{L^p(\O)^d}. 
\end{multline}
By definition of $P_\disc$, we have
\[
||\Pi_\disc P_\disc\varphi_l-\varphi_l||_{L^2(\O)}\le \widehat{S}_\disc(\varphi_l)
\]
and
\[
||\nabla_\disc P_\disc\varphi_l||_{L^p(\O)^d}\le \widehat{S}_\disc(\varphi_l)
 +||\nabla\varphi_l||_{L^p(\O)^d}\le C_{\varphi_l}
\]
with $C_{\varphi_l}$ not depending on $\disc$ (and therefore on $m$).
Since $t^{(n_2+1)}-t^{(n_1)}\le |s'-s|+\dt$ and $(\Pi_{\disc}v_m)_{m\in\N}$
is bounded in $L^\infty(0,T;L^2(\O))$, we deduce from \eqref{added1} that
\begin{multline*}
\left|\int_\O \left(\Pi_\disc v_m(\x,s')-\Pi_{\disc}v_m(\x,s)\right)\varphi_l(\x)\d\x
\right|\\
\le \left|\int_\O \left(\Pi_\disc v_m(\x,s')-\Pi_{\disc}v_m(\x,s)\right)\Pi_\disc P_\disc \varphi_l(\x)\d\x
\right|
+2||\Pi_\disc v_m||_{L^\infty(0,T;L^2(\O))}||\Pi_\disc P_\disc \varphi_l-\varphi_l||_{L^2(\O)}\\
\le  2C\widehat{S}_\disc(\varphi_l)+C^{1/q}C_{\varphi_l}|s'-s|^{1/q'}+C^{1/q}C_{\varphi_l}\dt^{1/q'}.
\end{multline*}
Plugged into the definition \eqref{def-distweak} of the distance in $E$, this shows that
\begin{eqnarray*}
d_E\Big(\lefteqn{\Pi_{\disc}v_m(s'),\Pi_{\disc}v_m(s)\Big)}&&\\
&\le& \sum_{l\in\N}\frac{\min(1,C^{1/q'}C_{\varphi_l}|s'-s|^{1/q'})}{2^l}
+\sum_{l\in\N}\frac{\min(1,2C\widehat{S}_{\disc_m}(\varphi_l)+C^{1/q'}C_{\varphi_l}\dt_m^{1/q'})}{2^l}\\
&=:&\omega(s,s')+\delta_m.
\end{eqnarray*}
Using the dominated convergence theorem for series, we see that $\omega(s,s')\to 0$ as $s-s'\to 0$
and that $\delta_m\to 0$ as $m\to\infty$ (we invoke the consistency to establish
that $\lim_{m\to\infty}\widehat{S}_{\disc_m}(\varphi_l)\to 0$ for any $l$).
Hence, the assumptions of Theorem \ref{genascoli} are satisfied and the proof
is complete. \end{proof}

\subsection{Technical results}

We state here a family of technical lemmas, starting with a few
properties on $\nu$ and $B$.

\begin{lemma} 
Under Assumptions \eqref{hypgnl} there holds
\begin{equation}\label{relchizeta}
 |\nu(a) - \nu(b)| \le L_\beta |\zeta(a) - \zeta(b)|,
\end{equation}
\begin{equation}\label{relchizetabeta}
 (\nu(a) - \nu(b))^2 \le L_\beta L_\zeta (\zeta(a) - \zeta(b)) (\beta(a) - \beta(b)).
\end{equation}
The function $B$ is convex continuous on $\overline{R_\beta}$,
the function $B\circ\beta:\R\to [0,\infty)$ is continuous,
\begin{equation}\label{relBbeta}
\forall s\in\R\,,\quad B(\beta(s))=\int_0^s \zeta(q)\beta'(q)\d q\,,
\end{equation}
\begin{equation}\label{growthB}
\exists K_0,K_1,K_2>0\mbox{ such that},\;
\forall s\in\R\,,\quad K_0 \beta(s)^2 - K_1\le B(\beta(s))\le K_2 s^2\,,
\end{equation}
\begin{equation}\label{sousdiffB}
\forall a\in\R\,,\;\forall S\in\overline{R_\beta}\,,\quad \zeta(a)(S-\beta(a))\le B(S) -B(\beta(a)),
\end{equation}
and
\be
\forall s,s'\in \R\,,\;
(\nu(s)-\nu(s'))^2\le 4 L_\beta L_\zeta  \left[B(\beta(s))+ B(\beta(s'))-2 B\left(\frac{\beta(s)+\beta(s')}{2}\right)\right].
\label{Bunifconv}\ee
\label{lembetaB}\end{lemma}

\begin{proof}

Inequality \eqref{relchizeta} is a straightforward consequence
of the estimate $\nu'=\zeta'\beta'\le L_\beta \zeta'$. Note that the
same inequality also holds with $\beta$ and $\zeta$ swapped.
Since these functions are non-decreasing, Inequality \eqref{relchizetabeta}
follows from \eqref{relchizeta} and the similar inequality
with $\beta$ and $\zeta$ swapped.

Since $\beta$ is non-decreasing, $\beta_r$ is also non-decreasing on $R_\beta$
and therefore locally bounded on $R_\beta$. Hence, $B$ is
locally Lipschitz-continuous on $R_\beta$, with an a.e. derivative $B'=\zeta(\beta_r)$.
$B'$ is therefore non-decreasing and $B$ is convex continuous on $R_\beta$, and thus also
on $\overline{R_\beta}$ by choice of its values at the endpoints of $R_\beta$.

To prove \eqref{relBbeta}, we denote by $P\subset R_\beta$ the countable set of plateaux values of $\beta$,
i.e. the $y\in\R$ such that $\beta^{-1}(\{y\})$ is not reduced to a singleton.
If $s\not\in \beta^{-1}(P)$ then $\beta^{-1}(\{\beta(s)\})$ is
the singleton $\{s\}$ and therefore $\beta_r(\beta(s))=s$. Moreover,
$\beta_r$ is continuous at $\beta(s)$ and thus $B$ is differentiable
at $\beta(s)$ with $B'(\beta(s))=\zeta(\beta_r(\beta(s)))=\zeta(s)$.
Since $\beta$ is differentiable a.e., we
deduce that, for a.e. $s\not\in \beta^{-1}(P)$, $(B(\beta))'(s)=B'(\beta(s))\beta'(s)
=\zeta(s)\beta'(s)$.
The set $\beta^{-1}(P)$ is a union of intervals on which $\beta$ and thus
$B(\beta)$ are locally constant; hence, for a.e. $s$ in this set,
$(B(\beta))'(s)=0$ and $\zeta(s)\beta'(s)=0$.
Hence, the locally Lipschitz-continuous functions $B(\beta)$ and $s\to \int_0^s \zeta(q)\beta'(q)\d q$
have identical derivatives a.e. on $\R$ and take the same value at $s=0$. They are thus
equal on $\R$ and the proof of \eqref{relBbeta} is complete. 

The continuity of $B\circ\beta$ is an obvious consequence of \eqref{relBbeta}.
The second inequality in \eqref{growthB} can also be easily deduced
from \eqref{relBbeta} by noticing that $|\zeta(s)\beta'(s)|\le L_\zeta L_\beta |s|$
(we can take $K_2=\frac{L_\zeta L_\beta}{2}$).
To prove the first inequality in \eqref{growthB}, we start by inferring from \eqref{hypzeta} the
existence of $S>0$ such that $|\zeta(q)|\ge \frac{M_0}{2}|q|\ge \frac{M_0}{2L_\beta}|\beta(q)|$
whenever $|q|\ge S$. We then write, for $s\ge S$,
\[
B(\beta(s))=\int_0^S \zeta(q)\beta'(q)\d q + \int_S^s \zeta(q)\beta'(q)\d q
\ge  \frac{M_0}{2L_\beta}\int_S^s \beta(q)\beta'(q)\d q
= \frac{M_0}{4L_\beta}\left(\beta(s)^2-\beta(S)^2\right).
\]
A similar inequality holds for $s\le -S$ (with $\beta(-S)$ instead of $\beta(S)$)
and the first inequality in \eqref{growthB}
therefore holds with $K_0=\frac{M_0}{4L_\beta}$ and $K_1=\frac{M_0}{4L_\beta}\max_{[-S,S]}\beta^2$.

We now prove \eqref{sousdiffB}, which states that $\zeta(a)$ belongs to the convex
sub-differential of $B$ at $\beta(a)$. We first start with the case $S\in R_\beta$,
that is $S=\beta(b)$ for some $b\in\R$.
If $\beta_r$ is continuous at $\beta(a)$ then this inequality is an obvious consequence
of the convexity of $B$ since $B$ is then differentiable at $\beta(a)$ with $B'(\beta(a))=\zeta(\beta_r(\beta(a)))=\zeta(a)$.
Otherwise, a plain reasoning also does the job:
\begin{multline*}
B(S)-B(\beta(a))=B(\beta(b))-B(\beta(a))\\
=\int_{a}^{b}\zeta(q)\beta'(q)\d q
=\int_{a}^{b}(\zeta(q)-\zeta(a))\beta'(q)\d q + \zeta(a)(\beta(b)-\beta(a))
\ge \zeta(a)(S-\beta(a)),
\end{multline*}
the inequality coming from the fact that $\beta'\ge 0$ and that $\zeta(q)-\zeta(a)$ has the same sign as $b-a$ when
$q$ is between $a$ and $b$. The general case $S\in \overline{R_\beta}$ is obtained by passing
to the limit on $b_n$ such that $\beta(b_n)\to S$ and by using the fact that $B$ has limits
(possibly $+\infty$) at the endpoints of $R_\beta$.

Let us now take $s,s'\in \R$. Let $\bar s\in\R$ be such that $\beta(\bar s) = \frac{\beta(s)+\beta(s')} 2$. We notice that
\begin{equation}\label{estBunifconv:1}
B(\beta(s))+ B(\beta(s')) -2 B(\beta(\bar s)) =
\int_{\bar s}^s (\zeta(q) - \zeta(\bar s)) \beta'(q)\d q + \int_{\bar s}^{s'} (\zeta(q) - \zeta(\bar s)) \beta'(q)\d q.
\end{equation}
We then notice that $|\zeta(q) - \zeta(\bar s)| \ge \frac 1 {L_\beta} |\nu(q) - \nu(\bar s)|$
and $\beta'(q)\ge \beta'(q)\frac{\zeta'(q)}{L_\zeta}=\frac{\nu'(q)}{L_\zeta}$.
If $\widetilde{s}=s$ or $s'$, since $\zeta(q)-\zeta(\bar s)$ has the same sign as
$\widetilde{s}-\bar s$ for all $q$ between $\bar s$ and $\widetilde{s}$, we can write
\begin{equation}\label{estBunifconv:2}
 \int_{\bar s}^{\widetilde{s}} (\zeta(q) - \zeta(\bar s)) \beta'(q)\d q \ge  \frac 1 {L_\beta L_\zeta}\int_{\bar s}^{\widetilde{s}} \nu'(q) (\nu(q) - \nu(\bar s))\d q = 
\frac 1 { 2 L_\beta L_\zeta} (\nu(\widetilde{s}) - \nu(\bar s))^2.
\end{equation}
Estimate \eqref{Bunifconv} follows from \eqref{estBunifconv:1}, \eqref{estBunifconv:2} and
the inequality $(\nu(s) - \nu(s'))^2 \le 2(\nu(s) - \nu(\bar s))^2 + 2 (\nu(s') - \nu(\bar s))^2$.
\end{proof}

The next lemma is an easy consequence of Fatou's lemma and the
fact that strongly lower semi-continuous convex functions are also
weakly lower semi-continuous.  We all the same provide its short proof.

\begin{lemma}\label{lem:Zconv}
Let $I$ be a closed interval of $\R$ and let $H:I\to (-\infty,\infty]$ be a convex continuous function
(continuity for possible infinite values, at the endpoints of $I$, corresponding to $H$ having
limits at these endpoints).
We denote by $L^2(\Omega;I)$ the convex set of functions in $L^2(\O)$
with values in $I$. Let $v\in L^2(\O;I)$ and let $(v_m)_{m\in\N}$ be a sequence of functions in $L^2(\O;I)$
that converges weakly to $v$ in $L^2(\O)$. Then 
\[
\int_\O H(v(\x))\d\x\le \liminf_{m\to\infty} \int_\O H(v_m(\x))\d\x.
\]
\end{lemma}

\begin{proof}

For $w\in L^2(\O;I)$ we set $\Phi(w)=\int_\O H(w(\x))\d\x$.
Since $H$ is convex, it is greater than a linear functional and $\Phi(w)$
is thus well defined in $(-\infty,\infty]$. Moreover, if $w_k\to w$ strongly
in $L^2(\Omega;I)$ then, up to a subsequence, $w_k\to w$ a.e. on $\O$
and therefore $H(w_k)\to H(w)$ a.e. on $\O$. Thanks to the
linear lower bound of $H$, we can apply Fatou's lemma to see that
$\Phi(w)\le \liminf_{k\to\infty}\Phi(w_k)$.

Hence, $\Phi$ is lower semi-continuous for the strong topology of $L^2(\O;I)$.
Since $\Phi$ (like $H$) is convex, we deduce that this lower semi-continuity property
is also valid for the weak topology of $L^2(\O;I)$, see \cite{ekeland-temam}.
The result of the lemma is just the translation of this weak lower semi-continuity of $\Phi$.
\end{proof}

The last technical result is a consequence of the Minty trick. It has been
proved and used in the $L^2$ case in \cite{eym-12-stef,koala}, but we need here an
extension to the non-Hilbertian case.

\begin{lemma}[Minty's trick]\label{lem-minty}
Let $H \in C^0(\R)$ be a nondecreasing function. Let $(X,\mu)$ be a measurable
set with finite measure and let $(u_n)_{n\in\N} \subset L^p(X)$, with $p>1$, satisfy
\begin{enumerate}
\item there exists  $u \in L^p(X)$ such that $(u_n)_{n\in\N}$
converges weakly to $u$ in $L^p(X)$;
\item $(H(u_n))_{n\in\N} \subset L^1(X)$ and there exists $w\in L^1(X)$ such that
$(H(u_n))_{n\in\N}$ converges strongly to $w$ in $L^1(X)$;
\end{enumerate}
Then $w = H(u)$ a.e. on $X$.
\end{lemma}

\begin{proof}

For $k,l>0$ we define the truncation at levels $-l$ and $k$ by $T_{k,l}(s)=\max(-l,\min(s,k))$
and we let $T_k=T_{k,k}$.
Since $H$ is non-decreasing, there exists sequences $(h_k)_{k\in\N}$ and $(m_k)_{k\in\N}$
that tend to $+\infty$ as $k\to\infty$ and such that $H(T_k(s))=T_{h_k,m_k}(H(s))$.
Thus, $H(T_k(u_n))\to T_{h_k,m_k}(w)$ in $L^1(X)$ as $n\to\infty$. Given that
$(H(T_k(u_n)))_{n\in\N}$ remains bounded in $L^\infty(X)$, its convergence
to $T_{h_k,m_k}(w)$ also holds in $L^{p'}(X)$.

Using fact that $H\circ T_k$ is non-decreasing, we write for any $g\in L^{p}(X)$
\[
\int_X (H(T_k(u_n))-H(T_k(g)))(u_n-g)\d\mu\ge 0.
\]
By strong convergence of $H(T_k(u_n))$ in $L^{p'}(X)$
and weak convergence of $u_n$ in $L^{p}(X)$, as well as the fact
that $H\circ T_k$ is bounded, we can take the
limit of this expression as $n\to \infty$ and we find
\be\label{minty1}
\int_X (T_{h_k,m_k}(w)-H(T_k(g)))(u-g)\d\mu\ge 0.
\ee
We then use Minty's trick. We pick a generic $\varphi\in L^{p}(X)$,
apply \eqref{minty1} to $g=u-t\varphi$, divide by $t$ and let $t\to \pm 0$
(using the dominated convergence theorem and the fact that $H\circ T_k$ is
continuous and bounded) to find
\[
\int_X (T_{h_k,m_k}(w)-H(T_k(u)))\varphi\d\mu = 0.
\]
Selecting $\varphi={\rm sign}(T_{h_k,m_k}(w)-H(T_k(u)))$, we deduce
that $T_{h_k,m_k}(w)=H(T_k(u))$ a.e. on $X$. Letting $k\to \infty$,
we conclude that $w=H(u)$ a.e. on $X$. \end{proof}

\subsection{Integration-by-parts for the continuous solution}

The last series of preliminary results are properties on the solution to
\eqref{ellgenfllt}, all based on the following integration-by-parts property.
This property, used in the proof of Theorems \ref{th:weakconv} and \ref{th:uniftime},
enables us to compute the value of the linear form
$\partial_t \beta(\bu)\in L^{p'}(0,T;W^{-1,p'}(\O))$ on the function $\zeta(\bu)
\in L^p(0,T;W^{1,p}_0(\O))$.
Because of the lack of regularity on $\bu$ and the double non-linearity ($\beta$
and $\zeta$), justifying this integration-by-parts is however not straightforward at all...

\begin{lemma}\label{ricipp}
Let us assume \eqref{hypzeta} and \eqref{hypbeta}.
Let $v:\O\times (0,T)\mapsto \R$ be measurable such that $\zeta(v)\in L^p(0,T;W^{1,p}_0(\O))$,
$B(\beta(v)) \in L^\infty(0,T;L^1(\O))$,
$\beta(v)\in C([0,T];L^2(\O)\weak)$ and $\partial_t \beta(v)\in L^{p'}(0,T;W^{-1,p'}(\O))$.
Then $t\in [0,T]\to \int_\O B(\beta(v)(\x,t))\d\x\in [0,\infty)$ is
continuous and, for all $t_1,t_2\in [0,T]$,
\begin{equation}\label{ippbeta}
\int_{t_1}^{t_2}\langle \partial_t \beta(v)(t),\zeta(v(\cdot,t))\rangle_{W^{-1,p'},W^{1,p}_0}\d t
=\int_\O B(\beta(v)(\x,t_2))\d\x-\int_\O B(\beta(v)(\x,t_1)) \d\x.
\end{equation}
\end{lemma}

\begin{remark} Similarly to the discussion at the end of Section \ref{sec:contprob},
we notice that it is important to keep in mind the separation between $\beta(v(\cdot,\cdot))$
and its continuous representative $\beta(v)(\cdot,\cdot)$.
\end{remark}

\begin{proof}

Without loss of generality, we assume that $0\le t_1<t_2\le T$.

\textbf{Step 1}: truncation, extension and approximation of $\beta(v)$.

We define $\overline{\beta(v)}:\R\to L^2(\O)$ by setting
\[
\overline{\beta(v)}(t)=\left\{\begin{array}{ll}
\beta(v)(t)&\mbox{ if $t\in [t_1,t_2]$},\\
\beta(v)(t_1)&\mbox{ if $t\le t_1$},\\
\beta(v)(t_2)&\mbox{ if $t\ge t_2$}.
\end{array}\right.
\]
By the continuity property of $\beta(v)$, this definition makes sense and gives
$\overline{\beta(v)}\in C(\R;L^2(\O)\weak)$ such that $\partial_t \overline{\beta(v)}
=\mathbf{1}_{(t_1,t_2)}\partial_t \beta(v)
\in L^{p'}(\R;W^{-1,p'}(\O))$ where $\mathbf{1}$ is the characteristic function
(no Dirac masses have been introduced at $t=t_1$ or $t=t_2$).
This regularity of $\partial_t\overline{\beta(v)}$ ensures that the function
$D_h\overline{\beta(v)}:\R\mapsto W^{-1,p'}(\O)$ defined by
\begin{equation}
\forall t\in\R\,,\;D_h\overline{\beta(v)}(t)=\frac{1}{h}\int_{t}^{t+h} \partial_t\overline{\beta(v)}(s)\d s=\frac{\overline{\beta(v)}(t+h)-\overline{\beta(v)}(t)}{h}
\label{defDh}\end{equation}
tends to $\partial_t \overline{\beta(v)}$ in $L^{p'}(\R;W^{-1,p'}(\O))$ as $h\to 0$.

\medskip

\textbf{Step 2}: we prove that $||B(\overline{\beta(v)}(t))||_{L^1(\O)}\le 
||B(\beta(v))||_{L^\infty(0,T;L^1(\O))}$ for \emph{all} $t\in \R$
(not only for a.e. $t$).

Let $t\in [t_1,t_2]$. Since $\beta(v)(\cdot,\cdot)=\beta(v(\cdot,\cdot))$ a.e. on $\O\times (t_1,t_2)$,
there exists a sequence $t_n\to t$ such that $\beta(v)(\cdot,t_n)=\beta(v(\cdot,t_n))$ in $L^2(\O)$
and $||B(\beta(v)(\cdot,t_n))||_{L^1(\O)}\le ||B(\beta(v))||_{L^\infty(0,T;L^1(\O))}$
for all $n$.
As $\beta(v)\in C([0,T];L^2(\O)\weak)$, we have $\beta(v)(\cdot,t_n)\to \beta(v)(\cdot,t)$
weakly in $L^2(\O)$.
We then use the convexity of $B$ and Lemma \ref{lem:Zconv}
to write, thanks to our choice of $t_n$,
\[
\int_\O B(\beta(v)(\x,t))\d\x\le\liminf_{n\to\infty}\int_\O B(\beta(v)(\x,t_n))\d\x
\le ||B(\beta(v))||_{L^\infty(0,T;L^1(\O))}
\]
and the proof is complete for $t\in [t_1,t_2]$. The result for $t\le t_1$ or
$t\ge t_2$ is obvious since $\overline{\beta(v)}(t)$ is then either $\beta(v)(t_1)$
or $\beta(v)(t_2)$.

\medskip

\textbf{Step 3}: We prove that for all $\tau\in\R$ and a.e. $t\in (t_1,t_2)$,
\be\label{dualprod}
\langle \overline{\beta(v)}(\tau)-\beta(v)(t),\zeta(v(\cdot,t))\rangle_{W^{-1,p'},W^{1,p}_0}
\le \int_\O B(\overline{\beta(v)}(\x,\tau))-B(\beta(v)(\x,t))\d\x.
\ee
If we could just replace the duality
product $W^{-1,p'}$--$W^{1,p}_0$ with an $L^2$ inner product, this formula
would be a straightforward consequence of \eqref{sousdiffB}. The problem is that
nothing ensures that $\zeta(v)(t)\in L^2(\O)$ for a.e. $t$.

We first notice that $\overline{\beta(v)}(\tau)-\beta(v)(t)
=\int_t^{\tau}\partial_t \overline{\beta(v)}(s)\d s$ belongs to
$W^{-1,p'}(\O)$ so the left-hand side of \eqref{dualprod} makes sense
provided that $t$ is chosen such that $\zeta(v(\cdot,t))\in W^{1,p}_0(\O)$
(which we do from here on).
To deal with the fact that $\zeta(v(\cdot,t))$ does not necessarily belong
to $L^2(\O)$, we replace it with a truncation. As in the proof of Lemma \ref{lem-minty},
we introduce $T_{k,l}(s)=  \max(-l,\min(s,k))$ and we let $T_k=T_{k,k}$.
By the monotony assumption \eqref{hypzeta} on $\zeta$ we see that
there exists sequences $(r_k)_{k\in\N}$ and $(l_k)_{k\in\N}$
that tend to $+\infty$ as $k\to+\infty$ and such that
$\zeta(T_k(v(\cdot,t)))=T_{r_k,l_k}(\zeta(v(\cdot,t)))$. Hence, $\zeta(T_k(v(\cdot,t)))\in W^{1,p}_0(\O)$
and converges, as $k\to \infty$, to $\zeta(v(\cdot,t))$ in $W^{1,p}_0(\O)$.

We can therefore write
\begin{multline}\label{dualprod1}
\langle \overline{\beta(v)}(\tau)-\beta(v)(t),\zeta(v(\cdot,t)) \rangle_{W^{-1,p'},W^{1,p}_0}
=\lim_{k\to\infty}
\langle \overline{\beta(v)}(\tau)-\beta(v)(t),\zeta(T_k(v(\cdot,t))) \rangle_{W^{-1,p'},W^{1,p}_0}\\
=\lim_{k\to\infty}
\int_\O \left[\overline{\beta(v)}(\x,\tau)-\beta(v(\x,t))\right]\zeta(T_k(v(\x,t)))\d\x,
\end{multline}
the replacement of the duality product by an $L^2(\O)$ inner product being
justified since $\overline{\beta(v)}(\tau)-\beta(v)(t)$ and $\zeta(T_k(v(\cdot,t)))$ both belong to
$L^2(\O)$. We also used that,
for a.e. $t\in (t_1,t_2)$, $\beta(v)(\cdot,t)=\beta(v(\cdot,t))$ a.e. on
$\O$; hence \eqref{dualprod1} is valid for a.e. $t\in (t_1,t_2)$.

We then write $\beta(v(\x,t))=\beta(T_k(v(\x,t)))+\left[\beta(v(\x,t))-\beta(T_k(v(\x,t)))\right]$
and apply \eqref{sousdiffB} with $S=\overline{\beta(v)}(\x,\tau)$ and $a=T_k(v(\x,t))$
to find
\begin{multline*}
\int_\O \left[\overline{\beta(v)}(\x,\tau)-\beta(v(\x,t))\right]\zeta(T_k(v(\x,t)))\d\x\\
=\int_\O \left[\overline{\beta(v)}(\x,\tau)-\beta(T_k(v(\x,t)))\right]\zeta(T_k(v(\x,t)))\d\x\\
-\int_\O \left[\beta(v(\x,t))-\beta(T_k(v(\x,t)))\right]\zeta(T_k(v(\x,t)))\d\x\\
\le \int_\O B(\overline{\beta(v)}(\x,\tau))-B(\beta(T_k(v(\x,t))))\d\x
-\int_\O \left[\beta(v(\x,t))-\beta(T_k(v(\x,t)))\right]\zeta(T_k(v(\x,t)))\d\x.
\end{multline*}
By the monotony of $\beta$, the sign of $\zeta$ and by
studying the cases $v(\x,t)\ge k$, $-k\le v(\x,t)\le k$ and $v(\x,t)\le -k$, we notice that the last
integrand is everywhere non-negative. We can therefore write
\[
\int_\O \left[\overline{\beta(v)}(\x,\tau)-\beta(v(\x,t))\right]\zeta(T_k(v(\x,t)))\d\x
\le \int_\O B(\overline{\beta(v)}(\x,\tau))-B(\beta(T_k(v(\x,t))))\d\x.
\]
We then use the continuity of $B\circ \beta$ and Fatou's lemma to deduce
\begin{multline*}
\limsup_{k\to\infty}
\int_\O \left[\overline{\beta(v)}(\x,\tau)-\beta(v(\x,t))\right]\zeta(T_k(v(\x,t)))\d\x\\
\le \int_\O B(\overline{\beta(v)}(\x,\tau))\d\x-\liminf_{k\to\infty}\int_\O B(\beta(T_k(v(\x,t))))\d\x\\
\le \int_\O B(\overline{\beta(v)}(\x,\tau))\d\x-\int_\O B(\beta(v(\x,t)))\d\x
\end{multline*}
which, combined with \eqref{dualprod1}, concludes the proof of \eqref{dualprod}
(recall that $t$ has been chosen such that $\beta(v(\cdot,t))=\beta(v)(\cdot,t)$
a.e. on $\O$).

\medskip

\textbf{Step 4}: proof of the formula

Since $\mathbf{1}_{(t_1,t_2)}\zeta(v) \in L^p(\R;W^{1,p}_0(\O))$
and $D_h\overline{\beta(v)}\to\partial_t\overline{\beta(v)}$
in $L^{p'}(\R;W^{-1,p'}(\O))$ as $h\to 0$, we have
\begin{multline}\label{debbeta}
\int_{t_1}^{t_2}\langle \partial_t \beta(v)(t),\zeta(v(\cdot,t))\rangle_{W^{-1,p'},W^{1,p}_0}\d t
=\int_\R\langle \partial_t \overline{\beta(v)}(t),\mathbf{1}_{(t_1,t_2)}(t)\zeta(v(\cdot,t))\rangle_{W^{-1,p'},W^{1,p}_0}\d t\\
=\lim_{h\to 0}\int_\R \langle D_h\overline{\beta(v)}(t),\mathbf{1}_{(t_1,t_2)}(t)\zeta(v(\cdot,t))\rangle_{W^{-1,p'},W^{1,p}_0}\d t\\
=
\lim_{h\to 0}\frac{1}{h}\int_{t_1}^{t_2} \langle \overline{\beta(v)}(t+h)-\overline{\beta(v)}(t),\zeta(v(\cdot,t)\rangle_{W^{-1,p'},W^{1,p}_0}\d t.
\end{multline}
We then use \eqref{dualprod} for a.e. $t\in (t_1,t_2)$ to obtain, for $h$ small enough
such that $t_1+h<t_2$,
\begin{eqnarray}
\frac{1}{h}\int_{t_1}^{t_2} \lefteqn{\langle \overline{\beta(v)}(t+h)-\overline{\beta(v)}(t),\zeta(v(\cdot,t))\rangle_{W^{-1,p'},W^{1,p}_0}\d t}&&\nonumber\\
&\le&\frac{1}{h}\int_{t_1}^{t_2} \int_\O B(\overline{\beta(v)}(\x,t+h))-B(\overline{\beta(v)}(\x,t))\d\x\d t\nonumber\\
&=&\frac{1}{h}\int_{t_2}^{t_2+h} \int_\O B(\overline{\beta(v)}(\x,t))\d\x\d t
-\frac{1}{h}\int_{t_1}^{t_1+h} \int_\O B(\overline{\beta(v)}(\x,t))\d\x\d t\label{S2used}\\
&=&\int_\O B(\beta(v)(\x,t_2))\d\x
-\frac{1}{h}\int_{t_1}^{t_1+h} \int_\O B(\beta(v)(\x,t))\d\x\d t.\nonumber
\end{eqnarray}
We used the estimate in Step 2 to justify the separation of the integrals in \eqref{S2used}.
We now take the $\limsup$ as $h\to 0$ of this inequality, using again Step 2 to see that
$B(\beta(v)(\cdot,t_2))$ is integrable and therefore take its integral out of the $\limsup$. Coming back
to \eqref{debbeta} we obtain
\begin{multline}\label{debfin}
\int_{t_1}^{t_2}\langle \partial_t \beta(v)(t),\zeta(v(\cdot,t))\rangle_{W^{-1,p'},W^{1,p}_0}\d t\\
\le 
\int_\O B(\beta(v)(\x,t_2))\d\x
-\liminf_{h\to 0}\frac{1}{h}\int_{t_1}^{t_1+h} \int_\O B(\beta(v)(\x,t))\d\x\d t.
\end{multline}
Since $\beta(v)\in C([0,T];L^2(\O)\weak)$, as $h\to 0$ we have
$\frac{1}{h}\int_{t_1}^{t_1+h}\beta(v)(t)\d t\to \beta(v)(t_1)$ weakly in $L^2(\O)$.
Hence, the convexity of $B$, Lemma \ref{lem:Zconv} and Jensen's inequality give
\begin{multline*}
\int_\O B(\beta(v)(\x,t_1))\d\x
\le \liminf_{h\to 0}\int_\O B\left(\frac{1}{h}\int_{t_1}^{t_1+h}\beta(v)(\x,t)\d t\right)\d\x\\
\le \liminf_{h\to 0}\int_\O \frac{1}{h}\int_{t_1}^{t_1+h} B(\beta(v)(\x,t))\d t\d\x.
\end{multline*}
Plugged into \eqref{debfin}, this inequality shows that \eqref{ippbeta} holds with $\le$ instead of $=$.
The reverse inequality is obtained by reversing the time. We consider
$\widetilde{v}(t)=v(t_1+t_2-t)$. Then $\zeta(\widetilde{v})$, $B(\beta(\widetilde{v}))$
and $\beta(\widetilde{v})$
have the same properties as $\zeta(v)$, $B(\beta(v))$ and $\beta(v)$, and $\beta(\widetilde{v})$
takes values $\beta(v)(t_1)$ at $t=t_2$ and $\beta(v)(t_2)$ at $t=t_1$.
Applying \eqref{ippbeta} with ``$\le$'' instead of ``$=$'' to $\widetilde{v}$ and using the fact
that $\partial_t \beta(\widetilde{v})(t)=-\partial_t \beta(v)(t_1+t_2-t)$,
we obtain \eqref{ippbeta} with ``$\ge$'' instead of ``$=$'' and the proof of \eqref{ippbeta} is complete.

The continuity of $t\in [0,T]\mapsto \int_\O B(\beta(v)(\x,t))\d\x$ is straightforward from
\eqref{ippbeta} as the left-hand side of this relation is continuous with respect to $t_1$ and $t_2$. 
\end{proof}

The following corollary states continuity properties and an essential formula
on the solution to \eqref{ellgenfllt}.

\begin{corollary}\label{ric:bou}
Under Assumptions \eqref{hypomega}--\eqref{hypfgnlt}, if $\bu$ is a solution of \eqref{ellgenfllt}
then:
\begin{enumerate}
\item the function $t\in [0,T]\mapsto \int_\O B(\beta(\bu)(\x,t))\d\x\in[0,\infty)$ is continuous
and bounded,
\item for any $T_0\in [0,T]$,
\begin{multline}
\int_\O B(\beta(\bu)(\x,T_0)) \d \x 
+ \int_{0}^{T_0}   \int_\O \bfa(\x,\nu(\bu(\x,t)),\nabla\zeta(\bu)(\x,t))
\cdot\nabla \zeta(\bu)(\x,t)\d\x \d t \\
=  \int_\O B(\beta(u_{\rm ini}(\x))) \d \x 
+  \int_{0}^{T_0}   \int_\O f(\x,t) \zeta(\bu)(\x,t) \d \x \d t, \label{form:ippsol}
\end{multline}
\item $\nu(\bu)$ is continuous $[0,T]\to L^2(\O)$.
\end{enumerate}
\end{corollary}

\begin{remark}
The continuity of $\nu(\bu)$ has to be understood in the same sense as the
continuity of $\beta(\bu)$, that is $\nu(\bu)$ is a.e. on $\O\times (0,T)$
equal to a continuous function $[0,T]\to L^2(\O)$. We use in particular the notation
$\nu(\bu)(\cdot,\cdot)$ for the continuous-in-time representative of $\nu(\bu(\cdot,\cdot))$,
similarly to the way we denote the continuous-in-time representative of $\beta(\bu(\cdot,\cdot))$.
\end{remark}

\begin{proof}

The continuity of $t\in [0,T]\mapsto \int_\O B(\beta(\bu)(\x,t))\d\x\in[0,\infty)$ and
Formula \eqref{form:ippsol} are straightforward consequences of Lemma \ref{ricipp}
with $v=\bu$ and using \eqref{ellgenfllt} with $\vbarre=\zeta(\bu)$. Note that the
bound on $\int_\O B(\beta(\bu)(\x,t))\d\x$ can be seen as a consequence of \eqref{form:ippsol},
or from Step 2 in the proof of Lemma \ref{ricipp}.

Let us prove the strong continuity of $\nu(\bu):[0,T]\mapsto L^2(\O)$. 
Let $\mathcal T$ be the set of $\tau\in [0,T]$ such that $\beta(\bu(\cdot,\tau))=\beta(\bu)(\cdot,\tau)$
a.e. on $\O$, and let $(s_l)_{l\in\N}$ and $(t_k)_{k\in\N}$ be two sequences in
$\mathcal T$ that converge to the same value $s$.
Invoking \eqref{Bunifconv} we can write
\begin{multline}\label{contnu1}
\int_\O (\nu(\bu(\x,s_l))-\nu(\bu(\x,t_k)))^2\d\x
\le 4L_\beta L_\zeta \left(\int_\O B(\beta(\bu)(\x,s_l))\d\x
+\int_\O B(\beta(\bu)(\x,t_k))\d\x\right)\\
-8L_\beta L_\zeta \int_\O B\left(\frac{\beta(\bu)(\x,s_l)+\beta(\bu)(\x,t_k)}{2}\right)\d\x.
\end{multline}
Since $\frac{\beta(\bu)(\cdot,s_l)+\beta(\bu)(\cdot,t_k)}{2}\to \beta(\bu)(\cdot,s)$
weakly in $L^2(\O)$ as $l,k\to\infty$,
Lemma \ref{lem:Zconv} gives
\[
\int_\O B\left(\beta(\bu)(\x,s)\right)\d\x
\le\liminf_{l,k\to\infty}\int_\O B\left(\frac{\beta(\bu)(\x,s_l)+\beta(\bu)(\x,t_k)}{2}\right)\d\x.
\]
Taking the $\limsup$ as $l,k\to\infty$ of \eqref{contnu1} and using
the continuity of $t\mapsto \int_\O B(\beta(\bu)(\x,t))\d\x$ thus shows
that 
\be\label{contnu2}
||\nu(\bu(\cdot,s_l))-\nu(\bu(\cdot,t_k))||_{L^2(\O)}\to 0\quad\mbox{ as $l,k\to\infty$}.
\ee

The existence of an a.e. representative of $\nu(\bu(\cdot,\cdot))$ which is
continuous $[0,T]\mapsto L^2(\O)$ is a direct consequence of this convergence.
Let $s\in [0,T]$ and $(s_l)_{l\in\N}\subset \mathcal T$ that converges to $s$.
Applied with $t_k=s_k$, \eqref{contnu2} shows that $(\nu(\bu(\cdot,s_l)))_{l\in\N}$ is a Cauchy sequence in $L^2(\O)$ and therefore that $\lim_{l\to\infty}
\nu(\bu(\cdot,s_l))$ exists in $L^2(\O)$. Moreover, \eqref{contnu2} shows that this
limit, that we denote by $\nu(\bu)(\cdot,s)$, does not depend on the sequence
in $\mathcal T$ that converges to $s$. Whenever $s\in\mathcal T$, the choice
$t_k=s$ in \eqref{contnu2} shows that $\nu(\bu)(\cdot,s)=\nu(\bu(\cdot,s))$ a.e. on $\O$,
and $\nu(\bu)(\cdot,\cdot)$ is therefore equal to $\nu(\bu(\cdot,\cdot))$ a.e. on $\O\times (0,T)$.

It remains to establish that $\nu(\bu)$ thus defined is continuous $[0,T]\mapsto L^2(\O)$.
For any $(\tau_r)_{r\in\N}\subset [0,T]$
that converges to $\tau\in [0,T]$, we can pick $s_r\in \mathcal T\cap (\tau_r-\frac{1}{r},\tau_r+\frac{1}{r})$
and $t_r\in \mathcal T\cap (\tau-\frac{1}{r},\tau+\frac{1}{r})$ such that
\[
||\nu(\bu)(\cdot,\tau_r)-\nu(\bu(\cdot,s_r))||_{L^2(\O)}
\le \frac{1}{r}\,,\quad
||\nu(\bu)(\cdot,\tau)-\nu(\bu(\cdot,t_r))||_{L^2(\O)}\le \frac{1}{r}.
\]
We therefore have
\[
||\nu(\bu)(\cdot,\tau_r)-\nu(\bu)(\cdot,\tau)||_{L^2(\O)}
\le
\frac{2}{r}+||\nu(\bu(\cdot,s_r))-\nu(\bu(\cdot,t_r))||_{L^2(\O)}.
\]
This proves by \eqref{contnu2} with $l=k=r$ that $\nu(\bu)(\cdot,\tau_r)
\to \nu(\bu)(\cdot,\tau)$ in $L^2(\O)$ as $r\to\infty$, and the proof is complete. \end{proof}

\section{Proof of the convergence theorems}\label{sec:proof}

\subsection{Estimates on the approximate solution}

As usual in the study of numerical methods for PDE with strong non-linearities
or without regularity assumptions on the data, everything starts with \emph{a priori} estimates.

\begin{lemma}[$L^\infty(0,T;L^2(\O))$ estimate and discrete $L^p(0,T;W^{1,p}_0(\O))$ estimate]
\label{estimldlp}
Under Assumptions \eqref{hypgnl}, let $\disc$ be a space-time gradient discretisation in the sense of 
Definition \ref{def-stcons}. Let $u$ be a solution to Scheme \eqref{ellgenfdiscllt}.

Then, for any $T_0\in (0,T]$, denoting by $k=1,\ldots,N$ the index such that $T_0\in (t^{(k-1)},t^{(k)}]$
we have
\begin{multline}\label{app:ipp}
\int_\O B(\Pi_\disc \beta(u)(\x,T_0))\d\x+\int_0^{T_0}\int_\O \bfa(\x,\Pi_\disc\nu(u)(\x,t),\nabla_\disc \zeta(u)(\x,t))\cdot\nabla_\disc \zeta(u)(\x,t)\d\x\d t\\
\le \int_\O B(\Pi_\disc \beta(\interp_\disc u_{\rm ini})(\x))\d\x
+\int_0^{t^{(k)}}\int_\O f(\x,t)\Pi_\disc \zeta(u)(\x,t)\d\x\d t.
\end{multline}
Consequently, there exists $\ctel{estimld}>0$ only depending on $p$, $L_\beta$, $C_P\ge C_\disc$
(see Definition \ref{def-coer}), $C_{\rm ini} \ge \Vert \Pi_\disc\interp_\disc u_{\rm ini}\Vert_{L^2(\O)}$, $f$, $\underline a$ and the constants $K_0$, $K_1$ and $K_2$ in \eqref{growthB} such that
\be\begin{array}{c}
\Vert\Pidisc B( \beta(u))\Vert_{L^\infty(0,T;L^1(\O))} \le \cter{estimld}\,,\;
\Vert\nabla_{\disc} \zeta(u) \Vert_{L^p(\O\times(0,T))^d} \le  \cter{estimld}
\\[0.5em]
\mbox{ and }
\Vert\Pidisc  \beta(u)\Vert_{L^\infty(0,T;L^2(\O))} \le \cter{estimld}.
\end{array}
\label{estimhund}\ee
\end{lemma}

\begin{proof}
By using \eqref{prop:pcr} and \eqref{sousdiffB} we notice that for any $n=0,\ldots,N-1$ and
any $t\in (t^{(n)},t^{(n+1)}]$
\begin{eqnarray*}
\Pi_\disc\delta_\disc \beta(u)(t) \Pi_\disc \zeta(u^{(n+1)}) 
&=&\frac{1}{\dt^{(n+\half)}}\left(\beta(\Pi_\disc u^{(n+1)})-\beta(u^{(n)})\right)
\zeta(\Pi_{\disc}u^{(n+1)})\\
&\ge&\frac{1}{\dt^{(n+\half)}}\left( B(\Pi_\disc\beta(u^{(n+1)})) -  B(\Pi_\disc\beta(u^{(n)}))\right).
\end{eqnarray*}
Hence, with $v = (\zeta(u^{(1)}),\ldots,\zeta(u^{(k)}),0,\ldots,0)
\subset X_{\disc,0}$ in \eqref{ellgenfdiscllt} we find
\begin{multline}\label{app:ippini}
\int_\O B(\Pi_\disc \beta(u)(\x,t^{(k)}))\d\x+ \int_0^{t^{(k)}}\int_\O \bfa(\x,\Pi_\disc\nu(u)(\x,t),
\nabla_\disc \zeta(u)(\x,t))\cdot\nabla_\disc \zeta(u)(\x,t)\d\x\d t\\
\le \int_\O B(\Pi_\disc \beta(u^{(0)})(\x))\d\x
+\int_0^{t^{(k)}} f(\x,t)\Pi_\disc \zeta(u)(\x,t)\d\x\d t.
\end{multline}
Equation \eqref{app:ipp} is a straightforward consequence of this estimate,
of the relation $\beta(u)(\cdot,T_0)=\beta(u)(\cdot,t^{(k)})$ (see \eqref{def-stfunctions})
and of the fact that the integrand involving $\bfa$ is nonnegative on $[T_0,t^{(k)}]$.

By using Young's inequality $ab\le \frac{1}{p}a^p+\frac{1}{p'}b^{p'}$, we can write
\begin{multline*}
\int_0^{t^{(k)}}\int_\O f(\x,t)\Pi_\disc\zeta(u)(\x,t)\d\x\d t
\\
\le \frac {2^{1/(p-1)} C_\disc^{p'}} { (p \underline a)^{1/(p-1)}\ p'}
\Vert f\Vert_{L^{p'}(\O\times(0,t^{(k)}))}^{p'} + \frac { \underline a} {2 C_\disc^{p}}\Vert\Pidisc \zeta(u)\Vert_{L^p(\O\times(0,t^{(k)}))}^{p}
\end{multline*}
and the first two estimates in \eqref{estimhund} therefore follow from \eqref{app:ippini},
\eqref{growthB}, the coercivity assumption \eqref{hypanl2} on $\bfa$ and the definition \ref{def-coer} of $C_\disc$. 
The estimate on $\Pi_\disc\beta(u)=\beta(\Pi_\disc u)$ in $L^\infty(0,T;L^2(\O))$ is a consequence
of the estimate on $B(\beta(\Pi_\disc u))$ in $L^\infty(0,T;L^1(\O))$ and
of \eqref{growthB}. \end{proof}

\begin{corollary}[Existence of a solution to the gradient scheme]\label{cor:exist}
Under Assumptions \eqref{hypgnl}, if $\disc$ is a gradient discretisation in the sense
of Definition \ref{def-stcons} then there exists at least a solution to the 
gradient scheme \eqref{ellgenfdiscllt}.
\end{corollary}

\begin{proof}
We endow $E=\{(u^{(n)})_{n=1,\ldots,N}\,:\,u^{(n)}\in X_{\disc,0}\mbox{ for all $n$}\}$
with the dot product ``$\cdot$'' 
coming from the degrees of freedom $I$ (see Remark \ref{rem:nlfxd}), and
we denote by $|\cdot|$ the corresponding norm.
Let $T:E\mapsto E$ be such that, for all $u,v\in E$,
\[
T(u)\cdot v = \int_0^T\int_\O \left[\Pi_\disc \delta_{\disc}\beta(u)(\x,t) \Pi_\disc v(\x,t) + \bfa(\x, \Pi_\disc \nu(u)(\x,t),\grad_\disc \zeta( u)(\x,t))\cdot\grad_\disc v(\x,t)\right]
\d\x\d t,
\]
where $\delta_\disc^{(\frac{1}{2})}\beta(u)$ is defined by setting $u^{(0)}=\interp_\disc u_{\rm ini}$.
Set $f_E\in E$ such that, for all $v\in E$,
$f_E\cdot v = \int_0^T\int_\O f(\x,t)  \Pi_\disc v(\x,t) \d\x\d t$.
A solution to \eqref{ellgenfdiscllt} is an element $u\in E$ such that $T(u)=f_E$. 
The continuity and growth properties of $\beta$, $\zeta$ and $\bfa$ clearly show that $T$
is continuous $E\mapsto E$, so we can prove that $T(u)=f_E$ has
has a solution by establishing that, for $R$ large
enough, $d(T,B(R),f_E)\not=0$ where $d$ is the Brouwer topological degree \cite{deimling}
and $B(R)$ is the open ball of radius $R$ in $E$.

Following the reasoning used to prove \eqref{app:ipp}, 
the coercivity property \eqref{hypanl2} on $\bfa$ and the equivalence of all norms
on $E$ give $\ctel{Cnotdepu}$ and $\ctel{Cnotdepu2}$ not depending on $u\in E$ such that
\[
T(u)\cdot\zeta(u) \ge \underline{a}||\nabla_\disc \zeta(u)||_{L^p(\O)^d}^p - ||B(\Pi_\disc \beta(\interp_\disc
u_{\rm ini}))||_{L^1(\O)}
\ge \cter{Cnotdepu} |u|^p - \cter{Cnotdepu2}.
\]
From the choice of the dot product on $E$ and Assumption \eqref{hypzeta} on $\zeta$,
we have $|\zeta(v)|\le L_\zeta|v|$ and $\zeta(v)\cdot v\ge \ctel{ceq1}|v|^2 - \ctel{ceq2}$,
with $\cter{ceq1}>0$ and $\cter{ceq2}$ not depending on $v\in E$.
Let us consider the homotopy $h(\rho,u)=\rho T(u)+(1-\rho)u$ between $T$ and ${\rm Id}$,
and assume that $u$ is a solution to $h(\rho,u)=f_E$ for some $\rho\in [0,1]$.
We have if $|u|\ge 1$
\begin{multline*}
|f_E|L_\zeta |u|\ge f_E\cdot\zeta(u) = \rho T(u)\cdot \zeta(u) + (1-\rho) u\cdot \zeta(u)\\
\ge \rho \cter{Cnotdepu}|u|^p - \rho \cter{Cnotdepu2} + (1-\rho)\cter{ceq1}|u|^2 - (1-\rho)\cter{ceq2}
\ge \min(\cter{Cnotdepu},\cter{ceq1})|u|^{\min(p,2)} - \cter{Cnotdepu2}-\cter{ceq2}.
\end{multline*}
Hence, if we select $R>1$ such that $|f_E|L_\zeta R<
\min(\cter{Cnotdepu},\cter{ceq1})R^{\min(p,2)}- \cter{Cnotdepu2}-\cter{ceq2}$, which is
possible since $\min(p,2)>1$, no solution to $h(\rho,u)=f_E$ can lie on $\partial B(R)$.
The invariance by homotopy of the 
topological degree then gives $d(T,B(R),f_E)=d({\rm Id},B(R),f_E)$, and this last degree is
equal to $1$ if we select $R$ such that $f_E\in B(R)$. The proof is complete. \end{proof}

\begin{lemma}[Estimate on the dual semi-norm of the discrete time derivative]\label{estl1l2snstdt}~

Under Assumptions \eqref{hypgnl}, let $\disc$ be a space-time gradient discretisation in the sense of 
Definition \ref{def-stcons}. Let $u$ be a solution to Scheme \eqref{ellgenfdiscllt}. 
Then there exists  $\ctel{estimdt}$ only depending on $p$, $L_\beta$, $C_P\ge C_{\disc}$,
$C_{\rm ini}\ge \Vert \Pi_\disc I_\disc u_{\rm ini}\Vert_{L^2(\O)}$, $f$,
$\underline a$, $\mu$, $\overline{a}$, $T$ and the constants $K_0$, $K_1$ and $K_2$
in \eqref{growthB} such that
\be
\int_0^T \vert \delta_{\disc}  \beta(u)(t)\vert_{\star,\disc}^{p'}\d t\le \cter{estimdt}.
\label{nsestimlquatier}\ee
\end{lemma}
\begin{proof}
Let us take a generic $v=(v^{(n)})_{n=1,\ldots,N}\subset X_{\disc,0}$ as a test function in Scheme \eqref{ellgenfdiscllt}. We have, thanks to Assumption \eqref{hypanl4} on $\bfa$,
\begin{multline*}
\int_0^T  \int_\O \Pi_\disc \delta_{\disc}\beta(u)(\x,t) \Pi_\disc v(\x,t)\d\x\d t
\le \int_0^T \int_\O(\overline{a}(\x) + \mu|\grad_\disc \zeta(u)(\x,t)|^{p-1})
\vert\grad_\disc v(\x,t)\vert \d\x\d t \\
+ \int_0^T\int_\O f(\x,t)  \Pi_\disc v(\x,t) \d\x\d t.
\end{multline*}
Using H\"older's inequality, Definition \ref{def-coer} and Estimates \eqref{estimhund},
this leads to the existence of $\ctel{cc1}>0$ only depending on
$p$, $L_\beta$, $C_P$, $C_{\rm ini}$, $f$, $\underline{a}$, $\overline{a}$, $\mu$ and $K_0$, $K_1$ and $K_2$ such that
\[
\int_0^T \int_\O \Pi_\disc \delta_{\disc} \beta(u)(\x,t) \Pi_\disc v(\x,t)\d\x\d t
\le\cter{cc1}\Vert \grad_\disc v\Vert_{L^p(0,T;L^p(\O))^d}.
\]
The proof of \eqref{nsestimlquatier} is completed by selecting
$v=(|\delta_\disc^{(n+\half)} \beta(u)|_{\star,\disc}^{p'-1}z^{(n)})_{n=1,\ldots,N}$ with
$(z^{(n)})_{n=1,\ldots,N}\subset X_{\disc,0}$ such that, for any $n=1,\ldots,N$,
$z^{(n)}$ realises the supremum in \eqref{nstnormemel} with $w=\delta_\disc^{(n+\half)}\beta(u)$. \end{proof}

\begin{lemma}[Estimate on the time translates of $\nu(u)$]\label{esttimetransl}~

Under Assumptions \eqref{hypgnl}, let $\disc$ be a space-time gradient discretisation in the sense of 
Definition \ref{def-stcons}. Let $u$ be a solution to Scheme \eqref{ellgenfdiscllt}. 
Then there exists  $\ctel{parnlestimtt}$ only depending on $p$, $L_\beta$, $L_\zeta$, $C_P\ge C_{\disc}$,
$C_{\rm ini}\ge \Vert \Pi_\disc I_\disc u_{\rm ini}\Vert_{L^2(\O)}$, $f$,
$\underline a$, $\mu$, $\overline{a}$, $T$ and $K_0$, $K_1$ and $K_2$ in \eqref{growthB} such that
\be
\norm{\Pi_\disc \nu(u)(\cdot,\cdot+\tau)-\Pi_\disc \nu(u)(\cdot,\cdot)}{L^2(\O\times(0,T-\tau))}^2 
\le\cter{parnlestimtt}(\tau + \dt), \quad\forall \tau\in(0,T).
\label{parnltrt}
\ee
\end{lemma}

\begin{proof} 
Let $\tau \in (0,T)$. Thanks to \eqref{relchizetabeta},
we can write
\be
\int_{\O\times (0,T-\tau)} \Bigl(\Pi_\disc \nu(u)(\x,t+\tau)-\Pi_\disc \nu(u)(\x,t)\Bigr)^2 \d\x \d t
\le L_\beta  L_\zeta \int_0^{T-\tau} A(t) \d t,
\label{parnltrt1}\ee
where
$$
A(t) =  \int_\O  \Bigl(\Pi_\disc \zeta(u)(\x,t+\tau)-\Pi_\disc\zeta(u)(\x,t)\Bigr)
\Bigl(\Pi_\disc \beta(u)(\x,t+\tau)-\Pi_\disc \beta(u)(\x,t)\Bigr) \d\x.
$$
For $s\in (0,T)$, we define $n(s) \in \{0,\ldots,N-1\}$ such that $t^{(n(s))} < s \le t^{(n(s)+1)}$.
Taking $t\in(0,T-\tau)$, we may write
\[
A(t) =  \int_\O  \Bigl(\Pi_\disc \zeta(u^{(n(t+\tau)+1)})(\x)-\Pi_\disc\zeta(u^{(n(t)+1)})(\x)\Bigr)\Bigl(\sum_{n = n(t)+1}^{n(t+\tau)}\dt^{(n+\half)}\Pi_\disc\delta_{\disc}^{(n+\half)} \beta(u)(\x)\Bigr) \d\x.
\]
We then use the definition \eqref{nstnormemel} of the discrete dual semi-norm to infer
\begin{equation}
A(t) \le  \sum_{n = n(t)+1}^{n(t+\tau)}\dt^{(n+\half)}
\left|\left|\nabla_\disc \left[ \zeta(u^{(n(t+\tau)+1)})-\zeta(u^{(n(t)+1)})\right]\right|\right|_{L^p(\O)^d}
|\delta_{\disc}^{(n+\half)} \beta(u)|_{\star,\disc}.
\label{parnltrt4a}\end{equation}
We apply the triangular inequality on the first norm in this right-hand side,
Young's inequality and we integrate over $t\in (0,T-\tau)$ to get
\begin{equation}\label{parnltrt-new}
\int_0^{T-\tau}A(t)\d t\le \mathcal{A}_\tau + \mathcal{A}_0+\mathcal{B}
\end{equation}
with, for $s=0$ or $s=\tau$,
\begin{equation}\label{parnltrt-new2}
\mathcal{A}_s=\frac{1}{p}\int_0^{T-\tau} \sum_{n = n(t)+1}^{n(t+\tau)}\dt^{(n+\half)}
||\nabla_\disc \zeta(u^{(n(t+s)+1)})||_{L^p(\O)^d}^p\d t\le \frac{\cter{estimld}^p}{p}(\tau+\dt)
\end{equation}
and
\begin{equation}\label{parnltrt-new1}
\mathcal{B}= \frac{2}{p'}\int_0^{T-\tau}  \sum_{n = n(t)+1}^{n(t+\tau)}\dt^{(n+\half)}
|\delta_{\disc}^{(n+\half)} \beta(u)|_{\star,\disc}^{p'}\d t\le \frac{2\cter{estimdt}}{p'}\tau.
\end{equation}
In \eqref{parnltrt-new2}, the quantity $\mathcal{A}_s$ has been estimated by using \eqref{esttt2} in Lemma \ref{estttt} and
the estimate on $\nabla_\disc \zeta(u)$ in \eqref{estimhund}. In \eqref{parnltrt-new1}, $\mathcal{B}$ has been estimated by applying \eqref{esttt1} in Lemma \ref{estttt} and by
using the bound \eqref{nsestimlquatier} on $\int_0^T \vert \delta_{\disc}  \beta(u)(t)\vert_{\star,\disc}^{p'}\d t$.
The proof is completed by gathering \eqref{parnltrt1}, \eqref{parnltrt-new},
\eqref{parnltrt-new2} and \eqref{parnltrt-new1}.
\end{proof}

\subsection{Proof of Theorem \ref{th:weakconv}}

{\bf Step 1} Application of compactness results.

Thanks to Theorem \ref{unifweakGDcomp} and Estimates \eqref{estimhund} and \eqref{nsestimlquatier}, we first extract a subsequence such that $(\Pi_{\disc_m}\beta(u_m))_{m\in\N}$ converges weakly in $L^2(\O)$
uniformly on $[0,T]$ (in the sense of Definition \ref{defweakunifconv}) to some function 
$\overline{\beta}\in C([0,T];L^2(\O)\weak)$ which satisfies $\overline{\beta}(\cdot,0) = \beta(u_{\rm ini})$ in $L^2(\O)$. Using again Estimates \eqref{estimhund} and applying Lemma \ref{lem:reglim},
we extract a further subsequence such that, for some $\overline{\zeta}\in L^p(0,T;W^{1,p}_0(\O))$,
$\Pi_{\disc_m}\zeta(u_m)\to\overline{\zeta}$ weakly in $L^p(\O\times (0,T))$ and $\nabla_{\disc_m}\zeta(u_m)
\to \nabla \overline{\zeta}$ weakly in $L^p(\O\times(0,T))^d$.
Estimates \eqref{estimhund}, Definition \ref{def-coer} and the growth
assumption \eqref{hypzeta} on $\zeta$ show that $(\Pi_{\disc_m}u_m)_{m\in\N}$
is bounded in $L^p(\O\times (0,T))$ and we can therefore assume, up to a subsequence,
that it converges weakly to some $\bu$ in this space.

We then prove, by means of the Kolmogorov theorem, that $(\Pi_{\disc_m}\nu(u_m))_{m\in\N}$
is relatively compact in $L^1(\O\times(0,T))$.
We first remark that $|\nu(a) - \nu(b)|\le L_\beta |\zeta(a) - \zeta(b)|$,
which implies, using Estimate \eqref{estimhund} and Definition \ref{def-comp} with $v=\zeta(u_m)$,
\begin{equation}\label{spacetransnu}
||\Pi_{\disc_m} \nu(u_m)(\cdot+\bxi,\cdot)-\Pi_{\disc_m} \nu(u_m)(\cdot,\cdot)||_{L^p(\R^d\times(0,T))} \le
L_\beta \cter{estimld}T_{\disc_m}(\bxi) 
\end{equation}
where $\Pi_{\disc_m} \nu(u_m)$ has been extended by $0$ outside $\O$,
and $\lim_{\bxi\to 0}\sup_{m\in\N}T_{\disc_m}(\bxi)=0$. This takes care of the
space translates. Let us now turn to the time translates.
Invoking Lemma \ref{esttimetransl} and, to control the time translates at both ends of $[0,T]$,
the fact that $\Pi_{\disc_m}\beta(u_m)$ -- and therefore also $\Pi_{\disc_m} \nu(u_m)$ since $|\nu|\le L_{\zeta}|\beta|$ -- remains bounded in $L^\infty(0,T; L^2(\O))$, we can write for any $M\in\N$
\begin{multline}\label{timetransnew}
\sup_{m\in\N}||\Pi_{\disc_m}\nu(u_m)(\cdot,\cdot+\tau)-\Pi_{\disc_m}\nu(u_m)(\cdot,\cdot)||_{L^2(\O\times(0,T))}^2\\
\le\max\left(
\max_{m\le M}||\Pi_{\disc_m}\nu(u_m)(\cdot,\cdot+\tau)-\Pi_{\disc_m}\nu(u_m)(\cdot,\cdot)||_{L^2(\O\times(0,T))}^2;
\cter{Cautre}\left(\tau+\sup_{m>M}\dt_m\right)\right),
\end{multline}
where $\ctel{Cautre}$ does not depend on $m$ or $\tau$, and the functions have been extended
by $0$ outside $(0,T)$.
Since each $||\Pi_{\disc_m}\nu(u_m)(\cdot,\cdot+\tau)-\Pi_{\disc_m}\nu(u_m)||_{L^2(\O\times(0,T))}^2$ tends to $0$ as $\tau\to 0$
and since $\dt_m\to 0$ as $m\to\infty$, taking in that order the limsup as $\tau\to 0$ and the limit as $M\to\infty$
of \eqref{timetransnew} shows that the left-hand side of this inequality tends to $0$
as $\tau \to 0$, as required.
Hence, Kolmogorov's theorem shows that, up to extraction of another subsequence,
$\Pi_{\disc_m}\nu(u_m)\to\overline{\nu}$ in $L^1(\O\times(0,T))$.

Let us now identify these limits $\overline{\beta}$, $\overline{\zeta}$ and $\overline{\nu}$.
Under the first case in the structural hypothesis \eqref{hyp:struct}, we have
$\beta={\rm Id}$, and therefore $\overline{\beta}=\bu=\beta(\bu)$
and $\nu=\zeta$. The strong convergence of $\Pi_{\disc_m}\nu(u_m)=
\Pi_{\disc_m}\zeta(u_m)$ to $\overline{\nu}=\overline{\zeta}$ allows us to apply
Lemma \ref{lem-minty} to see that $\overline{\zeta}=\zeta(\bu)$ and $\overline{\nu}=\nu(\bu)$.
Exchanging the roles of $\beta$ and $\zeta$, we see that $\overline{\beta}=\beta(\bu)$,
$\overline{\zeta}=\zeta(\bu)$ and $\overline{\nu}=\nu(\bu)$ still hold in the second case of
\eqref{hyp:struct}. We notice that this is the only place where we use this structural assumption \eqref{hyp:struct}
on $\beta,\zeta$.

Using the growth assumption \eqref{hypanl4} on $\bfa$ and Estimates \eqref{estimhund},
upon extraction of another subsequence we can also assume that
$\bfa\left(\cdot, \Pi_{\disc_m} \nu(u_m),\grad_{\disc_m} \zeta(u_m)\right)$
has a weak limit in $L^{p'}(\O\times(0,T))^d$, which we denote
by $\bfA$.

Finally, for any $T_0\in [0,T]$, since $\Pi_{\disc_m}\beta(u_m(\cdot,T_0))\to
\beta(\bu)(\cdot,T_0)$ weakly in $L^2(\O)$, Lemma \ref{lem:Zconv} gives
\be
\int_\O B(\beta(\bu)(\x,T_0))\d\x \le \liminf_{m\to\infty}\int_\O B(\beta(\Pi_{\disc_m} u_m)(\x,T_0))\d\x.
\label{minorliminf}\ee
With \eqref{estimhund}, this shows that $B(\beta(\bu))\in L^\infty(0,T;L^1(\O))$.

\medskip

{\bf Step 2} Passing to the limit in the scheme.

We drop the indices $m$ for legibility reasons.
Let $\varphi\in C^1_c(-\infty, T)$ and let $w\in W^{1,p}_0(\O)\cap L^2(\O)$. 
We introduce $v=(\varphi(t^{(n-1)})P_{\disc} w)_{n=1,\ldots,N}$ as a test function in \eqref{ellgenfdiscllt},
with $P_{\disc}$ defined by \eqref{def-PD}.
We get $T_1^{(m)} + T_2^{(m)} = T_3^{(m)}$ with 
\[
T_1^{(m)} = \sum_{n=0}^{N-1} \varphi(t^{(n)})\dt^{(n+\half)}\int_\O \Pi_\disc \delta_{\disc}^{(n+\half)} \beta(u)(\x) \Pi_\disc P_{\disc}  w(\x)\d\x,
\]
\[
T_2^{(m)} = \sum_{n=0}^{N-1} \varphi(t^{(n)})\dt^{(n+\half)} \int_\O \bfa\left(\x, \Pi_\disc \nu(u^{(n+1)}),\grad_\disc \zeta( u^{(n+1)})(\x)\right)\cdot\grad_\disc P_{\disc}   w(\x)\d\x,
\]
and
\[
T_3^{(m)} =  \sum_{n=0}^{N-1} \varphi(t^{(n)}) \int_{t^{(n)}}^{t^{(n+1)}}\int_\O f(\x,t)
\Pi_\disc P_{\disc}w(\x) \d\x\d t.
\]
Using discrete integrate-by-parts to transform the terms
$\varphi(t^{(n)}) (\Pi_\disc \beta(u^{(n+1)})-\Pi_\disc \beta(u^{(n)}))$
appearing in $T_1^{(m)}$ into $(\varphi(t^{(n)})-\varphi(t^{(n+1)}))\Pi_\disc \beta(u^{(n+1)})$,
we have
\[
T_1^{(m)} =  \dsp- \int_0^T\varphi'(t) \int_\O \Pi_\disc \beta(u)(\x,t)
\Pi_\disc P_{\disc} w(\x)\d\x\d t
- \varphi(0)\int_\O \Pi_\disc \beta(u^{(0)})(\x) \Pi_\disc P_{\disc}w(\x)\d\x.
\]
Setting $\varphi_\disc(t) = \varphi(t^{(n)})$ for $t\in (t^{(n)},t^{(n+1)})$, we have
\[\ba\dsp
T_2^{(m)} = \int_0^T\varphi_\disc(t) \int_\O  \bfa\left(\x, \Pi_\disc \nu(u)(\x,t),\grad_\disc \zeta(u)(\x,t)\right)\cdot\grad_\disc P_{\disc}w(\x)\d\x\d t
\\ \dsp
T_3^{(m)} =  \int_0^T\varphi_\disc(t)\int_\O f(\x,t)  \Pi_\disc P_{\disc}w(\x) \d\x\d t.\ea
\]
Since $\varphi_\disc\to \varphi$ uniformly on $[0,T]$, $\Pi_\disc P_\disc w\to w$
in $L^p(\O)\cap L^2(\O)$ and $\nabla_\disc P_\disc w\to \nabla w$ in $L^p(\O)^d$,
we may let $m\to\infty$ in $T_1^{(m)} + T_2^{(m)} = T_3^{(m)}$
to see that $\ubarre$ satisfies
\be\left\{\ba
\ubarre \in L^p(\O\times(0,T))\,,\;\zeta(\bu)\in L^p(0,T;W^{1,p}_0(\O))\,,\;
B(\beta(\bu))\in L^\infty(0,T;L^1(\O)),\\
\beta(\bu)\in C([0,T];L^2(\O)\weak)\,,\; \beta(\bu)(\cdot,0)=\beta(u_{\rm ini})\,,\\
\dsp - \int_0^T \varphi'(t)\int_\O \beta(\ubarre(\x,t)) w(\x)\d\x\d t -\varphi(0)\int_\O \beta(u_{\rm ini}(\x))w(\x)\d\x \\[1em]
\dsp 
+\int_0^T \varphi(t)\int_\O \bfA(\x,t)\cdot\grad w(\x)
\d\x\d t = \int_0^T \varphi(t)\int_\O f(\x,t) w(\x) \d\x\d t,\\
\qquad\forall w \in W^{1,p}_0(\O)\cap L^2(\O),\ \forall \varphi\in C^\infty_c(-\infty,T).
\ea\right.\label{ellgenflltdu}\ee
Note that the regularity properties on $\bu$, $\zeta(\bu)$, $\beta(\bu)$ and $B(\beta(\bu))$
have been established in Step~1.
Linear combinations of this relation show that \eqref{ellgenflltdu} also holds with $\varphi(t)w(\x)$ replaced
by a tensorial functions in $C^\infty_c(\O\times (0,T))$.
This proves that $\partial_t\beta(\ubarre)\in L^{p'}(0,T;W^{-1,p'}(\O))$
(see Remark \ref{rem-defdt_u}).
Using the density of tensorial functions in $L^p(0,T;W^{1,p}_0(\O))$ \cite{poly}, we then see that
$\ubarre$ satisfies
\be\ba
\dsp \int_0^T \langle \partial_t\beta(\ubarre)(\cdot,t), \vbarre(\cdot,t)\rangle_{W^{-1,p'},W^{1,p}_0}\d t  \\ 
\dsp  +  
\int_0^T \int_\O \bfA(\x,t)\cdot\grad \vbarre(\x,t)
\d\x\d t = \int_0^T \int_\O f(\x,t) \vbarre(\x,t) \d\x\d t\,,\quad
\forall \vbarre \in L^p(0,T;W^{1,p}_0(\O)).
\ea\label{ellgenflltdd}\ee

\medskip

{\bf Step 3} Proof that $ \ubarre$ is a solution to  \eqref{ellgenfllt}.

\medskip

It only remains to show that
\be
\bfA(\x,t) = \bfa(\x,\nu(\ubarre)(\x,t),\grad\zeta(\ubarre)(\x,t))\hbox{ for a.e. }(\x,t)\in \O\times(0,T).
\label{resumintyt}\ee
We take $T_0\in [0,T]$, write \eqref{app:ipp} with $\disc=\disc_m$ and take the
$\limsup$ as $m\to\infty$. We notice that the $t^{(k)}=:T_m$ from Lemma \ref{estimldlp}
converges to $T_0$ as $m\to\infty$. Hence, by using the convergence $\Pi_{\disc_m}
\interp_{\disc_m}u_{\rm ini}\to u_{\rm ini}$ in $L^2(\O)$ (consistency of $(\disc_m)_{m\in\N}$),
and the continuity and quadratic growth of $B\circ\beta$ (upper bound in
\eqref{growthB}), we obtain
\begin{multline}
\limsup_{m\to\infty}\int_0^{T_0}\int_\O \bfa(\x,\Pi_{\disc_m}\nu(u_m)(\x,t),\nabla_{\disc_m} \zeta(u_m)(\x,t))\cdot\nabla_{\disc_m} \zeta(u_m)(\x,t)\d\x\d t\\
\le \int_\O B( \beta(u_{\rm ini})(\x))\d\x
+\int_0^{T_0}\int_\O f(\x,t) \zeta(\bu)(\x,t)\d\x\d t 
-\liminf_{m\to\infty}\int_\O B(\beta(\Pi_{\disc_m} u_m)(\x,T_0))\d\x.
\label{ineqpasslim}\end{multline}
We take $\overline{v} = \zeta(\bu)\mathbf{1}_{[0,T_0]}$ in \eqref{ellgenflltdd} and apply Lemma \ref{ricipp} to get
\[
 \ba
\dsp \int_\O B(\beta(\bu)(\x,T_0))\d\x  -  \int_\O B(\beta(\bu)(\x,0))\d\x  \\ 
\dsp  +  
\int_0^{T_0} \int_\O \bfA(\x,t)\cdot\grad \zeta(\bu)(\x,t)
\d\x\d t = \int_0^{T_0} \int_\O f(\x,t) \zeta(\bu)(\x,t) \d\x\d t.
\ea
\]
This relation, combined with \eqref{ineqpasslim} and using \eqref{minorliminf}, shows that
\begin{multline}
 \limsup_{m\to\infty}\int_0^{T_0}\int_\O \bfa(\x,\Pi_{\disc_m}\nu(u_m)(\x,t),\nabla_{\disc_m} \zeta(u_m)(\x,t))\cdot\nabla_{\disc_m} \zeta(u_m)(\x,t)\d\x\d t
\\ \le \int_0^{T_0} \int_\O \bfA(\x,t)\cdot\grad \zeta(\bu)(\x,t)
\d\x\d t.
\label{ineqpminty}\end{multline}

It is now possible to apply Minty's trick. Consider for $\bG\in L^p(\O\times(0,T))^d$ the following
relation, stemming from the monotony \eqref{hypanl3} of $\bfa$:
\begin{equation}
\int_0^{T_0} \int_\O \left[\bfa(\cdot,\Pi_{\disc_m} \nu(u_m),\grad_{\disc_m} \zeta(u_m)) - 
\bfa(\cdot,\Pi_{\disc_m} \nu(u_m),\bG)\right] 
\cdot\left[\grad_{\disc_m} \zeta(u_m)-\bG\right]\d\x\d t \ge 0.
\label{ineqmonot}\end{equation}
By strong convergence of
$\Pi_{\disc_m}\nu(u_m)$ to $\nu(\bu)$ in $L^1(\O\times(0,T))$
and Assumptions \eqref{hypanl1}, \eqref{hypanl4} on $\bfa$,
we see that $\bfa(\cdot,\Pi_{\disc_m}\nu(u_m),\bG)\to \bfa(\cdot,\nu(\bu),\bG)$
strongly in $L^{p'}(\O\times (0,T))^d$. The development of \eqref{ineqmonot}
gives a sum of four terms, the first one being the integral in the left-hand side of \eqref{ineqpminty}
and the other three being integrals of products of weakly and strongly converging sequences.
We can thus take the $\limsup$ of \eqref{ineqmonot} with $T_0=T$ to find
\[
\int_0^{T} \int_\O \left[\bfA(\x,t) - \bfa(\x,\nu(\ubarre)(\x,t),\bG(\x,t))\right]
\cdot\left[\grad\zeta(\ubarre)(\x,t)-\bG(\x,t)\right]\d\x\d t \ge 0.
\]
Application of Minty's method \cite{min-63-mon} (i.e. taking $\bG=\nabla\zeta(\bu)+r\bvarphi$ for
$\bvarphi\in L^p(\O\times(0,T))^d$ and letting $r\to 0$)
 then shows
that \eqref{resumintyt} holds and concludes the proof that $\ubarre$ satisfies \eqref{ellgenfllt}.

\subsection{Proof of Theorem \ref{th:uniftime}}

Let $T_0\in [0,T]$ and $(T_m)_{m\ge 1}$ be a sequence in $[0,T]$
that converges to $T_0$. 
By setting $T_0=T_m$ and $\bG = \nabla \zeta(\bu)$ in the developed
form of \eqref{ineqmonot}, by taking the infimum limit (thanks to the strong convergence of
$\bfa(\cdot,\Pi_{\disc_m}\nu(u_m),\nabla\zeta(\bu))$) and by using \eqref{resumintyt},
we find
\begin{multline}
 \liminf_{m\to\infty}\int_0^{T_m}\int_\O \bfa(\x,\Pi_{\disc_m}\nu(u_m)(\x,t),\nabla_{\disc_m} \zeta(u_m)(\x,t))\cdot\nabla_{\disc_m} \zeta(u_m)(\x,t)\d\x\d t
\\ \ge \int_0^{T_0} \int_\O \bfa(\x,\nu(\ubarre)(\x,t),\grad\zeta(\ubarre)(\x,t))\cdot\grad \zeta(\bu)(\x,t)
\d\x\d t.
\label{eq:limtermembetant}\end{multline}
We then write \eqref{app:ipp} with $T_m$ instead of $T_0$ and we take the
$\limsup$ as $m\to\infty$. We notice that the $t^{(k)}$
such that $T_m\in (t^{(k-1)},t^{(k)}]$ converges to $T_0$ as $m\to\infty$.
Thanks to \eqref{eq:limtermembetant} and \eqref{form:ippsol} we obtain
\be\label{sufficient}
\limsup_{m\to\infty}\int_\O B({\beta}(\Pi_{\disc_m} u_m(\x,T_m))) \d\x
\le \int_\O B(\beta(\bu)(\x,T_0)) \d \x.
\ee
By Lemma \ref{equiv-unifconv}, the uniform-in-time weak convergence
of $\beta(\Pi_{\disc_m}u_m)$ to $\beta(\bar u)$ and the continuity of
$\beta(\bar u):[0,T]\to L^2(\O)\weak$, we have $\beta(\Pi_{\disc_m}u_m)(T_m)\to
\beta(\bar u)(T_0)$ weakly in $L^2(\O)$ as $m\to\infty$. Therefore,
for any $(s_m)_{m\in\N}$ converging to $T_0$,
$\frac{1}{2}({\beta}(\Pi_{\disc_m} u_m(T_m))+\beta(\bu)(s_m))\to
\beta(\bu)(T_0)$ weakly in $L^2(\O)$ as $m\to\infty$ and
Lemma \ref{lem:Zconv} gives, by convexity of $B$,
\be
\label{limbetatildemoy}
\int_\O B(\beta(\bu)(\x,T_0)) \d \x
\le
\liminf_{m\to\infty}\int_\O B\left(\frac{{\beta}(\Pi_{\disc_m} u_m(\x,T_m))+\beta(\bu)(\x,s_m)} 2\right)
\d\x.
\ee

Property \eqref{Bunifconv} of $B$ and the two inequalities \eqref{sufficient} and \eqref{limbetatildemoy}
 allow us to conclude the proof.
Let $(s_m)_{m\in\N}$ be a sequence in $\mathcal T$ (see proof of Corollary \ref{ric:bou})
that converges to $T_0$. Then $\nu(\bu(\cdot,s_m))\to \nu(\bu)(\cdot,T_0)$ in $L^2(\O)$
as $m\to\infty$.
Using \eqref{Bunifconv}, we get
\begin{multline*}
\Vert {\nu}(\Pi_{\disc_m} u_m(\cdot,T_m)) - \nu(\bu)(\cdot,T_0)\Vert_{L^2(\O)}^2 \\
\le 2\Vert {\nu}(\Pi_{\disc_m} u_m(\cdot,T_m)) - \nu(\bu(\cdot,s_m))\Vert_{L^2(\O)}^2
+2\Vert {\nu}(\bu(\cdot,s_m)) - \nu(\bu)(\cdot,T_0)\Vert_{L^2(\O)}^2\\
\le  8 L_\beta L_\zeta\int_\O \left[B({\beta}(\Pi_{\disc_m} u_m(\x,T_m))) + B(\beta(\bu(\x,s_m)))\right]\d\x\\
-16L_\beta L_\zeta\int_\O B\left(\frac{{\beta}(\Pi_{\disc_m} u_m(\x,T_m))+\beta(\bu(\x,s_m))} 2\right)\d\x\\
+2\Vert {\nu}(\bu(\cdot,s_m)) - \nu(\bu)(\cdot,T_0)\Vert_{L^2(\O)}^2.
\end{multline*}
We then take the $\limsup$ as $m\to\infty$ of this expression.
Thanks to \eqref{sufficient} and the continuity of $t\in [0,T]\mapsto 
\int_\O B(\beta(\bu)(\x,t))\d\x\in [0,\infty)$
(see Corollary \ref{ric:bou}), the first term in the right-hand side has a finite
$\limsup$, bounded above by $16L_\beta L_\zeta \int_\O B(\beta(\bu)(\x,T_0))\d\x$.
We can therefore split the $\limsup$
of this right-hand side without risking writing $\infty-\infty$ and we get, thanks to \eqref{limbetatildemoy},
\[
\limsup_{m\to\infty}\Vert {\nu}(\Pi_{\disc_m} u_m(\cdot,T_m)) - \nu(\bu)(\cdot,T_0)\Vert_{L^2(\O)}^2 \le 0.
\]
Thus, $\nu(\Pi_{\disc_m}u_m(\cdot,T_m))\to \nu(\bu)(T_0)$ strongly in $L^2(\O)$. By Lemma
\ref{equiv-unifconv} and the continuity of $\nu(\bu):[0,T]\mapsto L^2(\O)$
stated in Corollary \ref{ric:bou}, this concludes 
the proof of the convergence of $\nu(\Pi_{\disc_m}u_m)$ to $\nu(\bu)$ in $L^\infty(0,T;L^2(\O))$.

\begin{remark}
Since $\beta(\Pi_{\disc_m}u_m)(T_m)\to \beta(\bar u)(T_0)$ weakly in $L^2(\O)$ as $m\to\infty$,
Lemma \ref{lem:Zconv} shows that 
$\int_\O B(\beta(\bu)(\x,T_0))\d\x \le \liminf_{m\to\infty}\int_\O B(\beta(\Pi_{\disc_m} u_m)(\x,T_m))\d\x$.
Combined with \eqref{sufficient}, this gives
\be\label{limbetatilde-bis}
\lim_{m\to\infty}\int_\O B({\beta}(\Pi_{\disc_m} u_m(\x,T_m))) \d\x
=\int_\O B(\beta(\bu)(\x,T_0)) \d \x.
\ee
Item 1 in Corollary \ref{ric:bou} and Lemma \ref{equiv-unifconv} therefore show that
the functions $\int_\O B(\beta(\Pi_{\disc_m}u_m(\x,\cdot)))\d\x$
converges uniformly on $[0,T]$ to $\int_\O B(\beta(\bu)(\x,\cdot))\d\x$.
\end{remark}

\subsection{Proof of Theorem \ref{th:strgrad}}

By taking the $\limsup$ as $m\to\infty$ of \eqref{app:ipp} for $u_m$ with $T_0=T$,
and by using \eqref{limbetatilde-bis} (with $T_m\equiv T$) and the continuous integration-by-parts
formula \eqref{form:ippsol}, we find
\begin{multline*}
 \limsup_{m\to\infty}\int_0^{T}\int_\O \bfa(\x,\Pi_{\disc_m}\nu(u_m)(\x,t),\nabla_{\disc_m} \zeta(u_m)(\x,t))\cdot\nabla_{\disc_m} \zeta(u_m)(\x,t)\d\x\d t\\
\le \int_0^{T_0} \int_\O \bfa(\x,\nu(\ubarre)(\x,t),\grad\zeta(\ubarre)(\x,t))\cdot\grad \zeta(\bu)(\x,t)
\d\x\d t.
\end{multline*}
Combined with \eqref{eq:limtermembetant}, this shows that
\begin{multline}\label{cvaD}
\lim_{m\to\infty}\int_0^{T}\int_\O \bfa(\x,\Pi_{\disc_m}\nu(u_m)(\x,t),\nabla_{\disc_m} \zeta(u_m)(\x,t))\cdot\nabla_{\disc_m} \zeta(u_m)(\x,t)\d\x\d t\\
= \int_0^{T_0} \int_\O \bfa(\x,\nu(\ubarre)(\x,t),\grad\zeta(\ubarre)(\x,t))\cdot\grad \zeta(\bu)(\x,t)
\d\x\d t.
\end{multline}
Let us define
\[
f_m=\left[\bfa(\x,\Pi_{\disc_m} \nu(u_m),\grad_{\disc_m} \zeta(u_m)) -
\bfa(\x,\Pi_{\disc_m} \nu(u_m)(\cdot,t),\nabla\zeta(\bu))\right]
\cdot \left[\grad_{\disc_m} \zeta(u_m)-\nabla\zeta(\bu)\right]\ge 0.
\]
By developing this expression and using \eqref{cvaD}, \eqref{resumintyt} and \eqref{conv:base},
we see that $\int_0^T\int_\O f_m(\x,t)\d\x\d t\to 0$ as $m\to\infty$.
This shows that $f_m\to 0$ in $L^1(\O\times (0,T))$ and therefore
a.e. up to a subsequence. We can then reason as in \cite{dro-12-gra},
using the strict monotony \eqref{a:strictmon} of $\bfa$,
the coercivity assumption \eqref{hypanl2} and Vitali's theorem,
to deduce that $\nabla_{\disc_m}\zeta(u_m)\to \nabla\zeta(\bu)$
strongly in $L^p(\O\times(0,T))^d$ as $m\to\infty$.

\section{Removal of the assumption ``$\beta={\rm Id}$ or $\zeta={\rm Id}$''}\label{sec:hypstruct}

We show here that all previous results are actually true
without the structural assumption \eqref{hyp:struct} -- i.e. without assuming that $\beta={\rm Id}$
or $\zeta={\rm Id}$ -- provided that the range of $p$ is slightly restricted.
The main theorem in this section is the following convergence result.

\begin{theorem}\label{th:weaknonstruct}
Under Assumptions \eqref{hypgnl}, let $(\disc_m)_{m\in\N}$ be a sequence of
space-time gradient discretisations, in the sense of Definition \ref{def-stcons},
that is coercive, consistent, limit-conforming and compact (see Section \ref{sec:propgs}).
Let, for any $m\in\N$, $u_m$ be a solution to \eqref{ellgenfdiscllt} with $\disc=\disc_m$,
provided by Theorem \ref{th:weakconv}.

If $p\ge 2$ then there exists a solution $\bu$ to \eqref{ellgenfllt} such that, up to a subsequence,
\begin{itemize}
\item the convergences in \eqref{conv:base} hold,
\item $\Pi_{\disc_m}\nu(u_m)\to \nu(\bu)$ strongly in $L^\infty(0,T;L^2(\O))$ as $m\to\infty$,
\item under the strict monotony assumption on $\bfa$ (i.e. \eqref{a:strictmon}),
as $m\to \infty$ we have $\Pi_{\disc_m}\zeta(u_m)\to \zeta(\bu)$ strongly
in $L^p(\O\times(0,T))$ and $\nabla_{\disc_m}\zeta(u_m)\to \nabla\zeta(\bu)$
strongly in $L^p(\O\times(0,T))^d$.
\end{itemize}
\end{theorem}

\begin{proof}

We only need to prove the first conclusion of the theorem, i.e. that
the convergences \eqref{conv:base} hold. Theorems \ref{th:uniftime}
and \ref{th:strgrad} then provide the last two conclusions.
The difference with respect to Theorem \ref{th:weakconv} is the removal, here,
of the structural assumption \eqref{hyp:struct}.
The only place in the proof of Theorem \ref{th:weakconv} where this
assumption was used is in Step 1, to identify the limits
$\overline{\beta}$, $\overline{\zeta}$ and $\overline{\nu}$ of $\Pi_{\disc_m}\beta(u_m)$,
$\Pi_{\disc_m}\zeta(u_m)$ and $\Pi_{\disc_m}\nu(u_m)$. We will show that
these limits can still be identified without assuming \eqref{hyp:struct}.

Set $\mu=\beta+\zeta$, let $\overline{\mu}=\overline{\beta}+\overline{\zeta}$
and fix a measurable $\bu$ such that $(\mu+\nu)(\bu)=\overline{\mu}+\overline{\nu}$.
The existence of such a $\bu$ is ensured by Assumptions \eqref{hypzeta} and \eqref{hypbeta}.
Indeed, these assumptions show that the range of $\mu+\nu$ is $\R$ and therefore
that the pseudo-reciprocal $(\mu+\nu)_r$ of $\mu+\nu$ (defined as in
\eqref{defbetar}) has domain $\R$; this allows us to set, for example,
$\bu=(\mu+\nu)_r(\overline{\mu}+\overline{\nu})$. Let us now prove that,
for such a function $\bu$, we have $\overline{\beta}=\beta(\bu)$, $\overline{\zeta}=\zeta(\bu)$
and $\overline{\nu}=\nu(\bu)$.

By using estimates \eqref{spacetransnu} and \eqref{timetransnew},
Kolmogorov's compactness theorem shows that the convergence of $\Pi_{\disc_m}\nu(u_m)$
towards $\overline{\nu}$ is actually strong in $L^2(\O\times (0,T))$ (we use
$p\ge 2$ here). Since $\mu(\Pi_{\disc_m}u_m)=\beta(\Pi_{\disc_m}u_m)+\zeta(\Pi_{\disc_m}u_m)
\to \overline{\beta} + \overline{\zeta}=\overline{\mu}$ weakly
in $L^2(\O\times (0,T))$, we can apply Lemma \ref{lem:mintylike}
with $\varphi\equiv 1$, $w_m=\Pi_{\disc_m}u_m$, $w=\bu$ and $(\mu,\nu)$ instead of $(\beta,\zeta)$
to deduce that $\overline{\nu}=\nu(\overline{u})$ and $\overline{\mu}=\mu(\overline{u})$. The second
of these relations translates into $\overline{\beta} + \overline{\zeta} = (\beta+\zeta)(\overline{u})$.

We now turn to identifying $\overline{\beta}$ and $\overline{\zeta}$.
Lemmas \ref{estimldlp} and \ref{estl1l2snstdt} show that
$\beta_m=\beta(u_m)$ and $\zeta_m=\zeta(u_m)$ satisfy the assumptions
of the discrete compensated compactness theorem \ref{thm-comp-comp} below (we use
$p\ge 2$ here).
Hence, $\Pi_{\disc_m}\beta(u_m)\Pi_{\disc_m}\zeta(u_m)\to \overline{\beta}\,\overline{\zeta}$
in the sense of measures on $\O\times (0,T)$. Since we already established
that $(\beta+\zeta)(\bu)=\overline{\beta}+\overline{\zeta}$, we can
therefore apply Lemma \ref{lem:mintylike} with $\varphi\equiv 1$, $w_m=\Pi_{\disc_m}u_m$
and $w = \overline{u}$. This gives
 $\overline{\beta}=\beta(\overline{u})$ and
$\overline{\zeta}=\zeta(\overline{u})$ a.e. on $\O\times (0,T)$, as required.

To summarise, the limits of $\Pi_{\disc_m}\beta(u_m)$, $\Pi_{\disc_m}\zeta(u_m)$ and
$\Pi_{\disc_m}\nu(u_m)$ have been identified as $\beta(\overline{u})$,
$\zeta(\overline{u})$ and $\nu(\overline{u})$ for some $\overline{u}$.
Since $\zeta(\overline{u})=\overline{\zeta}\in L^p(\O\times(0,T))$, the growth
assumptions \eqref{hypzeta} on $\zeta$ ensure that $\overline{u}\in L^p(\O\times (0,T))$.
We can then take over the proof of Theorem \ref{th:weakconv} from
after the usage of \eqref{hyp:struct}, using the $\overline{u}$ we just found instead
of the one defined as the weak limit of $\Pi_{\disc_m}u_m$. This allows us to conclude
that $\overline{u}$ is a solution to \eqref{ellgenfllt}, and that the
convergences in \eqref{conv:base} hold.
\end{proof}

\begin{remark} It is not proved that $\overline{u}$ is a weak limit of $\Pi_{\disc_m}u_m$.
Such a limit is not stated in \eqref{conv:base} and is not necessarily expected for the model \eqref{pbintrot}, in which
the quantities of interest (physically relevant when this PDE
models a natural phenomenon) are $\beta(\bu)$, $\zeta(\bu)$
and $\nu(\bu)$.
\end{remark}

\begin{remark}[Maximal monotone operator]\label{rem:maxmongraph} Hypotheses \eqref{hypzeta} and \eqref{hypbeta} imply that the operator $T$ defined by the graph
$\mathcal{G}(T) = \{(\zeta(s),\beta(s)),s\in\R\}$ is a maximal monotone operator with domain $\R$, such that $0\in T(0)$. Indeed, assume that $x,y$ satisfy $(\zeta(s)-x)(\beta(s) - y)\ge 0$ for all $s\in\R$. Then, letting $w\in\R$ be such that 
\begin{equation}\label{rem:xy}
\frac{\beta(w)+\zeta(w)}{2}=\frac{x+y}{2},
\end{equation}
we have $(\zeta(w)-x)(\beta(w) - y) = -( \frac {\zeta(w)-\beta(w)} 2 - \frac {x-y} 2)^2 \ge 0$. This implies $\frac{\zeta(w)-\beta(w)}{2}=\frac{x-y}{2}$
which, combined with \eqref{rem:xy}, gives $x = \zeta(w)$ and $y = \beta(w)$ and hence
$(x,y)\in \mathcal{G}(T)$. 

Reciprocally, for any maximal monotone operator $T$ from $\R$ to $\R$ such that $0\in T(0)$, one can find
$\zeta$ and $\beta$ satisfying \eqref{hypzeta} and \eqref{hypbeta}, and such that $\mathcal{G}(T) = \{(\zeta(s),\beta(s)),s\in\R\}$. Indeed, for all $(x,y)\in \mathcal{G}(T)$ and $(x',y')\in \mathcal{G}(T)$
satisfying $x+y = x'+y'$, since $(x-x')(y-y')\ge 0$ we have $x=x'$ and $y=y'$. We can therefore define $\zeta$ and $\beta$ by: for all $(x,y)\in\mathcal G(T)$,
$x = \zeta(\frac {x+y} 2)$ and $y = \beta(\frac {x+y} 2)$. We observe that these functions are nondecreasing
and Lipschitz-continuous with constant 2, and that $\zeta+\beta = 2 {\rm Id}$.

Hence, Theorem \ref{th:weaknonstruct} applies to the model considered in \cite{rul1996opt}, but provides convergence results for much more general equations and various numerical methods in any space dimension.
\end{remark}

We now state the two key results that allowed us to remove Assumption \eqref{hyp:struct} if $p\ge 2$.
The first one is a discrete version of a compensated compactness
result in \cite{KAZ98}. The second is a Minty-like result, useful to identify
weak non-linear limits.

We note that Theorem \ref{thm-comp-comp} states a more general convergence result than
needed for the proof of Theorem \ref{th:weaknonstruct} (which only requires $\varphi\equiv 1$).
We nevertheless state the general form in order to obtain the genuine discrete equivalent
of the result in \cite{KAZ98}. We also believe that this discrete
compensated compactness theorem will find many more applications in the numerical analysis
of degenerate or coupled parabolic models. We also refer to \cite{ACM} for another
transposition to the discrete setting of a compensated compactness result.

\begin{theorem}[Discrete compensated compactness]
\label{thm-comp-comp}
We take $T>0$, $p\ge 2$ and a sequence $(\disc_m)_{m\in\N} = (X_{\disc_m,0}, \Pi_{\disc_m},\nabla_{\disc_m}, \interp_{\disc_m},(t_m^{(n)})_{n=0,\ldots,N_m})_{m\in\N}$
of space-time gradient discretisations, in the sense of Definition \ref{def-stcons}, that
is consistent and compact in the sense of Definitions \ref{def-cons} and \ref{def-comp}.

For any $m\in\N$, let $\beta_m=(\beta_m^{(n)})_{n=0,\ldots,N_m} \subset X_{\disc_m,0}$ and $\zeta_m=(\zeta_m^{(n)})_{n=0,\ldots,N_m} \subset X_{\disc_m,0}$ be such that
\begin{itemize}
\item the sequences $(\int_0^T |\delta_m \beta_m(t)|_{\star,\disc_m})_{m\in\N}$ 
and $(||\nabla_{\disc_m}\zeta_m||_{L^2(0,T;L^p(\O)^d)})_{m\in\N}$ are bounded,
\item as $m\to\infty$,
$\Pi_{\disc_m}\beta_m\to \overline{\beta}$ and $\Pi_{\disc_m}\zeta_m\to \overline{\zeta}$ weakly in $L^2(\O\times(0,T))$.
\end{itemize}

Then $(\Pi_{\disc_m}\beta_m)(\Pi_{\disc_m}\zeta_m)\to \overline{\beta}\,\overline{\zeta}$ in
the sense of measures on $\O\times(0,T)$, that is, for all
$\varphi \in C(\overline{\O}\times [0,T])$,
\begin{equation}
\lim_{m\to\infty} \int_0^T\int_\O \Pi_{\disc_m}\beta_m(\x,t)\Pi_{\disc_m}\zeta_m(\x,t)\varphi(\x,t)
\d\x\d t = \int_0^T\int_\O \overline{\beta}(\x,t)\,\overline{\zeta}(\x,t)\varphi(\x,t)\d\x\d t.
 \label{eq-comp-comp}
\end{equation}
\end{theorem}

\begin{proof}

The idea is to reduce to the case where $\Pi_{\disc_m}\zeta_m$ is a
tensorial function, in order to separate the space and time variables and make
use of the compactness of $\Pi_{\disc_m}\zeta_m$ and $\Pi_{\disc_m}\beta_m$
with respect to each of these variables. Note that the technique we use here apparently provides
a new proof for the continuous equivalent of this compensated compactness result.

\medskip

\textbf{Step 1}: reduction of $\Pi_{\disc_m}\zeta_m$ to tensorial functions.

Let us take $\delta>0$ and let us consider a covering $(A^\delta_k)_{k=1,\ldots,K}$ of $\Omega$ in 
disjoint cubes of length $\delta$.
Let  $R_\delta:L^2(\O)\to L^2(\O)$ be the operator defined by:
\[
\forall g\in L^2(\O)\,,\;\forall k=1,\ldots,K\,,\;\forall \x\in A_k^\delta\cap \O\,:\,
R_\delta g(\x) = \frac{1}{{\rm meas}(A_k^\delta)}\int_{A^k_\delta} g(\y)\d\y,
\]
where $g$ has been extended by $0$ outside $\O$. Let $\x\in A_k^\delta\cap \O$.
Using Jensen's inequality, the fact that ${\rm meas}(A_k^\delta)=\delta^d$ and the change of variable
$\y\in A_k^\delta\mapsto \bxi=\y-\x\in (-\delta,\delta)^d$, we can write
\[
|R_\delta g(\x)-g(\x)|^2 \le \delta^{-d}\int_{A_k^\delta}|g(\y)-g(\x)|^2\d\y
\le\delta^{-d}\int_{(-\delta,\delta)^d}|g(\x+\bxi)-g(\x)|^2\d \bxi.
\]
Integrating over $\x\in A_k^\delta$ and summing over $k=1,\ldots,K$ gives
\begin{eqnarray}
||R_\delta g-g||_{L^2(\O)}^2 &\le&\delta^{-d}
 \int_{(-\delta,\delta)^d}\int_{\R^d}|g(\x+\bxi)-g(\x)|^2\d\x\d \bxi\nonumber\\
&\le& 2^d\sup_{\bxi\in (-\delta,\delta)^d} \int_{\R^d}|g(\x+\bxi)-g(\x)|^2\d\x.
\label{estRd}\end{eqnarray}
The compactness of $(\disc_m)_{m\in\N}$ (Definition \ref{def-comp}) and
the fact that $p\ge 2$ give $\epsilon(\bxi)$ such that
$\epsilon(\bxi)\to 0$ as $\bxi\to 0$ and, for all $m\in\N$ and all $v\in X_{\disc_m,0}$,
\[
||\Pi_{\disc_m}v(\cdot+\bxi)-\Pi_{\disc_m}v||_{L^2(\R^d)}^2\le \epsilon(\bxi)||\nabla_{\disc_m} v||_{L^p(\O)^d}^2.
\]
Combining this with \eqref{estRd} and using the bound on $||\nabla_{\disc_m}\zeta_m||_{L^2(0,T;L^p(\O)^d)}$
shows that
\be\label{estRd2}
||R_\delta \Pi_{\disc_m}\zeta_m-\Pi_{\disc_m}\zeta_m||_{L^2(\O\times(0,T))}
\le C\sup_{|\bxi|_\infty\le \delta}\sqrt{\epsilon(\bxi)}=:\omega(\delta)
\ee
where $C$ does not depend on $m$, and $\omega(\delta)\to 0$ as $\delta\to 0$. Note that
a similar estimate holds with $\Pi_{\disc_m}\zeta_m$ replaced with $\overline{\zeta}$
since $\overline{\zeta}\in L^2(\O\times(0,T))$.

If we respectively denote by $\mathcal A_m(\Pi_{\disc_m}\zeta_m)$ and $\mathcal A(\overline{\zeta})$ the 
integrals in the left-hand side
and right-hand side \eqref{eq-comp-comp}, then since $(\Pi_{\disc_m}\beta_m)_{m\in\N}$
is bounded in $L^2(\O\times (0,T))$ we have by \eqref{estRd2}
\begin{equation}\label{previneq}
|\mathcal A_m(\Pi_{\disc_m}\zeta_m)-\mathcal A(\overline{\zeta})|\le C\omega(\delta)
+|\mathcal A_m(R_\delta\Pi_{\disc_m}\zeta_m)-\mathcal A(R_\delta\overline{\zeta})|.
\end{equation}
Let us assume that we can prove that, for a fixed $\delta$,
\be\label{wwntp}
\mathcal A_m(R_\delta\Pi_{\disc_m}\zeta_m)\to \mathcal A(R_\delta\overline{\zeta})\mbox{ as $m\to\infty$}.
\ee
Then \eqref{previneq} gives $\limsup_{m\to\infty}
|\mathcal A_m(\Pi_{\disc_m}\zeta_m)-\mathcal A(\overline{\zeta})|\le C\omega(\delta)$. Letting $\delta\to 0$
in this inequality gives $\mathcal A_m(\Pi_{\disc_m}\zeta_m)\to \mathcal A(\overline{\zeta})$ as wanted.
Hence, we only need to prove \eqref{wwntp}. 

The definition of $R_\delta$ shows that
\[
R_\delta g=\sum_{k=1}^K\frac{1}{{\rm meas}(A_k^\delta)}\mathbf{1}_{A_k^\delta}
[g]_{A_k^\delta},
\]
where $\mathbf{1}_{A_k^\delta}$ is the characteristic function of $A_k^\delta$
and $[g]_A=\int_A g(\x)\d \x$. Hence, \eqref{wwntp} follows if we can prove that
for any measurable set $A$
\begin{multline}
\lim_{m\to\infty} \int_0^T\int_\O \Pi_{\disc_m}\beta_m(\x,t)[\Pi_{\disc_m}\zeta_m]_A(t)
\varphi(t,\x)\mathbf{1}_A(\x)
\d\x\d t \\
= \int_0^T\int_\O \overline{\beta}(\x,t)[\,\overline{\zeta}\,]_A(t)
\varphi(t,\x)\mathbf{1}_A(\x)\d\x\d t
\label{eq-comp-comp2}
\end{multline}
where for $g\in L^2(\O\times (0,T))$ we set $[g]_A(t)=\int_A g(t,\y)\d\y$.

\medskip

\textbf{Step 2}: further reductions.

We now reduce $\varphi$ to a tensorial function and $\mathbf{1}_A$ to a smooth function.
It is well-known that there exists tensorial functions $\varphi_r=\sum_{l=1}^{L_r}
\theta_{l,r}(t)\gamma_{l,r}(\x)$, with $\theta_{l,r}\in C^\infty([0,T])$ and $\gamma_{l,r}\in
C^\infty(\overline{\Omega})$, such that $\varphi_r\to \varphi$ uniformly on $\O\times (0,T)$
as $r\to\infty$. Moreover, there exists $\rho_r\in C^\infty_c(\O)$ such that
$\rho_r\to \mathbf{1}_A$ in $L^2(\O)$ as $r\to\infty$.

Hence, as $r\to\infty$ the function $(t,\x)\mapsto \varphi_r(t,\x)\rho_r(\x)$ converges
in $L^\infty(0,T;L^2(\O))$ to the function $(t,\x)\mapsto \varphi(t,\x)\mathbf{1}_A(\x)$.
Since the sequence of functions $(t,\x)\mapsto \Pi_{\disc_m}\beta_m(t,\x)
[\Pi_{\disc_m}\zeta_m]_A(t)$ is bounded in $L^1(0,T;L^2(\O))$ (notice
that $([\Pi_{\disc_m}\zeta_m]_A)_{m\in\N}$ is bounded in $L^2(0,T)$
since $(\Pi_{\disc_m}\zeta_m)_{m\in\N}$ is bounded in $L^2(\O\times(0,T))$),
a reasoning similar to the one used in Step 1 shows that we only need
to prove \eqref{eq-comp-comp2} with $\varphi(t,\x)\mathbf{1}_A(\x)$ replaced
with $\varphi_r(t,\x)\rho_r(\x)$ for a fixed $r$.

We have $\varphi_r(t,\x)\rho_r(\x)=\sum_{l=1}^{L_r}
\theta_{l,r}(t)(\gamma_{l,r}\rho_r)(\x)$ and $\gamma_{l,r}\rho_r\in C^\infty_c(\O)$.
Hence, \eqref{eq-comp-comp2} with $\varphi(t,\x)\mathbf{1}_A(\x)$ replaced
with $\varphi_r(t,\x)\rho_r(\x)$ will follow if we can establish that
for any $\theta\in C^\infty([0,T])$, any $\psi\in C^\infty_c(\O)$
and any measurable set $A$
\begin{equation}
\lim_{m\to\infty} \int_0^T\int_\O \theta(t)\Pi_{\disc_m}\beta_m(\x,t)[\Pi_{\disc_m}\zeta_m]_A(t)
\psi(\x)\d\x\d t = \int_0^T\int_\O \theta(t)\overline{\beta}(\x,t)[\,\overline{\zeta}\,]_A(t)\psi(\x)\d\x\d t.
\label{eq-comp-comp3}
\end{equation}

\medskip

\textbf{Step 3}: proof of \eqref{eq-comp-comp3}.

We now use the estimate on $\delta_m\beta_m$ to conclude. We write
\begin{equation}
\int_0^T\int_\O \theta(t)\Pi_{\disc_m}\beta_m(\x,t)[\Pi_{\disc_m}\zeta_m]_A(t)
\psi(\x)\d\x\d t
=\int_0^T \theta(t)[\Pi_{\disc_m}\zeta_m]_A(t) F_m(t)
\label{defFm}
\end{equation}
with $F_m(t)= \int_\O \Pi_{\disc_m}\beta_m(\x,t)\psi(\x)\d\x$. It is clear
from the weak convergence of $\Pi_{\disc_m}\zeta_m$ that $[\Pi_{\disc_m}\zeta_m]_A
\to [\,\overline{\zeta}\,]_A$ weakly in $L^2(0,T)$. Hence, if we can prove that $F_m\to 
F:=\int_\O \overline{\beta}(\x,\cdot)\psi(\x)\d\x$ strongly in $L^2(0,T)$, we can
pass to the limit in \eqref{defFm} and obtain \eqref{eq-comp-comp3}.
Since $F_m$ weakly converges to $F$ in $L^2(0,T)$ (thanks to the weak convergence
of $\Pi_{\disc_m}\beta_m$ in $L^2(\O\times(0,T))$), we only have to prove that
$(F_m)_{m\in\N}$ is relatively compact in $L^2(0,T)$.

We introduce the interpolant $P_{\disc_m}$ defined by \eqref{def-PD}
and we define $G_m$ as $F_m$ with $\psi$ replaced with $\Pi_{\disc_m}P_{\disc_m}\psi$. We then have
\[
|F_m(t)-G_m(t)| \le ||\Pi_{\disc_m}\beta_m(\cdot,t)||_{L^2(\O)}S_{\disc_m}(\psi).
\]
The consistency of $(\disc_m)_{m\in\N}$ thus shows that 
\begin{equation}
\mbox{$F_m-G_m\to 0$ strongly in $L^2(0,T)$ as $m\to\infty$.}
\label{strFG}
\end{equation}
We now study the strong convergence of $G_m$. This function is, like $\Pi_{\disc_m}\beta_m$,
piecewise constant on $(0,T)$ and, by definition of
$|\cdot|_{\star,\disc_m}$, its discrete derivative satisfies
\[
|\delta_mG_m(t)|\le |\delta_m\beta_m (t)|_{\star,\disc_m}||\nabla_{\disc_m}P_{\disc_m}\psi||_{L^p(\O)^d}.
\]
Since $||\nabla_{\disc_m}P_{\disc_m}\psi||_{L^p(\O)^d}\le S_{\disc_m}(\psi)+||\nabla\psi||_{L^p(\O)^d}$ is bounded
uniformly with respect to $m$, the assumption on $\delta_m \beta_m$ proves that
$(||\delta_m G_m||_{L^1(0,T)})_{m\in\N}$ is bounded. We have $||\delta_m G_m||_{L^1(0,T)}=|G_m|_{BV(0,T)}$,
and $(\Pi_{\disc_m}\beta_m)_{m\in\N}$ is bounded in $L^2(\O\times (0,T))$;
hence, $(G_m)_{m\in\N}$ is bounded in $BV(0,T)\cap L^2(0,T)$
and therefore relatively compact in $L^2(0,T)$ (see \cite[Theorem 10.1.4]{ABM06}). Combined with \eqref{strFG}, this shows that
$(F_m)_{m\in \N}$ is relatively compact in $L^2(0,T)$ and concludes the proof. \end{proof}

\begin{remark} If we assume that $(\Pi_{\disc_m}\beta_m)_{m\in\N}$ is bounded
in $L^\infty(0,T;L^2(\O))$ and that, for some $q>1$,
$(\int_0^T |\delta_m \beta_m(t)|^q_{\star,\disc_m})_{m\in\N}$ is bounded, then
Step 3 becomes a trivial consequence of Theorem \ref{unifweakGDcomp}.
Indeed, this theorem shows that $(\Pi_{\disc_m}\beta_m)_{m\in\N}$ is
relatively compact uniformly-in-time and weakly in $L^2(\O)$, which translates
into the relative compactness of $(F_m)_{m\in \N}$ in $L^\infty(0,T)$.
\end{remark}

\begin{lemma} \label{lem:mintylike}
Let $V$ be a non-empty measurable subset of $\R^N$, $N\ge 1$.
Let $\beta,\zeta\in C^0(\R)$ be two
nondecreasing functions such that $\beta(0)=\zeta(0)=0$. We assume that there
exists a sequence $(w_m)_{m\in\N}$ of measurable functions on $V$, and two functions $\overline{\beta},\overline{\zeta} \in L^2(V)$ such that:
\begin{itemize}
\item $\beta(w_m)\to \overline{\beta}$ and $\zeta(w_m)\to\overline{\zeta}$ weakly in $L^2(V)$,
\item there exists $\varphi\in L^\infty(V)$ such that $\varphi>0$ a.e. on $V$ and 
\begin{equation}
\lim_{m\to\infty} \int_V \varphi(\z)\beta(w_m(\z))\zeta(w_m(\z))\d\z = \int_V \varphi(\z)\overline{\beta}(\z)\,\overline{\zeta}(\z)\d\z. 
\label{lem:cvint}\end{equation}
\end{itemize}
Then, for any measurable function $w$ such that $(\beta + \zeta)(w) = \overline{\beta}+\overline{\zeta}$ a.e. in $V$, we have
\begin{equation}
\overline{\beta} = \beta(w)\hbox{ and } \overline{\zeta} = \zeta(w) \hbox{ a.e. in $V$}.
\label{eq:mintylike}
\end{equation}
\end{lemma}

\begin{proof}
We first notice that $\beta(w)$ and $\zeta(w)$ belong to $L^2(V)$ since they have the same sign and therefore verify $|\beta(w)| + |\zeta(w)| = |\overline{\beta}+\overline{\zeta}|\in L^2(V)$. Using the fact that
$\beta$ and $\zeta$ are non-decreasing, we can write
\[
 \int_V \varphi(\z ) \left[\beta(w_m(\z)) - \beta(w(\z))\right]\,\left[\zeta(w_m(\z)) - \zeta(w(\z))\right]\d\z \ge 0.
\]
Letting $m\to\infty$ in the above inequality, and using the convergences of $\beta(w_m)$,
$\zeta(w_m)$ and \eqref{lem:cvint}, we obtain
\begin{equation}\label{ineq-pos}
 \int_V \varphi(\z) \left[\overline{\beta}(\z) - \beta(w(\z))\right]
\left[\, \overline{\zeta}(\z) - \zeta(w(\z))\right]\d\z \ge 0.
\end{equation}
We then remark that $\overline{\beta}+\overline{\zeta}= \beta(w) + \zeta(w)$
gives $\beta(w) = \frac {\overline{\beta}+\overline{\zeta}} 2 + \left(\frac {\beta-\zeta} 2\right)(w)$
and $\zeta(w) = \frac {\overline{\beta}+\overline{\zeta}} 2 - \left(\frac {\beta-\zeta} 2\right)(w)$.
Hence, \eqref{ineq-pos} leads to
\[
- \int_V \varphi(\z) \left[\frac {\overline{\beta}-\overline{\zeta}} 2(\z) - 
\left(\frac {\beta-\zeta} 2\right)(w(\z))\right]^2\d\z \ge 0.
\]
Since $\varphi$ is almost everywhere strictly positive on $V$, we deduce
that $\frac {\overline{\beta}-\overline{\zeta}} 2 = \frac {\beta(w)-\zeta(w)} 2$ a.e. in $V$,
and \eqref{eq:mintylike} follows from $\frac{\overline{\beta}+\overline{\zeta}}{2}=
\frac{\beta(w) + \zeta(w)}{2}$.
\end{proof}

\section{Appendix: uniform-in-time compactness results for time-dependent problems} \label{sec:comptime}

We establish in this appendix some generic results, unrelated to the framework
of gradient schemes, that form the starting point for our uniform-in-time convergence
results.

\medskip

Solutions of numerical schemes for parabolic equations are usually
piecewise constant, and therefore not continous, in time. As their
jumps nevertheless tend to become small as the time step goes to $0$, it
is possible to establish uniform-in-time convergence properties
using a generalisation to non-continuous functions of the classical Ascoli-Arzel\`a theorem.

\begin{definition}\label{def-mathcalF}
If $(K,d_K)$ and $(E,d_E)$ are metric spaces, we denote by
$\mathcal F(K,E)$ the space of functions $K\to E$ endowed with
the uniform metric $d_\mathcal F(v,w)=\sup_{s\in K}d_E(v(s),w(s))$ (note that
this metric may take infinite values).
\end{definition}

\begin{theorem}[discontinuous Ascoli-Arzel\`a's theorem]\label{genascoli}
Let $(K,d_K)$ be a compact metric spa\-ce, $(E,d_E)$ be a complete metric space
and $(\mathcal F(K,E),d_{\mathcal F})$ be as in Definition \ref{def-mathcalF}.

Let $(v_m)_{m\in\N}$ be a sequence in $\mathcal F(K,E)$ such that there exists
a function $\omega:K\times K\to [0,\infty]$ and a sequence $(\delta_m)_{m\in\N}
\subset [0,\infty)$ satisfying
\begin{equation}\label{quasi-equic}
\begin{array}{l}
\dsp\lim_{d_K(s,s')\to 0}\omega(s,s')=0\,,\quad \lim_{m\to \infty}\delta_m=0\,,\\[1em]
\dsp\forall (s,s')\in K^2\,,\;\forall m\in\N\,,\; 
d_E(v_m(s),v_m(s'))\le \omega(s,s')+\delta_m.
\end{array}
\end{equation}
We also assume that, for all $s\in K$, $\{v_m(s)\,:\,m\in\N\}$ is relatively compact
in $(E,d_E)$.

Then $(v_m)_{m\in\N}$ is relatively compact in $(\mathcal F(K,E),d_{\mathcal F})$
and any adherence value of $(v_m)_{m\in\N}$ in this space is continuous $K\to E$.
\end{theorem}

\begin{proof}
Let us first notice that the last conclusion of the theorem, i.e. that any adherence value
$v$ of $(v_m)_{m\in\N}$ in $\mathcal F(K,E)$ is continuous, is trivially
obtained by passing to the limit in \eqref{quasi-equic}, which shows that the modulus of continuity of $v$
is bounded above by $\omega$.

The proof of the compactness result is an easy generalisation of the
proof of the classical Ascoli-Arzel\`a theorem. We start by taking a countable dense subset $\{s_l\,:\,l\in\N\}$ in $K$ (the existence
of this set is ensured since $K$ is compact metric). Since each set $\{v_m(s_l)\,:\,m\in\N\}$
is relatively compact in $E$, by diagonal extraction we can select a subsequence
of $(v_m)_{m\in\N}$, denoted the same way, such that, for any $l\in\N$,
$(v_m(s_l))_{m\in\N}$ converges in $E$.
We then proceed to show that $(v_m)_{m\in\N}$ is a Cauchy sequence in $(\mathcal F(K,E),d_{\mathcal F})$.
Since this space is complete, this will prove that this sequence converges in
this space, which will complete the proof.

Let $\varepsilon>0$ and, using \eqref{quasi-equic}, take $\rho>0$ and $M\in\N$
such that $\omega(s,s')\le \varepsilon$ whenever $d_K(s,s')\le \rho$ and
$\delta_m\le \varepsilon$ whenever $m\ge M$.
Select a finite set $\{s_{l_1},\ldots,s_{l_N}\}$ such that
any $s\in K$ is within distance $\rho$ of a $s_{l_i}$. Then for any
$m,m'\ge M$
\begin{eqnarray*}
d_E(v_m(s),v_{m'}(s))&\le& d_E(v_m(s),v_m(s_{l_i}))+d_E(v_m(s_{l_i}),v_{m'}(s_{l_i}))
+d_E(v_{m'}(s_{l_i}),v_{m'}(s))\\
&\le& \omega(s,s_{l_i})+\delta_m + d_E(v_m(s_{l_i}),v_{m'}(s_{l_i})) + \omega(s,s_{l_i})+\delta_{m'}\\
&\le& 4\varepsilon + d_E(v_m(s_{l_i}),v_{m'}(s_{l_i})).
\end{eqnarray*}
Since $\{(v_m(s_{l_i}))_{m\in\N}\,:\,i=1,\ldots,N\}$ forms a finite number of converging
sequences in $E$, we can find $M'\ge M$ such that, for all $m,m'\ge M'$ and all
$i=1,\ldots,N$, $d_E(v_m(s_{l_i}),v_{m'}(s_{l_i}))\le \varepsilon$. This shows that,
for all $m,m'\ge M'$ and all $s\in K$, $d_E(v_m(s),v_{m'}(s))\le 5\varepsilon$
and concludes the proof that $(v_m)_{m\in\N}$ is a Cauchy sequence in $(\mathcal F(K,E),d_{\mathcal F})$.
\end{proof}

\begin{remark}
Conditions \eqref{quasi-equic} are usually the most
practical when $(v_m)_{m\in\N}$ are piecewise constant-in-time solutions to numerical schemes
(see e.g. the proof of Theorem \ref{unifweakGDcomp}). Here, $\omega$ is expected to measure
the size of the cumulated jumps of $v_m$ between $s$ and $s'$,
and $\delta_m$ accounts for boundary effects which may occur in the small time intervals
containing $s$ and $s'$.

It is easy to see that \eqref{quasi-equic} can be replaced with
\begin{equation}\label{cond-ascoli}
d_E(v_m(s),v_m(s')) \to 0\,, \mbox{ as $m\to\infty$ and $d_K(s,s')\to 0$}
\end{equation}
(under this condition, the proof can be carried out by selecting
$M\in\N$ and $\rho>0$ such that $d_E(v_m(s),v_m(s'))\le \varepsilon$
whenever $m\ge M$ and $d_K(s,s')\le \rho$).
It turns out that \eqref{cond-ascoli} is actually a necessary and sufficient
condition for the theorem's conclusions to hold true.
\end{remark}

The following lemma states an equivalent condition for the uniform convergence
of functions, which proves extremely useful to establish uniform-in-time
convergence of numerical schemes for parabolic equations when no smoothness
is assumed on the data.

\begin{lemma} \label{equiv-unifconv}
Let $(K,d_K)$ be a compact metric space, $(E,d_E)$ be a metric space
and $(\mathcal F(K,E),d_{\mathcal F})$ as in Definition \ref{def-mathcalF}.
Let $(v_m)_{m\in\N}$ be a sequence in $\mathcal F(K,E)$ and $v:K\mapsto E$ be
continuous.

Then $v_m\to v$ for $d_{\mathcal F}$ if and only if, for any $s\in K$ and
any sequence $(s_m)_{m\in\N}\subset K$ converging to $s$ for $d_K$, we have
$v_m(s_m)\to v(s)$ for $d_E$.
\end{lemma}

\begin{proof} If $v_m\to v$ for $d_{\mathcal F}$ then for any sequence $(s_m)_{m\in\N}$
converging to $s$
\[
d_E(v_m(s_m),v(s))\le d_E(v_m(s_m),v(s_m))+d_E(v(s_m),v(s))
\le d_{\mathcal F}(v_m,v)+d_E(v(s_m),v(s)).
\]
The right-hand side tends to $0$ by definition of $v_m\to v$ for
$d_{\mathcal F}$ and by continuity of $v$, which shows that
$v_m(s_m)\to v(s)$ for $d_E$.

Let us now prove the converse by contradiction. 
If $(v_m)_{m\in\N}$ does not converge to $v$ for $d_{\mathcal F}$ then there
exists $\varepsilon>0$ and a subsequence $(v_{m_k})_{k\in\N}$,
such that, for any $k\in\N$,
$\sup_{s\in K} d_E(v_{m_k}(s),v(s))\ge \varepsilon$. We can then find
a sequence $(r_k)_{k\in\N}\subset K$ such that, for any $k\in\N$,
\begin{equation}\label{contr}
d_E(v_{m_k}(r_k),v(r_k))\ge \varepsilon/2.
\end{equation}
$K$ being compact, up to another subsequence denoted the same way, we can assume that
$r_k$ converges as $k\to \infty$ to some $s$ in $K$. It is then trivial to construct a sequence
$(s_m)_{m\in\N}$ converging to $s$ and such that $s_{m_k}=r_k$ (just
take $s_m=s$ when $m$ is not an $m_k$).
We then have $v_m(s_m)\to v(s)$ in $E$ and, by continuity of $v$,
$v(s_m)\to v(s)$ in $E$. This shows that $d_E(v_m(s_m),v(s_m))\to 0$, which
contradicts \eqref{contr} and concludes the proof.
\end{proof}

The next result is classical. Its short proof is recalled for
the reader's convenience.

\begin{proposition}\label{propweakunifconv}
Let $E$ be a closed bounded ball in $L^2(\O)$ and let $(\varphi_l)_{l\in\N}$ be a dense
sequence in $L^2(\O)$. Then, on $E$, the weak topology
of $L^2(\O)$ is the topology given by the metric
\begin{equation}\label{def-distweak}
d_E(v,w)=\sum_{l\in\N} \frac{\min(1,|\langle v-w,\varphi_l\rangle_{L^2(\O)}|)}{2^l}.
\end{equation}
Moreover, a sequence of functions $u_m:[0,T]\to E$ converges uniformly-in-time
to $u:[0,T]\to E$ for the weak topology of $L^2(\O)$ (see Definition \ref{defweakunifconv})
if and only if, as $m\to\infty$, $d_E(u_m,u):[0,T]\to [0,\infty)$ converges uniformly to $0$.
\end{proposition}

\begin{proof}
The sets $E_{\varphi,\varepsilon}=
\{v\in E\,:\,|\langle v,\varphi\rangle_{L^2(\O)}|<\varepsilon\}$, for 
$\varphi\in L^2(\O)$ and $\varepsilon>0$, define a basis of neighborhoods
of $0$ for the weak $L^2(\O)$ topology on $E$, and a basis of neighborhoods of any other
point is obtained by translation. If $R$ is the radius
of the ball $E$ then for any $\varphi\in L^2(\O)$, $l\in\N$ and $v\in E$ we have
\[
|\langle v,\varphi\rangle_{L^2(\O)}|\le R||\varphi-\varphi_l||_{L^2(\O)}
+|\langle v,\varphi_l\rangle_{L^2(\O)}|.
\]
By density of $(\varphi_l)_{l\in\N}$ we can select $l\in\N$ such that
$||\varphi-\varphi_l||_{L^2(\O)}<\varepsilon/(2R)$ and we then see that
$E_{\varphi_l,\varepsilon/2}\subset E_{\varphi,\varepsilon}$. Hence, a basis
of neighborhoods of $0$ in $E$ for the weak $L^2(\O)$
is also given by $(E_{\varphi_l,\varepsilon})_{l\in\N,\,\varepsilon>0}$.

{}From the definition of $d_E$ we see that, for any $l\in\N$,
$\min(1,|\langle v,\varphi_l\rangle_{L^2(\O)}|)\le 2^l d_E(0,v)$. If $d_E(0,v)<2^{-l}$
this shows that $|\langle v,\varphi_l\rangle_{L^2(\O)}|\le 2^l d_E(0,v)$ and
therefore that 
\[
B_{d_E}(0,\min(2^{-l},\varepsilon 2^{-l}))\subset E_{\varphi_l,\varepsilon}.
\]
Hence, any neighborhood of $0$ in $E$ for the $L^2(\O)$ weak topology
is a neighborhood of $0$ for $d_E$. Conversely, for any $\varepsilon>0$,
selecting $N\in\N$ such that $\sum_{l\ge N+1}2^{-l}<\varepsilon/2$
gives, from the definition \eqref{def-distweak} of $d_E$,
\[
\bigcap_{l=1}^N E_{\varphi_l,\varepsilon/4}\subset B_{d_E}(0,\varepsilon).
\]
Hence, any ball for $d_E$ centered at $0$ is a neighborhood of $0$ for
the $L^2(\O)$ weak topology. Since $d_E$ and the $L^2(\O)$ weak neighborhoods
are invariant by translation, this concludes the proof that this weak topology
is identical to the topology generated by $d_E$.

The conclusion on weak uniform convergence of sequences of functions follows
from the preceding result, and more precisely by noticing that all previous
inclusions are, when applied to $u_m(t)-u(t)$, uniform with respect to $t\in [0,T]$.
\end{proof}

The following lemma has been initially established in \cite[Proposition 9.3]{eym03}.

\begin{lemma} \label{estttt}~

Let $(t^{(n)})_{n\in\Z}$ be a stricly increasing sequence of real values
such that $\dt^{(n+\half)} := t^{(n+1)} - t^{(n)}$ is uniformly bounded by $\dt>0$,
$\dsp\lim_{n\to-\infty} t^{(n)} = -\infty$
and $\dsp\lim_{n\to\infty} t^{(n)} = \infty$. For all $t\in\R$, we denote
by $n(t)$ the element $n\in\Z$ such that $t\in (t^{(n)},t^{(n+1)}]$.
Let $(a^{(n)})_{n\in\Z}$ be a family of non negative real numbers with a finite number of non
zero values.
Then
\begin{equation}
\int_{\R} \sum_{n=n(t)+1}^{n(t+\tau)} (\dt^{(n+\half)} a^{(n+1)}) \d t = \tau
\sum_{n\in\Z} (\dt^{(n+\half)} a^{(n+1)}), \quad \forall \tau>0,
\label{esttt1}
\end{equation}
and
\begin{equation}
\int_{\R} \left(\sum_{n=n(t)+1}^{n(t+\tau)}\dt^{(n+\half)}\right)  a^{n(t+s)+1} \d t \le
(\tau + \dt) \sum_{n\in\Z} \dt^{(n+\half)} a^{(n+1)}, \quad \forall \tau>0,
\ \forall s\in\R.
\label{esttt2}
\end{equation}
\end{lemma}

\begin{proof}

Let us define $\chi$ by $\chi(t,n,\tau) = 1$ if $t^{(n)}\in [t,t+\tau)$, 
otherwise $\chi(t,n,\tau) = 0$. We have
\begin{align*}
\int_{\R} \sum_{n=n(t)+1}^{n(t+\tau)}(\dt^{(n+\half)}  a^{(n+1)}) \d t  &=
\int_{\R} \sum_{n\in\Z} (\dt^{(n+\half)} a^{(n+1)} \chi(t,n,\tau))\d t\\ &=
\sum_{n\in\Z} \left(\dt^{(n+\half)} a^{(n+1)} \int_{\R} \chi(t,n,\tau)\d t\right).
\end{align*}
Since $\int_{\R} \chi(t,n,\tau)\d t = \int_{t^{(n)}-\tau}^{t^{(n)}} \d t = \tau$,
Relation \eqref{esttt1} is proved.

We now turn to the proof of \eqref{esttt2}. We define $\widetilde\chi$
by $\widetilde\chi(n,t) = 1$ if $n(t) = n$, otherwise
$\widetilde\chi(n,t) = 0$. We have
\[
\int_{\R} \left(\sum_{n=n(t)+1}^{n(t+\tau)}\dt^{(n+\half)}\right)  a^{(n(t+s)+1)} \d t  =
\int_{\R} \left(\sum_{n=n(t)+1}^{n(t+\tau)}\dt^{(n+\half)}\right)
\sum_{m\in\Z}a^{(m+1)}\widetilde\chi(m,t+s) \d t,
\]
which yields
\begin{equation}
\int_{\R} \left(\sum_{n=n(t)+1}^{n(t+\tau)}\dt^{(n+\half)}\right)  a^{(n(t+s)+1)}  \d t =
\sum_{m\in\Z} a^{(m+1)} \int_{t^{(m)}-s}^{t^{(m+1)}-s} \left(\sum_{n=n(t)+1}^{n(t+\tau)}\dt^{(n+\half)}\right) \d t.
\label{est5}\end{equation}
Since
\[
\sum_{n=n(t)+1}^{n(t+\tau)}\dt^{(n+\half)} = \sum_{n\in\Z,\ t\le t^{(n)} < t+\tau} (t^{(n+1)} - t^{(n)}) \le
\tau +  \dt,
\]
we deduce from \eqref{est5} that
\begin{align*}
\dsp\int_{\R} \left(\sum_{n=n(t)+1}^{n(t+\tau)}\dt^{(n+\half)}\right)  a^{(n(t+s)+1)} \d t & \le
(\tau +  \dt)
\sum_{m\in\Z} a^{(m+1)} \int_{t^{(m)}-s}^{t^{(m+1)}-s} \d t \\ &=
(\tau +  \dt)\sum_{m\in\Z} a^{(m+1)} \dt^{(m+\half)},
\end{align*}
which is exactly \eqref{esttt2}.

\end{proof}

\textbf{Acknowledgements}: The authors would like to thank Cl\'ement Canc\`es for fruitful discussions on discrete compensated compactness theorems.

\bibliographystyle{abbrv}
\bibliography{unifconv-gs-paradeg}

\begin{thebibliography}{10}

\bibitem{aav-96-dis}
I.~Aavatsmark, T.~Barkve, O.~Boe, and T.~Mannseth.
\newblock Discretization on non-orthogonal, quadrilateral grids for
  inhomogeneous, anisotropic media.
\newblock {\em J. Comput. Phys.}, 127(1):2--14, 1996.

\bibitem{akr2009opt}
G.~Akrivis, C.~Makridakis, and R.~H. Nochetto.
\newblock Optimal order a posteriori error estimates for a class of
  {R}unge-{K}utta and {G}alerkin methods.
\newblock {\em Numer. Math.}, 114(1):133--160, 2009.

\bibitem{akr2011gal}
G.~Akrivis, C.~Makridakis, and R.~H. Nochetto.
\newblock Galerkin and {R}unge-{K}utta methods: unified formulation, a
  posteriori error estimates and nodal superconvergence.
\newblock {\em Numer. Math.}, 118(3):429--456, 2011.

\bibitem{amann}
H.~Amann.
\newblock Compact embeddings of vector-valued {S}obolev and {B}esov spaces.
\newblock {\em Glas. Mat. Ser. III}, 35(55)(1):161--177, 2000.
\newblock Dedicated to the memory of Branko Najman.

\bibitem{and-07-dis}
B.~Andreianov, F.~Boyer, and F.~Hubert.
\newblock Discrete duality finite volume schemes for {L}eray-{L}ions-type
  elliptic problems on general 2{D} meshes.
\newblock {\em Numer. Methods Partial Differential Equations}, 23(1):145--195,
  2007.

\bibitem{ACM}
B.~Andreianov, C.~Canc\`es, and A.~Moussa.
\newblock A nonlinear time compactness result and applications to
  discretization of degenerate parabolic-elliptic {PDE}s.
\newblock HAL: hal-01142499. Submitted, 2015.

\bibitem{ABM06}
H.~Attouch, G.~Buttazzo, and G.~Michaille.
\newblock {\em Variational analysis in {S}obolev and {BV} spaces}, volume~6 of
  {\em MPS/SIAM Series on Optimization}.
\newblock Society for Industrial and Applied Mathematics (SIAM), Philadelphia,
  PA; Mathematical Programming Society (MPS), Philadelphia, PA, 2006.

\bibitem{pelletier}
M.~Bertsch, P.~De~Mottoni, and L.~Peletier.
\newblock The {S}tefan problem with heating: appearance and disappearance of a
  mushy region.
\newblock {\em Trans. Amer. Math. Soc}, 293:677--691, 1986.

\bibitem{brezzi-fortin}
F.~Brezzi and M.~Fortin.
\newblock {\em Mixed and hybrid finite element methods}, volume~15 of {\em
  Springer Series in Computational Mathematics}.
\newblock Springer-Verlag, New York, 1991.

\bibitem{bre-05-fam}
F.~Brezzi, K.~Lipnikov, and V.~Simoncini.
\newblock A family of mimetic finite difference methods on polygonal and
  polyhedral meshes.
\newblock {\em Math. Models Methods Appl. Sci.}, 15(10):1533--1551, 2005.

\bibitem{jungel2}
X.~Chen, A.~J{\"u}ngel, and J.-G. Liu.
\newblock A note on {A}ubin-{L}ions-{D}ubinski\u\i\ lemmas.
\newblock {\em Acta Appl. Math.}, 133:33--43, 2014.

\bibitem{cia-91-fin}
P.~Ciarlet.
\newblock The finite element method.
\newblock In P.~G. Ciarlet and J.-L. Lions, editors, {\em Part I}, Handbook of
  Numerical Analysis, III. North-Holland, Amsterdam, 1991.

\bibitem{cou-10-dis}
Y.~Coudi{\`e}re and F.~Hubert.
\newblock A 3d discrete duality finite volume method for nonlinear elliptic
  equations.
\newblock {\em SIAM Journal on Scientific Computing}, 33(4):1739--1764, 2011.

\bibitem{crouzeix-raviart-73}
M.~Crouzeix and P.-A. Raviart.
\newblock Conforming and nonconforming finite element methods for solving the
  stationary {S}tokes equations. {I}.
\newblock {\em Rev. Fran\c caise Automat. Informat. Recherche Op\'erationnelle
  S\'er. Rouge}, 7(R-3):33--75, 1973.

\bibitem{deimling}
K.~Deimling.
\newblock {\em Nonlinear functional analysis}.
\newblock Springer-Verlag, Berlin, 1985.

\bibitem{dia94}
J.~Diaz and F.~de~Thelin.
\newblock On a nonlinear parabolic problem arising in some models related to
  turbulent flows.
\newblock {\em SIAM J. Math. Anal.}, 25(4):1085--1111, 1994.

\bibitem{jungel1}
M.~Dreher and A.~J{\"u}ngel.
\newblock Compact families of piecewise constant functions in {$L^p(0,T;B)$}.
\newblock {\em Nonlinear Anal.}, 75(6):3072--3077, 2012.

\bibitem{poly}
J.~Droniou.
\newblock Int\'egration et espaces de sobolev \`a valeurs vectorielles.
\newblock Polyco\-pi\'es de l'Ecole Doctorale de Math\'ematiques-Informatique
  de Marseille, available at \texttt{http://www-gm3.univ-mrs.fr/polys}, 2001.

\bibitem{dro-06-ll}
J.~Droniou.
\newblock Finite volume schemes for fully non-linear elliptic equations in
  divergence form.
\newblock {\em ESAIM: Mathematical Modelling and Numerical Analysis},
  40(6):1069, 2006.

\bibitem{dro-06-mix}
J.~Droniou and R.~Eymard.
\newblock A mixed finite volume scheme for anisotropic diffusion problems on
  any grid.
\newblock {\em Numer. Math.}, 105(1):35--71, 2006.

\bibitem{koala}
J.~Droniou, R.~Eymard, T.~Gallou\"et, C.~Guichard, and R.~Herbin.
\newblock Gradient schemes for elliptic and parabolic problems.
\newblock 2015.
\newblock In preparation.

\bibitem{dro-10-uni}
J.~Droniou, R.~Eymard, T.~Gallou{\"e}t, and R.~Herbin.
\newblock A unified approach to mimetic finite difference, hybrid finite volume
  and mixed finite volume methods.
\newblock {\em Math. Models Methods Appl. Sci.}, 20(2):265--295, 2010.

\bibitem{dro-12-gra}
J.~Droniou, R.~Eymard, T.~Gallou{\"e}t, and R.~Herbin.
\newblock Gradient schemes: a generic framework for the discretisation of
  linear, nonlinear and nonlocal elliptic and parabolic equations.
\newblock {\em Math. Models Methods Appl. Sci. (M3AS)}, 23(13):2395--2432,
  2013.

\bibitem{DE-fvca7}
J.~Droniou, R.~Eymard, and C.~Guichard.
\newblock Uniform-in-time convergence of numerical schemes for {R}ichards' and
  {S}tefan's models.
\newblock In {\em Finite Volumes for Complex Applications VII}. Springer, 2014.

\bibitem{edw-98-mpfa}
M.~G. Edwards and C.~F. Rogers.
\newblock Finite volume discretization with imposed flux continuity for the
  general tensor pressure equation.
\newblock {\em Comput. Geosci.}, 2(4):259--290, 1998.

\bibitem{ekeland-temam}
I.~Ekeland and R.~T{\'e}mam.
\newblock {\em Convex analysis and variational problems}, volume~28 of {\em
  Classics in Applied Mathematics}.
\newblock Society for Industrial and Applied Mathematics (SIAM), Philadelphia,
  PA, english edition, 1999.
\newblock Translated from the French.

\bibitem{ern2004theory}
A.~Ern and J.-L. Guermond.
\newblock {\em Theory and practice of finite elements}, volume 159.
\newblock Springer, Berlin, 2004.

\bibitem{eym-12-stef}
R.~Eymard, P.~Feron, T.~Gallou\"et, R.~Herbin, and C.~Guichard.
\newblock Gradient schemes for the {S}tefan problem.
\newblock {\em International Journal On Finite Volumes}, 10s, 2013.

\bibitem{sushi}
R.~Eymard, T.~Gallou{\"e}t, and R.~Herbin.
\newblock Discretization of heterogeneous and anisotropic diffusion problems on
  general nonconforming meshes {SUSHI}: a scheme using stabilization and hybrid
  interfaces.
\newblock {\em IMA J. Numer. Anal.}, 30(4):1009--1043, 2010.

\bibitem{EGHNS}
R.~Eymard, T.~Gallou{\"e}t, D.~Hilhorst, and Y.~Na{\"{\i}}t~Slimane.
\newblock Finite volumes and nonlinear diffusion equations.
\newblock {\em RAIRO Mod\'el. Math. Anal. Num\'er.}, 32(6):747--761, 1998.

\bibitem{eym-12-sma}
R.~Eymard, C.~Guichard, and R.~Herbin.
\newblock Small-stencil 3d schemes for diffusive flows in porous media.
\newblock {\em M2AN}, 46:265--290, 2012.

\bibitem{zamm2013}
R.~Eymard, C.~Guichard, R.~Herbin, and R.~Masson.
\newblock Gradient schemes for two-phase flow in heterogeneous porous media and
  {R}ichards equation.
\newblock {\em ZAMM Z. Angew. Math. Mech.}, 94(7-8):560--585, 2014.

\bibitem{MR1750075}
R.~Eymard, M.~Gutnic, and D.~Hilhorst.
\newblock The finite volume method for {R}ichards equation.
\newblock {\em Comput. Geosci.}, 3(3-4):259--294, 1999.

\bibitem{eym-11-gra}
R.~Eymard and R.~Herbin.
\newblock Gradient scheme approximations for diffusion problems.
\newblock {\em Finite Volumes for Complex Applications VI Problems \&
  Perspectives}, pages 439--447, 2011.

\bibitem{eym03}
R.~Eymard, R.~Herbin, and A.~Michel.
\newblock Mathematical study of a petroleum-engineering scheme.
\newblock {\em M2AN Math. Model. Numer. Anal.}, 37(6):937--972, 2003.

\bibitem{gal-12-com}
T.~Gallou{\"e}t and J.-C. Latch{\'e}.
\newblock Compactness of discrete approximate solutions to parabolic
  {PDE}s---application to a turbulence model.
\newblock {\em Commun. Pure Appl. Anal.}, 11(6):2371--2391, 2012.

\bibitem{glo03}
R.~Glowinski and J.~Rappaz.
\newblock Approximation of a nonlinear elliptic problem arising in a
  non-newtonian fluid flow model in glaciology.
\newblock {\em M2AN Math. Model. Numer. Anal.}, 37(1):175--186, 2003.

\bibitem{gon2002bac}
C.~Gonz{\'a}lez, A.~Ostermann, C.~Palencia, and M.~Thalhammer.
\newblock Backward {E}uler discretization of fully nonlinear parabolic
  problems.
\newblock {\em Math. Comp.}, 71(237):125--145, 2002.

\bibitem{gwi2014ful}
J.~Gwinner and M.~Thalhammer.
\newblock Full discretisations for nonlinear evolutionary inequalities based on
  stiffly accurate {R}unge-{K}utta and {$hp$}-finite element methods.
\newblock {\em Found. Comput. Math.}, 14(5):913--949, 2014.

\bibitem{her-03-app}
F.~Hermeline.
\newblock Approximation of diffusion operators with discontinuous tensor
  coefficients on distorted meshes.
\newblock {\em Computer methods in applied mechanics and engineering},
  192(16):1939--1959, 2003.

\bibitem{KAZ98}
A.~V. Kazhikhov.
\newblock {\em Recent developments in the global theory of two-dimensional
  compressible {N}avier-{S}tokes equations}, volume~25 of {\em Seminar on
  Mathematical Sciences}.
\newblock Keio University, Department of Mathematics, Yokohama, 1998.

\bibitem{lub1993runge}
C.~Lubich and A.~Ostermann.
\newblock Runge-{K}utta methods for parabolic equations and convolution
  quadrature.
\newblock {\em Math. Comp.}, 60(201):105--131, 1993.

\bibitem{lub1995lin}
C.~Lubich and A.~Ostermann.
\newblock Linearly implicit time discretization of non-linear parabolic
  equations.
\newblock {\em IMA J. Numer. Anal.}, 15(4):555--583, 1995.

\bibitem{lub1995run}
C.~Lubich and A.~Ostermann.
\newblock Runge-{K}utta approximation of quasi-linear parabolic equations.
\newblock {\em Math. Comp.}, 64(210):601--627, 1995.

\bibitem{lub1996run}
C.~Lubich and A.~Ostermann.
\newblock Runge-{K}utta time discretization of reaction-diffusion and
  {N}avier-{S}tokes equations: nonsmooth-data error estimates and applications
  to long-time behaviour.
\newblock {\em Appl. Numer. Math.}, 22(1-3):279--292, 1996.
\newblock Special issue celebrating the centenary of Runge-Kutta methods.

\bibitem{MR1916296}
E.~Maitre.
\newblock Numerical analysis of nonlinear elliptic-parabolic equations.
\newblock {\em M2AN Math. Model. Numer. Anal.}, 36(1):143--153, 2002.

\bibitem{min-63-mon}
G.~Minty.
\newblock On a “monotonicity” method for the solution of non- linear
  equations in {B}anach spaces.
\newblock {\em Proceedings of the National Academy of Sciences of the United
  States of America}, 50(6):1038, 1963.

\bibitem{MR954786}
R.~H. Nochetto and C.~Verdi.
\newblock Approximation of degenerate parabolic problems using numerical
  integration.
\newblock {\em SIAM J. Numer. Anal.}, 25(4):784--814, 1988.

\bibitem{ost2002con}
A.~Ostermann and M.~Thalhammer.
\newblock Convergence of {R}unge-{K}utta methods for nonlinear parabolic
  equations.
\newblock {\em Appl. Numer. Math.}, 42(1-3):367--380, 2002.
\newblock Ninth Seminar on Numerical Solution of Differential and
  Differential-Algebraic Equations (Halle, 2000).

\bibitem{ost2004sta}
A.~Ostermann, M.~Thalhammer, and G.~Kirlinger.
\newblock Stability of linear multistep methods and applications to nonlinear
  parabolic problems.
\newblock {\em Appl. Numer. Math.}, 48(3-4):389--407, 2004.
\newblock Workshop on Innovative Time Integrators for {PDE}s.

\bibitem{MR2230161}
I.~S. Pop.
\newblock Numerical schemes for degenerate parabolic problems.
\newblock In {\em Progress in industrial mathematics at {ECMI} 2004}, volume~8
  of {\em Math. Ind.}, pages 513--517. Springer, Berlin, 2006.

\bibitem{rul1996opt}
J.~Rulla and N.~J. Walkington.
\newblock Optimal rates of convergence for degenerate parabolic problems in two
  dimensions.
\newblock {\em SIAM J. Numer. Anal.}, 33(1):56--67, 1996.

\end{thebibliography}
\end{document}